\pdfoutput=1

\newif\ifpersonal
\personaltrue

\documentclass[11pt,a4paper,dvipsnames,x11names]{amsart}

\DeclareMathAlphabet{\mathpzc}{OT1}{pzc}{m}{it}

\usepackage[utf8]{inputenc}
\usepackage[T1]{fontenc}
\usepackage[english]{babel}

\usepackage[a4paper,margin=1in]{geometry}

\usepackage{microtype,lmodern}
\usepackage{mathpazo}
\usepackage{euler}

\usepackage{amsmath,amssymb,amsthm,mathtools,mathrsfs,bm,eucal,tensor}
\usepackage{stmaryrd}
\usepackage{dsfont}

\usepackage[all,cmtip]{xy}
\usepackage{tikz}
\usetikzlibrary{arrows.meta,calc,positioning,fit,backgrounds,shapes.geometric}
\usepackage{tikz-cd}

\tikzset{
  >=Stealth,
  box/.style={draw, rounded corners, inner sep=6pt},
  v/.style={draw, circle, inner sep=1.4pt},
  arr/.style={->, thick},
  darr/.style={->, thick, dashed}
}

\usepackage{enumerate,comment,braket,xspace,csquotes}

\usepackage{xcolor}
\usepackage{hyperref}
\hypersetup{
  colorlinks=true,
  linkcolor=blue!60!black,
  citecolor=blue!60!black,
  urlcolor=blue!60!black
}

\usepackage[capitalize,nameinlink]{cleveref}

\numberwithin{equation}{section}

\theoremstyle{plain}
\newtheorem{theorem}{Theorem}[section]
\newtheorem{proposition}[theorem]{Proposition}
\newtheorem{lemma}[theorem]{Lemma}
\newtheorem{corollary}[theorem]{Corollary}

\newtheorem{assumption}[theorem]{Assumption}

\theoremstyle{definition}
\newtheorem{definition}[theorem]{Definition}
\newtheorem{remark}[theorem]{Remark}
\newtheorem{problem}[theorem]{Problem}

\newcommand*{\defeq}{\mathrel{\vcenter{\baselineskip0.5ex \lineskiplimit0pt
                     \hbox{\scriptsize.}\hbox{\scriptsize.}}}%
                     =}

\newcommand{\OO}{\mathcal{O}}
\newcommand{\JJ}{\mathcal{J}}
\newcommand{\FF}{\mathcal{F}}
\newcommand{\Xred}{X_{\mathrm{red}}}
\newcommand{\CC}{\mathbb C}

\newcommand{\ZZ}{\mathbb Z}

\newcommand{\id}{\mathrm{id}}

\newcommand{\Hom}{\mathrm{Hom}}
\newcommand{\Aut}{\mathrm{Aut}}
\newcommand{\MC}{\operatorname{MC}}
\newcommand{\Der}{\operatorname{Der}}
\newcommand{\At}{\operatorname{At}}
\newcommand{\at}{\operatorname{at}}
\newcommand{\Ext}{\operatorname{Ext}}
\newcommand{\Sym}{\operatorname{Sym}}
\newcommand{\Def}{\operatorname{Def}}
\newcommand{\rk}{\operatorname{rk}}

\newcommand{\Res}{\operatorname{Res}}
\newcommand{\fil}{\operatorname{fil}}

\newcommand{\Tot}{\operatorname{Tot}}
\newcommand{\spl}{\mathrm{spl}}
\newcommand{\cl}{\mathrm{cl}}
\newcommand{\der}{\mathrm{der}}
\newcommand{\Tw}{\mathfrak{Tw}}

\newcommand{\proj}{\mathrm{proj}}
\newcommand{\res}{\mathrm{res}}
\newcommand{\Spc}{\mathbf{Spc}}

\providecommand{\Spec}{\operatorname{Spec}}

\newcommand{\Gr}{\mathrm{Gr}}
\newcommand{\gr}{\mathrm{gr}}

\title{Formal moduli and the splitting theory of supermanifolds}
\author{Mauricio Corr\^ea}
\address{Dipartimento di Matematica, Universit\`a degli Studi di Bari Aldo Moro, Bari, Italy}
\email{mauricio.barros@uniba.it}
\author{Simone Noja}
\address{Dipartimento di Matematica, Universit\`a degli Studi di Bari Aldo Moro, Bari, Italy}
\email{simone.noja@uniba.it}
\date{}



\begin{document}

\begin{abstract}

We develop a formal moduli theory for the splitting problem of complex supermanifolds.
Starting from Green's obstruction tower, we construct a finite-step filtered dg Lie algebra
which controls splittings by filtered Maurer--Cartan theory.  We prove that the classical
successive obstruction classes are recovered as the leading terms of adapted
Maurer--Cartan representatives, and we transfer the theory to a minimal filtered
\(L_\infty\)-model whose higher brackets give the intrinsic Kuranishi relations among
the obstruction coordinates.  We also prove that, in a precise filtered sense, the affine
Atiyah class contains the entire Green obstruction tower: the Donagi--Witten component
gives the primary obstruction, while the higher obstructions arise as successive projected
defects and residual classes of the same Atiyah cocycle.  We then pass to families, by
constructing the fixed-retract formal moduli problem with prescribed split model and by
identifying the formal neighbourhood of the split section with the fibrewise deformation
theory of the split model; under standard finiteness, base-change, and descent hypotheses
this yields relative tangent-obstruction and Kuranishi presentations.  Finally, we work
out explicit examples of non-split supergeometries, including cases with residual higher
obstructions and a first nonlinear Kuranishi relation.  These examples illustrate the
interaction between Green obstructions, Atiyah classes, and higher \(L_\infty\)-operations.
\end{abstract}

\maketitle
 
\section{Introduction}

A complex supermanifold is locally obtained by adjoining odd variables to an ordinary
complex manifold.  Thus, locally, its structure sheaf is an exterior algebra over the
structure sheaf of its reduced space.  The main question is whether this local
description can be made global.  In other words, one asks whether the odd directions of
the supermanifold can be globally separated from the even geometry.  This is the
\emph{splitting problem}.

\smallskip

More explicitly, if \(X = (|X|, \mathcal{O}_X)\) is a complex supermanifold\footnote{Here described as a locally ringed space, where $|X|$ is a topological space and \( \mathcal{O}_X\) is a sheaf of superalgebras on \(|X|\).} with reduced space
\(
\Xred=(|X|,\mathcal O_X/\mathcal J_X)
\) 
and odd conormal bundle \(
\FF_X\defeq \mathcal J_X/\mathcal J_X^2
\), then \(X\) has a canonical associated split model
\(
\widehat X
=
\bigl(\Xred,\bigwedge\nolimits^\bullet \FF_X\bigr).
\)
The splitting problem asks whether \(X\) is globally isomorphic to \(\widehat X\), compatibly
with the reduced space and with the odd conormal bundle.  A positive answer means that
the supermanifold is completely reconstructed from classical data: an ordinary complex manifold -- the reduced space $\Xred$ -- together with a holomorphic vector bundle -- the odd conormal bundle $\FF_X$.  A negative answer means that the
supermanifold carries genuinely supergeometric information which is not contained in
these two pieces of data.

\smallskip

This distinction is one of the fundamental features of complex supergeometry.  In the
smooth category, every supermanifold is split by Batchelor's theorem
\cite{BatchelorStructure}.  Smooth supermanifolds therefore have no intrinsic splitting
obstructions: up to a non-canonical isomorphism, they are exterior algebras over ordinary
manifolds.  In the holomorphic and algebraic categories the situation is different.
There exist complex supermanifolds which are not split, or
not projected.  Such objects cannot be reconstructed in an elementary way from their
reduced spaces.  In the terminology of Donagi and Witten, the relevant supergeometric
space then ``has a life of its own'' \cite{DonagiWittenNotProjected}.  For this reason,
the splitting problem is not a peripheral question: it is a central problem in the
structure theory and classification of complex supermanifolds.

The classical approach to this problem goes back to Green \cite{Green1982} and Palamodov \cite{PalamodovInvariants}, and was further
developed by Manin, Onishchik and others; see
\cite{ManinGaugeField,OnishchikClassification,VoronovManinPenkovElements}.
It expresses the splitting problem in terms of a finite tower of non-abelian
cohomological lifting problems.  Starting from the canonical first-order identification
of \(X\) with its split model, one tries to lift a splitting order by order along the
\(\mathcal J_X\)-adic filtration.  At each stage a cohomology class with values in a
certain Green obstruction sheaf measures the failure of the lifting.  This produces a
finite sequence of obstruction classes.  If all of them vanish, a splitting is obtained.

\smallskip

The splitting problem also has a concrete role in theoretical and mathematical physics.  Super Riemann
surfaces and their moduli spaces are the geometric objects underlying perturbative
superstring theory \cite{WittenSRSModuli}.  In practice, many constructions in low genus were based, explicitly
or implicitly, on the possibility of reducing integrals over supermoduli space to
ordinary moduli spaces of curves with spin structure.  Such a reduction is naturally
available when the relevant supermoduli space is projected.  The work of Donagi and
Witten showed that this expectation fails in general: the moduli space of super Riemann
surfaces is not projected in genus \(g\ge 5\), and the obstruction is already visible in
the geometry of the super Atiyah class
\cite{DonagiWittenNotProjected,DonagiWittenSuperAtiyah}.  This clarified the geometric
origin of difficulties that had appeared in the perturbative theory; see also Witten's
analysis of super Riemann surfaces and supermoduli
\cite{WittenSRSModuli}.  Later work on Ramond punctures and superstring measures
further confirms that projectedness and splitting are not technical issues but part of
the geometry that the theory has to retain
\cite{DonagiOttBadLocus,DonagiOttMeasureRamond,DonagiOttRamond,
OttVoronovGenusZeroRamond}.

\smallskip

The purpose of this paper is to give a complete deformation-theoretic formulation of the
splitting problem for complex supermanifolds.  By this we mean a formulation which does
not only record the classical sheaf-theoretic obstruction classes, but also explains their
relations, their homotopy-theoretic meaning, their Atiyah-theoretic origin, and their
global moduli interpretation.  The classical Green tower remains the starting point, but
we show that it is the visible cohomological shadow of a richer and more structured
object.

The main point is that the same splitting problem admits several equivalent, but
complementary, descriptions.  First, it is a finite sheaf-theoretic obstruction problem in
the sense of Green.  Second, it is a filtered Maurer--Cartan problem controlled by a
finite-step filtered differential graded Lie algebra.  Third, after homotopy transfer, it is
encoded by a minimal filtered \(L_\infty\)-algebra on cohomology, whose higher brackets
give the Kuranishi relations among the obstruction coordinates.  Fourth, the same
obstruction tower is contained in the affine Atiyah class of the supermanifold: the
Donagi--Witten component gives the primary obstruction, and the higher Green
obstructions are obtained as successive projected defects and residual classes of the
same filtered Atiyah cocycle.  Finally, the fixed-retract theory globalizes to families: the formal neighbourhood of
the split model is controlled fibrewise by the filtered deformation complex of the
relative split model, and it admits relative and Kuranishi presentations under the
corresponding base-change and finite-model hypotheses.

This gives, in particular, two extensions of the existing picture.  On the one hand, it
brings the splitting theory of complex supermanifolds into the framework of higher
algebra and formal moduli problems in characteristic zero, as developed by Lurie,
Pridham, and subsequent work on global formal moduli
\cite{LurieDAGX,PridhamUDDT,HennionTangentLie,CalaqueGrivauxFMP}.  On the
other hand, it completes the Atiyah-class interpretation initiated by Donagi and Witten:
the affine Atiyah class contains not only the primary splitting obstruction, but the whole
finite Green obstruction tower.

\smallskip

The first step is to translate Green's obstruction tower into infinitesimal
deformation theory.  On the split model
\(
\widehat X=\bigl(\Xred,\wedge^\bullet \FF_X\bigr),
\)
the exterior degree in \(\wedge^\bullet \FF_X\) gives a natural grading.  The tangent
sheaf of \(\widehat X\) is therefore filtered by the order to which a derivation
raises the odd degree.  Equivalently, one filters the even derivations of
\(\mathcal O_{\widehat X}\) by those which are invisible modulo higher and higher
powers of the odd ideal.  This is the Green filtration. 

Thus we attach to \(X\) the finite-step filtered differential graded Lie algebra\footnote{Analogously, the filtered dg Lie algebra with finite nilpotent Green filtration.}
\[
L_X
=
A^{0,\bullet}\!
\bigl(\Xred,\mathcal T_{\widehat X,\bar 0}^{[\ge 2]}\bigr),
\]
where \(\mathcal T_{\widehat X,\bar 0}^{[\ge 2]}\) denotes the sheaf of even
derivations of the split model which raise the exterior degree by at least two.  The
differential is the Dolbeault operator of the split model and the bracket is the
commutator of derivations.  The Green filtration on
\(\mathcal T_{\widehat X,\bar 0}^{[\ge 2]}\) induces a finite filtration
\(F^\bullet L_X\).  We prove that the splitting problem is equivalent to a filtered
Maurer--Cartan problem for \(L_X\).
More precisely, the holomorphic structure of
\(X\), after choosing a \(C^\infty\)-splitting, determines a Maurer--Cartan gauge class
\[
[\mu_X]\in MC(L_X)/\mathrm{gauge},
\]
and the successive attempts to split \(X\) correspond to moving this gauge class into
deeper steps of the filtration.  Adapted Maurer--Cartan representatives recover the
classical Green obstruction classes through their leading homogeneous terms.

This gives a formal moduli interpretation of the splitting problem.  The functor
\[
\Def_X(A)
=
MC_\infty(L_X\otimes \mathfrak m_A),
\]
defined on local Artin dg algebras, is the pointed formal moduli problem controlled by
\(L_X\).  Its filtration induces a finite tower of formal moduli problems
\[
\Def_X=\Def_X^{\ge 1}
\longleftarrow
\Def_X^{\ge 2}
\longleftarrow
\cdots
\longleftarrow
\Def_X^{\ge N}
\longleftarrow
\Def_X^{\ge N+1}=*,
\qquad
N=\Bigl\lfloor \frac{\rk \FF_X}{2}\Bigr\rfloor .
\]
In this sense the classical finite obstruction tower is realized as the filtration tower
of a formal moduli problem.  This places the splitting problem of complex
supermanifolds within the general framework of derived deformation theory and formal
moduli problems in characteristic zero, as developed by Lurie and Pridham
\cite{LurieDAGX,PridhamUDDT}, and in the global form considered by
Hennion, Calaque--Grivaux, and Nuiten
\cite{HennionTangentLie,CalaqueGrivauxFMP,Nuiten}.  To the best of our knowledge,
this is the first systematic use of filtered dg Lie, \(L_\infty\), and formal moduli
methods in the splitting theory of complex supermanifolds.

The second step is to pass from the dg Lie controller to cohomology.  After choosing a
filtered contraction of \(L_X\) onto its cohomology
\(
H_X\defeq H^\bullet(L_X),
\)
homotopy transfer gives a minimal filtered \(L_\infty\)-algebra
\(
(H_X,\{\ell_m\}_{m\ge 2})
\)
which controls the same formal moduli problem.  A Maurer--Cartan element in this
minimal model has the form
\[
\alpha_X
=
\alpha_{X,1}+\cdots+\alpha_{X,N},
\qquad
\alpha_{X,j}\in
H^1\!\bigl(\Xred,\mathcal T_{\widehat X,\bar0}^{\langle 2j\rangle}\bigr).
\]
The higher brackets give the Kuranishi equations
\[
\kappa_j(\alpha_{X,1},\ldots,\alpha_{X,j-1})=0,
\qquad
2\le j\le N.
\]
These equations are the intrinsic compatibility relations among the successive
cohomological coordinates.  The individual higher coordinates depend on the choice of
partial splitting tower and on the chosen contraction, as is expected in Kuranishi
theory; the filtered formal moduli problem itself is independent of these choices.

The third step concerns Atiyah classes.  Donagi and Witten showed that the primary
splitting obstruction is a distinguished component of the restricted super Atiyah class
\cite{DonagiWittenSuperAtiyah}.  We extend this picture to the full obstruction tower.
More precisely, we show that the affine Atiyah class of \(X\) contains all Green
obstructions in a filtered sense.  The primary obstruction is the Donagi--Witten
component.  At higher orders, after a lower-order splitting has been fixed, the first
non-split homogeneous part of the affine Atiyah cocycle has a pure odd Hessian symbol.
Its even projection gives the next Green obstruction, and, after this obstruction
vanishes and a further splitting is chosen, its residual odd projection gives the next
obstruction.  Thus the affine Atiyah class does not decompose into a direct sum of
higher obstruction classes; rather, it contains them as successive projected defects and
residual classes of one filtered Atiyah cocycle.  This completes the Atiyah-theoretic
interpretation of the splitting tower.

The fourth part of the paper globalizes the preceding construction.  We fix a smooth
proper family
\(
\pi:Y\to S
\)
and a vector bundle \(F\) on \(Y\), regarded as the prescribed odd conormal bundle.  The
corresponding split relative model is
\(
A_{\spl}=\bigwedge\nolimits^\bullet_{\mathcal O_Y}F .
\)
The correct fixed-retract moduli problem is not the moduli problem of filtered
\(\mathcal O_Y\)-superalgebras: imposing an \(\mathcal O_Y\)-algebra structure would
amount to choosing a projection to the reduced family and would exclude precisely the
non-projected phenomena detected by Green's obstruction tower.  Instead, we consider
filtered \(\pi^{-1}\mathcal O_S\)-superalgebras whose associated graded algebra is
identified with \(A_{\spl}\).  This gives a relative moduli problem of supermanifold
structures with fixed reduced family and fixed odd conormal bundle.  We call it the
Green--Onishchik fixed-retract moduli problem, in reference to Green's non-abelian
cohomological classification and Onishchik's moduli theory of complex analytic
supermanifolds with fixed retract
\cite{Green1982,OnishchikClassification}.

The main result of the relative theory is the construction and identification of the
derived formal neighbourhood of the split section.  For every test family \(T\to S\), it
is controlled by the fibrewise Green dg Lie algebra
\(
R\Gamma\bigl(Y_T,\mathfrak t_{F_T/T}\bigr),
\)
where \(\mathfrak t_{F/S}\) is the sheaf of even relative derivations of the split model
which raise the odd Green filtration by at least two.  Thus the same filtered
Maurer--Cartan mechanism that controls the splitting of a single supermanifold also
controls the formal geometry of the fixed-retract moduli problem in families.

When the relative controller satisfies derived base change, the fibrewise description
assembles into the compact relative expression
\[
\widehat\Tw_{F/S}
\simeq
\Def_S\bigl(R\pi_*\mathfrak t_{F/S}\bigr).
\]
This is the base-change-compatible form of the Green--Onishchik formal moduli problem.
It gives a relative tangent-obstruction theory whose homogeneous pieces recover the
relative Green obstruction sheaves.  Under finite perfect model, amplitude, and descent
hypotheses, this formal theory further admits derived Kuranishi presentations, locally
and, when the descent data are available, globally.

The final section is devoted to examples.  Explicit non-split and non-projected
supermanifolds are relatively scarce in the literature, especially beyond the primary
obstruction level.  Important classes of examples have been constructed and studied by
Manin, Onishchik, Vishnyakova, Bettadapura, Cacciatori--Noja--Re, Noja, Leites--Tikhomirov, and others; see, for instance,
\cite{ManinGaugeField,
OnishchikClassification,
VishnyakovaEvenHomogeneous,
VishnyakovaCP1Supermanifolds,
VishnyakovaSplittingHomogeneous,
VishnyakovaGradedCovering,
BettadapuraHigherObstructions,
BettadapuraObstructedThickenings,
CacciatoriNojaReCalabiYau,
NojaPiProjectiveSpaces,
NojaBVSupermanifolds,
LeitesTikhomirovNonSplitSuperstrings}.
Against this background, the examples considered here have two purposes.  First, they
show that the preceding formalism is not only a reformulation of the classical theory, but
also gives effective computational tools for detecting and organizing non-split
supermanifolds.  Second, they enlarge the supply of explicit geometries for which the
Green obstruction tower, the affine Atiyah class, and the Kuranishi relations can be
compared directly.

In odd rank three over an elliptic curve we exhibit concrete \v{C}ech cocycles for which
the same weight-two tangent class contains both the primary even obstruction and, after
normalization, the residual odd obstruction.  Over a K3 surface we obtain further
nontrivial residual classes in a setting where the cohomology is richer but the formal
controller remains abelian for weight reasons.  We then discuss a relative
Green--Onishchik example in odd rank three, illustrating the formal moduli problem in
families.  Finally, in odd rank four on an abelian surface, we obtain the first nonlinear
Kuranishi relation.  This last example shows that, in higher odd rank, the Green
coordinates are not merely independent linear obstruction parameters: their Lie brackets
can impose genuine quadratic compatibility equations.
\medskip

We now describe the organization of the paper.  Section~2 fixes the basic notation for
complex supermanifolds, projections, splittings, and truncations.  Section~3 recalls the
classical Green lifting tower and the associated obstruction sheaves.  Section~4
constructs the filtered dg Lie algebra \(L_X\), identifies the splitting problem with a
filtered Maurer--Cartan problem, and defines the corresponding formal moduli problem.
Section~5 transfers this dg Lie model to a minimal filtered \(L_\infty\)-model and
identifies the Kuranishi equations.  Section~6 relates the affine Atiyah class to the
entire Green obstruction tower.  Section~7 gives a \v{C}ech--Thom--Whitney model and
makes the comparison between \v{C}ech coordinates, Atiyah symbols, and Kuranishi
compatibilities explicit.  Section~8 constructs the Green--Onishchik formal moduli
problem and studies its possible algebraizations.  Section~9 contains the examples.

We conclude the introduction with a schematic roadmap of the paper in Figure \ref{fig:introduction-roadmap}. The diagram is not
meant to record the technical constructions in detail; rather, it indicates the main
viewpoints on the splitting problem and how they are related.  The starting point is the
classical Green obstruction tower.  The paper then develops three complementary
enhancements of this tower: a formal deformation-theoretic one, an Atiyah-theoretic one,
and a global moduli-theoretic one.  The examples at the end show how these viewpoints
interact in concrete non-split geometries.

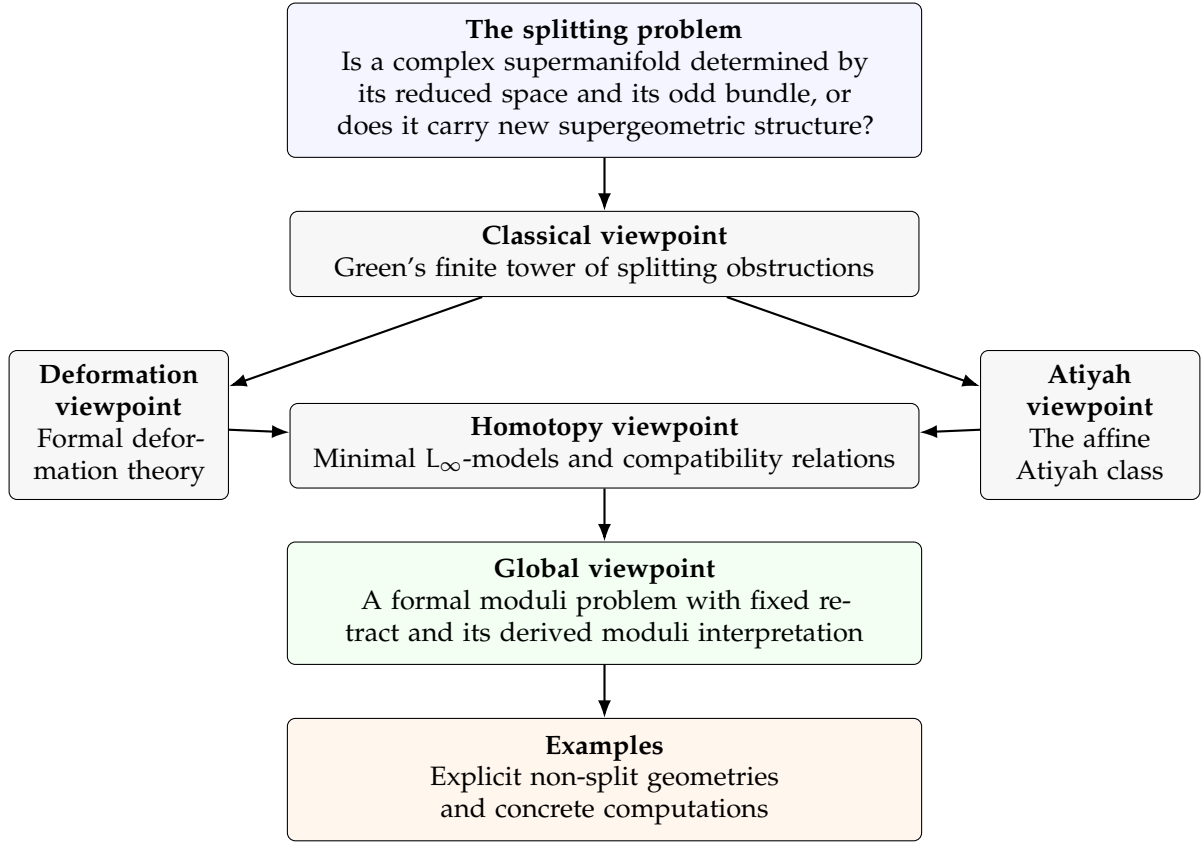
\begin{figure}[h]
\centering
\begin{tikzpicture}[
  >=Latex,
  node distance=7mm and 8mm,
  box/.style={
    draw,
    rounded corners=3pt,
    align=center,
    font=\small,
    inner sep=5pt,
    minimum height=10mm,
    text width=.50\linewidth,
    fill=gray!6
  },
  smallbox/.style={
    draw,
    rounded corners=3pt,
    align=center,
    font=\small,
    inner sep=5pt,
    minimum height=10mm,
    text width=.16\linewidth,
    fill=gray!6
  },
  bluebox/.style={
    draw,
    rounded corners=3pt,
    align=center,
    font=\small,
    inner sep=6pt,
    minimum height=11mm,
    text width=.50\linewidth,
    fill=blue!4
  },
  greenbox/.style={
    draw,
    rounded corners=3pt,
    align=center,
    font=\small,
    inner sep=6pt,
    minimum height=11mm,
    text width=.50\linewidth,
    fill=green!5
  },
  orangebox/.style={
    draw,
    rounded corners=3pt,
    align=center,
    font=\small,
    inner sep=6pt,
    minimum height=11mm,
    text width=.50\linewidth,
    fill=orange!7
  },
  arrow/.style={->, thick}
]

\node[bluebox] (split) {%
\textbf{The splitting problem}\\
Is a complex supermanifold determined by its reduced space and its odd bundle,
or does it carry new supergeometric structure?};

\node[box, below=of split] (green) {%
\textbf{Classical viewpoint}\\
Green's finite tower of splitting obstructions};

\node[smallbox, below left=of green] (formal) {%
\textbf{Deformation viewpoint}\\
Formal deformation theory};

\node[smallbox, below right=of green] (atiyah) {%
\textbf{Atiyah viewpoint}\\
The affine Atiyah class};

\node[box, below=14mm of green] (linfty) {%
\textbf{Homotopy viewpoint}\\
Minimal \(L_\infty\)-models and compatibility relations};

\node[greenbox, below=of linfty] (global) {%
\textbf{Global viewpoint}\\
A formal moduli problem with fixed retract and its derived moduli interpretation};

\node[orangebox, below=of global] (examples) {%
\textbf{Examples}\\
Explicit non-split geometries and concrete computations};

\draw[arrow] (split) -- (green);
\draw[arrow] (green) -- (formal);
\draw[arrow] (green) -- (atiyah);
\draw[arrow] (formal) -- (linfty);
\draw[arrow] (atiyah) -- (linfty);
\draw[arrow] (linfty) -- (global);
\draw[arrow] (global) -- (examples);

\end{tikzpicture}
\caption{Conceptual roadmap of the paper.  The splitting problem is approached through
classical obstruction theory, deformation theory, and Atiyah theory; these are then
assembled into a homotopy-algebraic and global moduli-theoretic picture, and tested on
explicit examples.}
\label{fig:introduction-roadmap}
\end{figure}

\section{Basics}

\subsection{Preliminaries on complex supermanifolds.}
This section recalls the basic supergeometric material needed in the rest of the paper and fixes
the notation used throughout. All of this is standard; we refer, for example, to
\cite{BruzzoHernandezPolishchukHilbertPicard,EderHuertaNojaPoincareDuality,KesslerSupergeometry,ManinGaugeField,NojaGeometryForms,VoronovManinPenkovElements}
for background from the smooth, holomorphic, and algebraic points of view. Our main purpose
here is not to develop the general theory, but to isolate the specific objects that later control the
splitting problem: the odd conormal bundle \(\FF_X\), the split model \(\widehat X\), and the finite
tower of truncations \(X^{(k)}\) and \(\widehat X^{(k)}\).

 \begin{definition}      
  A {superspace} is a locally ringed space $X \defeq (|X|,
  \mathcal{O}_{X})$ whose structure sheaf is given by a sheaf of
  $\mathbb{Z}_2$-graded supercommutative rings $\mathcal{O}_{X} \defeq
  \mathcal{O}_{X, 0} \oplus \mathcal{O}_{X, 1}$. Given two superspaces
  $X$ and $Y$, a {morphism of superspaces} $\varphi :
  X \to Y$ is a pair $\varphi \defeq (|\varphi|, \varphi^\sharp)$ where
  \begin{enumerate}
  \item $|\varphi| : |{X}| \to |{Y}|$ is a continuous map
    of topological spaces;
  \item $\varphi^\sharp : \OO_{Y} \to |\varphi|_*
    \OO_X$ is a morphism of sheaves of $\ZZ_2$-graded
    supercommutative rings, having the properties that it preserves the
    $\ZZ_2$-grading and the unique maximal ideal of each stalk, i.e. $\varphi^\sharp_x
    (\mathfrak{m}_{|\varphi| (x)}) \subseteq \mathfrak{m}_x$ for all $x \in |{X}|.$
  \end{enumerate}
\end{definition}

\noindent
Given a superspace $X$, the odd sections in $\mathcal{O}_X$ define a sheaf of ideals.

\begin{definition} 
  Let $X$ be a superspace. The sheaf of ideals $\mathcal{J}_X  \defeq \langle \mathcal{O}_{X, 1} \rangle =
  \mathcal{O}_{X,1} + (\mathcal{O}_{X, 1})^2$, generated by all the odd
  sections in $\OO_X = \OO_{X, 0 } \oplus \OO_{X,
  1}$ is called the {odd-generated ideal} of $X$.
\end{definition}

\noindent
For a general superspace, this ideal need not coincide with the full nilradical, since
there may be even nilpotent sections not generated by odd sections. However, for a
supermanifold, $\mathcal{J}_X$ is precisely the nilpotent ideal defining the reduced
space. The supermanifold case is our exclusive concern here, and we turn to the
definition now.

\begin{definition} 
  \label{supermanifold}
  A {complex supermanifold} is a superspace $X$ such that
  \begin{enumerate}
  \item $X_{\mathrm{red}} \defeq (|X|, \mathcal{O}_{X} /
    \mathcal{J}_{X})$ is a complex manifold;
  \item the quotient sheaf $\mathcal{J}_{X} / \mathcal{J}^2_X$ is a
    locally-free sheaf of $\mathcal{O}_{X} / \mathcal{J}_X$-modules and the
    structure sheaf $\mathcal{O}_{X}$ is locally isomorphic to the sheaf of
    exterior algebras $\wedge^\bullet \left( \mathcal{J}_{X} /
      \mathcal{J}^2_X \right)$ over $\mathcal{O}_{X} / \mathcal{J}_X$.
  \end{enumerate}
 We call \(\Xred\) the \emph{reduced space} of the supermanifold \(X\), and we denote
\(
\mathcal O_{\Xred}\defeq \mathcal O_X/\mathcal J_X.
\)
Further, we write
\(
\FF_X \defeq \mathcal J_X/\mathcal J_X^2
\)
for the holomorphic \emph{odd conormal bundle} of \(X\). If \(\dim \Xred=m\) and \(\rk(\FF_X)=n\),
we say that \(X\) has dimension \(m|n\). 
Finally, a {morphism of complex supermanifolds} is a morphism
  of superspaces. 
\end{definition}

\begin{remark} 
Definition~\ref{supermanifold} implies that a complex supermanifold of dimension \(m|n\) is locally isomorphic,
as a locally ringed space, to the standard complex superdomain
\[
\CC^{m|n}
=
\Bigl(\CC^m,\mathcal O_{\CC^m}\otimes_\CC \wedge^\bullet_\CC(\theta^1,\dots,\theta^n)\Bigr),
\]
where \(\mathcal O_{\CC^m}\) denotes the sheaf of holomorphic functions on \(\CC^m\), and
\(\wedge^\bullet_\CC(\theta^1,\dots,\theta^n)\) is the complex exterior algebra on odd generators
\(\theta^1,\dots,\theta^n\). The standard coordinate functions \(x^1,\dots,x^m\) on \(\CC^m\) are the
even coordinates, while \(\theta^1,\dots,\theta^n\) are the odd coordinates. We will often denote
such a coordinate system by
\(
x^a \mid \theta^\alpha,
\)
where the symbol ``\(\mid\)'' separates the even and odd variables, just as in the dimension
notation \(m|n\).
\end{remark}

\subsection{Projections and splittings}

Definition \ref{supermanifold} implies that any supermanifold comes endowed with a
short exact sequence of sheaves that encodes the relation between the supermanifold
itself and its reduced space:
\begin{eqnarray}
  \label{ses}
  \xymatrix{
  0 \ar[r] & \mathcal{J}_X \ar[r] & \mathcal{O}_{X} \ar[r]_{\iota^\sharp} & \ar@{-->}@/_1.2pc/[l]_{\pi^\sharp} \mathcal{O}_{\Xred} \ar[r] & 0.
  }
\end{eqnarray}
The morphism of sheaves $\iota^\sharp : \mathcal{O}_X \to \mathcal{O}_{\Xred}$ is canonical: it mods out all
the nilpotent sections from the structure sheaf $\mathcal{O}_{X}$. That is, for every open set $U \in |X|$ we have 
\begin{align}
\xymatrix@R=1.5pt{
\iota^\sharp_U : \OO_X (U) \ar[r] & \OO_{\Xred} (U) \\
f \ar@{|->}[r] & f \ \mbox{mod}\ \mathcal{J}_X (U).
}
\end{align}
At the level of spaces, it corresponds to the embedding $\iota : \Xred \to X$: intuitively, we can view $\Xred$ as the subspace defined by setting all nilpotent
sections to zero. 

\medskip

If it exists, a splitting of the short exact sequence \ref{ses}, that is a map $\pi^\sharp : \OO_{\Xred} \to \OO_X$ such that $\iota^\sharp \circ \pi^\sharp = \id_{{\Xred}}$ is called a \emph{projection} of $X$. On every open set $|U| \subset X$, the morphism $\pi^\sharp_U : \OO_{\Xred}(U) \to \OO_{X} (U)$, chooses a representative in $\OO_X (U)$ of every class (of holomorphic functions) mod $\mathcal{J}_X (U)$ in a way that is compatible 
with the restrictions -- hence, $\pi^\sharp$ is non-canonical. Whenever it exists, the morphism $\pi^\sharp$ identifies $\OO_{\Xred}$ with a subsheaf of rings of $\OO_X$ which is complementary to $\mathcal{J}_X$ in the sense that one has an isomorphism of sheaves of $\OO_{\Xred}$-modules
\begin{align}
\OO_X \cong \pi^\sharp (\OO_{\Xred}) \oplus \mathcal{J}_X \cong \OO_{\Xred}\oplus \mathcal{J}_X.
\end{align}   
Notice, though, that the multiplicative structure of $\OO_X$ remains non-trivial, as $\mathcal{J}_X$ is not a subsheaf of rings with unity, but a nilpotent ideal instead.  \\
At the level of locally ringed space, the morphism $\pi^\sharp$ corresponds to a projection $\pi : X \to \Xred$ such that $\pi \circ \iota = \id_{\Xred}.$ 
\begin{definition}  \label{projected}
A complex supermanifold $X$ is said to be projected if it admits a projection $\pi : X \to \Xred$. \\
Equivalently, $X$ is projected if the canonical quotient map $\iota^\sharp $ admits a section as a morphism of sheaves of $\CC$-algebras, \emph{i.e.}\
if there exists a morphism $\pi^\sharp : \OO_{\Xred} \to \OO_X$ such that $\iota^\sharp \circ \pi^\sharp = \id_{\mathcal O_{\Xred}}.$ 
\end{definition}

Given a complex supermanifold $X$ its structure sheaf admits a canonical filtration induced by $\JJ_X \subset \OO_X$.
\begin{definition}  The $\JJ_X$-adic filtration $F^\bullet \OO_X$ of $\OO_X$ is the decreasing filtration 
   \begin{eqnarray} 
    \label{filtX}
 F^\bullet \OO_X: \quad   \mathcal{O}_X = F^0 \OO_X \supset F^1 \OO_X \supset F^2 \mathcal{O}_X \ldots 
  \end{eqnarray}
defined for every $k \geq 0$ by 
$
F^k \OO_X \defeq \JJ_X^k
$
with the convention that $\JJ_X^0 = \mathcal{O}_X$.
\end{definition}
Note that since $\JJ_X$ is nilpotent, if the odd dimension of $X$ is $n$, then $\JJ^{n+1}_X = 0$, hence this filtration is finite: 
  \begin{eqnarray}
    \label{filt}
    \mathcal{O}_X \supset \mathcal{J}_X \supset \mathcal{J}_X^2 \supset  \ldots \supset \mathcal{J}_X^{n} \supset \mathcal{J}_X^{n+1} = 0.
  \end{eqnarray}
The associated graded sheaf to the $\JJ_X$-adic filtration of $\OO_X$ is the $\ZZ$-graded sheaf (of superalgebras)
\begin{equation}
\mbox{Gr}^{\bullet} \mathcal{O}_X \defeq \bigoplus_{k \geq 0} \mbox{gr}^k \OO_X, \qquad \mbox{gr}^k \OO_X\defeq F^k \OO_X / F^{k+1}\OO_X.
\end{equation}
Explicitly, 
  \begin{eqnarray}
    \label{grsheaf}
    \mbox{{Gr}}^{\bullet} \mathcal{O}_X \defeq
    \bigoplus_{k = 0}^{n} \JJ_X^k / \JJ_X^{k+1} = \mathcal{O}_{\Xred }
    \oplus {\mathcal{J}_X}/{\mathcal{J}^2_{X}} \oplus \ldots \oplus
    {\mathcal{J}_X^{n-1}}/{\mathcal{J}^n_{X}} \oplus
    {\mathcal{J}^n_X}.
  \end{eqnarray}
  The following is well-known
  \begin{lemma}
There is a canonical isomorphism of \(\mathbb Z\)-graded
\(\mathcal O_{\Xred}\)-superalgebras
\begin{align}
\operatorname{Gr}^\bullet \mathcal O_X
\;\xrightarrow{\ \sim\ }\;
\wedge^\bullet\FF_X.
\end{align}
More precisely, for every \(k\ge 0\) there is a canonical isomorphism
\[
\operatorname{gr}^k\mathcal O_X
\defeq
\mathcal J_X^k/\mathcal J_X^{k+1}
\;\xrightarrow{\ \sim\ }\;
\wedge^k \FF_X.
\]
\end{lemma}
\begin{definition} 
We denote
\(
\widehat{\mathcal O}_X \defeq \wedge^\bullet \FF_X
\)
and we call
\(
\widehat X \defeq (\Xred,\widehat{\mathcal O}_X)
\)
the split model associated with \(X\).
For every \(k\ge 0\), we also set
\[
\widehat{\mathcal O}_X^{(k)} \defeq \wedge^{\le k}\FF_X
=
\bigoplus_{j=0}^k \wedge^j \FF_X
\]
with the truncated exterior product, or equivalently
\[
\wedge^{\le k}\FF_X
=
\wedge^\bullet \FF_X
\Big/
\bigoplus_{j>k}\wedge^j\FF_X.
\]
We denote by
\(
\widehat X^{(k)} \defeq (\Xred ,\widehat{\mathcal O}_X^{(k)})
\)
the \(k\)-th split truncation.
\end{definition}

The sheaf \(\operatorname{Gr}^\bullet \OO_X \cong \wedge^\bullet \mathcal{F}_X \) carries its canonical augmentation 
\begin{equation}
\varepsilon: \wedge^\bullet \FF_X \to \wedge^0 \FF_X = \mathcal O_{\Xred},
\end{equation}
given by projection onto degree \(0\), while \(\mathcal O_X\) carries the canonical quotient map introduced above
\(
\iota^\sharp:\mathcal O_X \to \mathcal O_{X_{\mathrm{red}}}.
\)
\begin{definition} \label{splitting}
A \emph{splitting} of \(X\) is an isomorphism of sheaves of
\(\CC\)-superalgebras on \(|X|=|\Xred|\)
\[
 \sigma: \wedge^\bullet \FF_X \xrightarrow{\sim} \mathcal O_X
\]
such that the following hold:
\begin{enumerate}
\item 
$ \sigma $ is compatible with the augmentations, \emph{i.e.}\
$\iota^\sharp \circ \sigma = \varepsilon.$ Equivalently, the following diagram commutes:
\begin{equation}
\begin{tikzcd}
\wedge^\bullet \FF_X \arrow[r, "\sigma", "\sim"'] \arrow[d, "\varepsilon"'] &
\mathcal O_X \arrow[d, "\iota^\sharp"] \\
\mathcal O_{X_{\mathrm{red}}} \arrow[r, equals] &
\mathcal O_{X_{\mathrm{red}}}.
\end{tikzcd}
\end{equation}
\item the induced morphism
\begin{align}
\operatorname{gr}^1(\sigma):
\FF_X 
\cong
\frac{\ker(\varepsilon)}{\ker(\varepsilon)^2}
\longrightarrow
\frac{\mathcal J_X}{\mathcal J_X^2}
=
\FF_X
\end{align}
is the identity.
\end{enumerate}
\end{definition}
\begin{definition}
A complex supermanifold \(X\) is said to be \emph{split} if it admits a
splitting.
\end{definition}
\begin{remark}  \label{rmk:norm}
Condition~(2) in Def.\ \ref{splitting} is only a normalization and does not change the notion of split supermanifold.
Indeed, let
\(
\sigma:\wedge^\bullet \FF_X \xrightarrow{\sim} \mathcal O_X
\)
be any isomorphism of augmented sheaves of \(\CC\)-superalgebras compatible with the
augmentations. Then \(\sigma\) induces an automorphism
\(
\phi \defeq \gr_1(\sigma):\FF_X \xrightarrow{\sim} \FF_X.
\)
Precomposing \(\sigma\) with the exterior algebra automorphism
\[
\wedge^\bullet(\phi^{-1}):\wedge^\bullet \FF_X \xrightarrow{\sim} \wedge^\bullet \FF_X
\]
produces an augmentation-compatible splitting whose induced map on
\(\FF_X=\mathcal{J}_X/\mathcal{J}_X^2\) is the identity. Thus every splitting admits a canonical normalization after fixing the original
augmentation-compatible isomorphism \cite{Green1982,ManinGaugeField,OnishchikClassification}.
\end{remark}


The notion of split supermanifold should be compared with the weaker notion of projected supermanifold, Def.\ \ref{projected}. Namely, the following is well-known.
\begin{lemma} Every split supermanifold is projected.
\end{lemma}
\begin{proof}
If \(X\) is split, there exists an isomorphism
\(
\sigma:\wedge^\bullet \FF_X 
\xrightarrow{\sim}\mathcal O_X
\)
compatible with the augmentations. Therefore, the composition
\(\mathcal O_{X_{\mathrm{red}}}\hookrightarrow
\wedge^\bullet \FF_X
\xrightarrow{\ \sigma\ }
\mathcal O_X
\)
is a section of the quotient map \(\iota^\sharp:\mathcal O_X\to \mathcal O_{X_{\mathrm{red}}}\), so \(X\) is projected.
\end{proof}

\subsection{Truncations and $k$--splittings}

A natural finite-order analogue of a splitting is obtained by requiring the split model to agree with the structure sheaf of \(X\) only up to 
the \((k+1)\)-power of the nilpotent ideal, leading to the notion of a \(k\)-splitting. For this, we denote for every $k \geq 1$
\begin{equation}
\varepsilon_k: \wedge^{\leq k} \FF_X \to \OO_{\Xred},  \qquad \iota^\sharp_k: \mathcal O_X/\mathcal J_X^{k+1}\to \mathcal O_{\Xred},
\end{equation}
the truncated canonical augmentation and the natural quotient map, respectively. 

\begin{definition}[$k$-splitting] \label{ksplitting} A \emph{\(k\)-splitting} of \(X\) is an isomorphism of sheaves of
\(\CC\)-superalgebras on \(|X|=|\Xred|\)
\[
\sigma^{(k)}: \wedge^{\leq k} \FF_X \xrightarrow{\sim}\mathcal O_X/\mathcal J_X^{k+1}
\]
such that:

\begin{enumerate}
\item \(\sigma^{(k)}\) is compatible with the augmentations, namely
\(
\iota^\sharp_k\circ \sigma^{(k)}=\varepsilon_k;
\)

\item the induced morphism
\[
\operatorname{gr}^1(\sigma^{(k)}):
\FF_X
\cong
\frac{\ker(\varepsilon_k)}{\ker(\varepsilon_k)^2}
\longrightarrow
\frac{\mathcal J_X/\mathcal J_X^{k+1}}{(\mathcal J_X/\mathcal J_X^{k+1})^2}
\cong
\frac{\mathcal J_X}{\mathcal J_X^2}
=
\FF_X
\]
is the identity.
\end{enumerate}
\end{definition}
\begin{definition}[$k$-thickening] \label{def:thick}For every \(k\geq 0\) we define the \emph{$k$-thickening} of $\Xred$ as
\begin{equation}
X^{(k)}
\defeq
\bigl(X_{\mathrm{red}},\mathcal O_X/\mathcal J_X^{k+1}\bigr).
\end{equation}
\end{definition}
\begin{remark}
As for full splittings, condition~(2) in Definition~\ref{ksplitting} is only a normalization, and similar considerations as in Remark \ref{rmk:norm} apply to $k$-splittings.
\end{remark}

\begin{remark}  After Def.\ \ref{def:thick}, the splitting problem naturally becomes a sequence of finite-order comparison problems:
for each \(k\ge0\), one asks whether the \(k\)-th thickening
\(
X^{(k)}=(\Xred,\mathcal O_X/\mathcal J_X^{k+1})
\)
is isomorphic to the corresponding split truncation
\(
\widehat X^{(k)}=(\Xred,\wedge^{\le k}\FF_X).
\)
A \(k\)-splitting is precisely such an isomorphism, normalized at first order. In
particular, a splitting in the sense of Def.\ \ref{splitting} induces, by reduction modulo
\(\mathcal J^{k+1}\), a \(k\)-splitting for every \(k\geq 1\). 
\end{remark}


Now fix \(k\geq 1\), and define the sheaves
\begin{equation}
\mathcal A_X^{(k)}
\defeq
\bigwedge\nolimits^{\leq k}\mathcal F_X,
\qquad
\mathcal B_X^{(k)}
\defeq
\mathcal O_X/\mathcal J_X^{k+1}.
\end{equation}
for the sake of notation. Endow \(\mathcal A_X^{(k)}\) with the decreasing filtration
\begin{equation}
F^p\mathcal A_X^{(k)}
\defeq
\bigoplus_{j\geq p}\wedge^j\mathcal F_X,
\qquad 0\leq p\leq k+1,
\end{equation}
and \(\mathcal B_X^{(k)}\) with the decreasing filtration
\begin{equation}
F^p\mathcal B_X^{(k)}
\defeq
\mathcal J_X^p/\mathcal J_X^{k+1},
\qquad 0\leq p\leq k+1.
\end{equation}

\begin{lemma} If
\(
\sigma^{(k)}:\mathcal A_X^{(k)}\xrightarrow{\sim}\mathcal B_X^{(k)}
\)
is a \(k\)-splitting of \(X\), then \(\sigma^{(k)}\) is a filtered isomorphism, \emph{i.e.}\
\begin{equation}
\sigma^{(k)}\bigl(F^p\mathcal A_X^{(k)}\bigr)=F^p\mathcal B_X^{(k)}
\qquad \text{for all }p,
\end{equation}
and the induced isomorphism on associated graded algebras
\begin{equation}
\gr\bigl(\sigma^{(k)}\bigr):
\gr_F\mathcal A_X^{(k)}
\longrightarrow
\gr_F\mathcal B_X^{(k)}
\end{equation}
coincides with the canonical one
\[
\bigwedge\nolimits^{\leq k}\mathcal F_X
\;\xrightarrow{\sim}\;
\bigoplus_{p=0}^k \mathcal J_X^p/\mathcal J_X^{p+1}.
\]
\end{lemma}

\begin{proof}
By definition, \(\sigma^{(k)}\) is compatible with the augmentations:
\(
\iota^\sharp_k\circ \sigma^{(k)}=\varepsilon_k,
\)
where
\(
\varepsilon_k:\mathcal A_X^{(k)}\to \mathcal O_{\Xred}
\)
is the projection onto degree \(0\), and
\(
\iota^\sharp_k:\mathcal B_X^{(k)}=\mathcal O_X/\mathcal J_X^{k+1}\to \mathcal O_{\Xred}
\)
is the natural quotient map. Therefore
\[
\sigma^{(k)}\bigl(\ker(\varepsilon_k)\bigr)\subseteq \ker(\iota^\sharp_k).
\]
Since
\(
\ker(\varepsilon_k)=F^1\mathcal A_X^{(k)},
\) and 
\(
\ker(\iota^\sharp_k)=F^1\mathcal B_X^{(k)},
\)
it follows that
\[
\sigma^{(k)}\bigl(F^1\mathcal A_X^{(k)}\bigr)\subseteq F^1\mathcal B_X^{(k)}.
\]
Now \(\sigma^{(k)}\) is an algebra morphism, hence for every \(p\geq 1\),
\[
\sigma^{(k)}\Bigl(\bigl(F^1\mathcal A_X^{(k)}\bigr)^p\Bigr)
\subseteq
\bigl(F^1\mathcal B_X^{(k)}\bigr)^p.
\]
But
\(
\bigl(F^1\mathcal A_X^{(k)}\bigr)^p
=
F^p\mathcal A_X^{(k)},
\) and 
\(
\bigl(F^1\mathcal B_X^{(k)}\bigr)^p
=
F^p\mathcal B_X^{(k)},
\)
so we obtain for all $p$
\[
\sigma^{(k)}\bigl(F^p\mathcal A_X^{(k)}\bigr)\subseteq F^p\mathcal B_X^{(k)}.
\]
Since \(\sigma^{(k)}\) is an isomorphism, from
\(
\iota^\sharp_k\circ \sigma^{(k)}=\varepsilon_k
\)
we also get
\(
\varepsilon_k\circ \bigl(\sigma^{(k)}\bigr)^{-1}=\iota^\sharp_k.
\)
Hence the same argument applied to \(\bigl(\sigma^{(k)}\bigr)^{-1}\) yields
\[
\bigl(\sigma^{(k)}\bigr)^{-1}\bigl(F^p\mathcal B_X^{(k)}\bigr)\subseteq F^p\mathcal A_X^{(k)},
\]
which is equivalent to
\[
F^p\mathcal B_X^{(k)}\subseteq \sigma^{(k)}\bigl(F^p\mathcal A_X^{(k)}\bigr).
\]
Therefore for all $p$
\[
\sigma^{(k)}\bigl(F^p\mathcal A_X^{(k)}\bigr)=F^p\mathcal B_X^{(k)},
\]
hence \(\sigma^{(k)}\) is a filtered isomorphism. It follows that it induces an isomorphism on associated graded algebras
\[
\gr\bigl(\sigma^{(k)}\bigr):
\gr_F\mathcal A_X^{(k)}
\longrightarrow
\gr_F\mathcal B_X^{(k)}.
\]
Now
\(
\gr_F^p\mathcal A_X^{(k)}
=
F^p\mathcal A_X^{(k)}/F^{p+1}\mathcal A_X^{(k)}
\cong
\wedge^p_{\mathcal O_{\Xred}}\mathcal F_X,
\) 
whereas
\(
\gr_F^p\mathcal B_X^{(k)}
=
F^p\mathcal B_X^{(k)}/F^{p+1}\mathcal B_X^{(k)}
\cong
\mathcal J_X^p/\mathcal J_X^{p+1}.
\)\\
In degree \(0\), the induced map is the identity on \(\mathcal O_{\Xred} \), due to compatibility
\(
\iota^\sharp_k\circ \sigma^{(k)}=\varepsilon_k.
\)
In degree \(1\), the induced map is precisely
\(
\gr^1\bigl(\sigma^{(k)}\bigr)
\)
which is the identity by the defining condition of a \(k\)-splitting.

Since \(\gr_F\mathcal A_X^{(k)}\cong \bigwedge^{\leq k}\mathcal F_X\) is generated as an
\(\mathcal O_{\Xred}\)-algebra by its degree-\(1\) part \(\mathcal F_X\), the graded algebra morphism
\(\gr(\sigma^{(k)})\) is uniquely determined by its components in degrees \(0\) and \(1\). Hence it must be the canonical morphism
\(
\wedge^p\mathcal F_X \rightarrow \mathcal J_X^p/\mathcal J_X^{p+1}
\)
induced by multiplication in \(\mathcal O_X\).
\end{proof}

The content of the previous lemma can be summarized by the following obvious corollary. 

\begin{corollary} A \(k\)-splitting of \(X\) is a filtered isomorphism
\(
\sigma^{(k)}:
\bigwedge\nolimits^{\leq k} \mathcal F_X
\xrightarrow{\sim}
\mathcal O_X/\mathcal J_X^{k+1}
\)
lifting the canonical isomorphism of associated graded algebras
\(
\bigwedge\nolimits^{\leq k}\mathcal F_X
\;\xrightarrow{\sim}\;
\Gr  \bigl(\mathcal O_X/\mathcal J_X^{k+1}\bigr).
\)
\end{corollary}

Moreover, one should notice that the 1-splitting exists canonically.

\begin{lemma}[Canonical first-order splitting] \label{lem:canonical-first-splitting}
There is a canonical isomorphism of augmented sheaves of \(\CC\)-superalgebras
\[
\mathcal O_X/\mathcal J_X^2
\cong
\mathcal O_{\Xred}\oplus \FF_X,
\]
where \(\FF_X\) is square-zero. Equivalently, \(X\) has a canonical normalized
\(1\)-splitting
\[
\sigma^{(1)}_{\mathrm{can}}:
\wedge^{\le 1}\FF_X
=
\mathcal O_{\Xred}\oplus \FF_X
\xrightarrow{\sim}
\mathcal O_X/\mathcal J_X^2.
\]
\end{lemma}

\begin{proof}
Modulo \(\mathcal J_X^2\), the ideal \(\mathcal J_X/\mathcal J_X^2\) is square-zero and
purely odd. The parity decomposition of the quotient
\(\mathcal O_X/\mathcal J_X^2\) identifies its even part canonically with
\(\mathcal O_X/\mathcal J_X=\mathcal O_{\Xred}\), and its odd square-zero ideal with
\(\mathcal J_X/\mathcal J_X^2=\FF_X\). This gives the required augmented
superalgebra isomorphism, and the induced map on \(\FF_X\) is the identity.
\end{proof}

The splitting problem is therefore organized by a finite tower of truncations
\(
X^{(0)},X^{(1)},\dots,X^{(n)},
\)
to be compared with the split truncations
\(
\widehat X^{(0)},\widehat X^{(1)},\dots,\widehat X^{(n)}.
\)
Section~3 studies the successive lifting problems in this tower and introduces the corresponding
Green obstruction sheaves and obstruction classes.


\section{The classical splitting tower: liftings and obstructions} 

The splitting problem for a complex supermanifold is naturally finite and inductive. Indeed,
once the split model \(\widehat X\) and its truncations \(\widehat X^{(k)}\) have been fixed, one may ask
at each order whether a given \(k\)-splitting extends to order \(k+1\). This leads to the classical
tower of lifting problems introduced by Green in the holomorphic setting
\cite{Green1982}, and further developed in subsequent work on the classification and obstruction
theory of complex supermanifolds; see for example
\cite{BettadapuraHigherObstructions,BettadapuraObstructedThickenings,
BettadapuraKoszulSplitting, OnishchikClassification}. The resulting obstruction classes live in degree-one cohomology
groups of certain sheaves on \(X_{\mathrm{red}}\), which we call the \emph{Green obstruction sheaves}.

The purpose of this section is to formulate this classical obstruction tower in the precise form
needed later in the paper. We organize the splitting problem as a successive lifting problem for
the truncations
\[
X^{(0)} \subset X^{(1)} \subset \cdots \subset X^{(n)}=X,
\]
identify the sheaves controlling the ambiguity and existence of liftings, and define the associated
obstruction classes. These are the sheaf-theoretic objects that Sections~4 and~5 will later
reinterpret in terms of filtered Maurer--Cartan theory and minimal filtered \(L_\infty\)-models.

\smallskip

Given a certain $k$-splitting of $X$, it is natural to ask if it is possible to obtain a $(k+1)$-splitting in a certain compatible way.  
\begin{definition}
Let \(k\geq 1\), and let
\(
\sigma^{(k)}:
\bigwedge\nolimits^{\leq k}\mathcal F_X
\xrightarrow{\sim}
\mathcal O_X/\mathcal J_X^{k+1}
\)
be a \(k\)-splitting. A \emph{lifting} of \(\sigma^{(k)}\) to order \(k+1\) is a
\((k+1)\)-splitting
\(
\sigma^{(k+1)}:
\bigwedge\nolimits^{\leq k+1}\mathcal F_X
\xrightarrow{\sim}
\mathcal O_X/\mathcal J_X^{k+2}
\)
such that the diagram
\begin{equation}
\begin{tikzcd}
\bigwedge^{\leq k+1}\mathcal F_X
\arrow[r, "\sigma^{(k+1)}", "\sim"']
\arrow[d]
&
\mathcal O_X/\mathcal J_X^{k+2}
\arrow[d]
\\
\bigwedge^{\leq k}\mathcal F_X
\arrow[r, "\sigma^{(k)}", "\sim"']
&
\mathcal O_X/\mathcal J_X^{k+1}
\end{tikzcd}
\end{equation}
commutes, where the vertical arrows are the natural truncation morphisms.
\end{definition}

Thus the splitting problem for \(X\) becomes the problem of constructing a compatible tower of
successive liftings of the canonical \(1\)-splitting:
\[
\sigma^{(1)}_{\mathrm{can}}
\leadsto
\sigma^{(2)}
\leadsto
\cdots
\leadsto
\sigma^{(k)}
\leadsto
\sigma^{(k+1)}
\leadsto \cdots \sigma^{(n)} = \sigma.
\]
A (global) splitting of \(X\) is precisely a compatible choice of liftings at all orders.

\medskip

We now aim at developing a (sheaf-theoretic) obstruction theory for the existence of liftings of $k$-splittings. 

\begin{definition}
For every \(k\geq 1\), let \(\mathscr K_X^{(k+1)}\) be the sheaf on \(X_{\mathrm{red}}\)
whose sections over an open set \(U\subseteq X_{\mathrm{red}}\) are the automorphisms
\begin{equation}
g:
\bigwedge\nolimits^{\leq k+1}_{\mathcal O_U}\mathcal F_X|_U
\xrightarrow{\sim}
\bigwedge\nolimits^{\leq k+1}_{\mathcal O_U}\mathcal F_X|_U
\end{equation}
of sheaves of \(\CC\)-superalgebras such that:

\begin{enumerate}
\item \(g\) is compatible with the augmentations;
\item \(\gr^1(g)=\mathrm{id}_{\mathcal F_X|_U}\);
\item \(g\) induces the identity on
\(
\bigwedge\nolimits^{\leq k}_{\mathcal O_U}\mathcal F_X|_U.
\)
\end{enumerate}
\end{definition}

In the following, we will denote by 
\begin{equation}
\mathcal{T}_X \defeq \mathcal D er_{\mathbb C}
\bigl(\mathcal O_X,\mathcal O_X \bigr)
\end{equation} 
the holomorphic tangent sheaf of the supermanifold $X$. Furthermore, it is crucial to observe that restricting $\mathcal{T}_X$ to the reduced space $\Xred$ one gets the decomposition 
\begin{equation}
\mathcal{T}_X |_{\Xred} \defeq \mathcal{T}_X \otimes_{\OO_X} \OO_{\Xred} \cong \mathcal{T}_{\Xred} \oplus \mathcal H om_{\mathcal O_{X_{\mathrm{red}}}} (\FF_X, \OO_{\Xred}),
\end{equation}
where $\mathcal{T}_{\Xred}$ is the holomorphic tangent sheaf of $\Xred.$ \\
The sheaves appearing in the above decomposition $\mathcal{T}_X |_{\Xred} \defeq \mathcal{T}_{+} \oplus \mathcal{T}_-$ enter the following proposition.

\begin{proposition}
For every \(k\geq 1\), there is a canonical isomorphism of sheaves of abelian groups
\begin{equation}
\mathscr K_X^{(k+1)}\cong \mathcal Q_X^{(k+1)},
\end{equation}
where
\begin{equation} \label{Qsheaf}
\mathcal Q_X^{(k+1)}
\defeq
\begin{cases}
\mathcal H om_{\mathcal O_{X_{\mathrm{red}}}}
\bigl(\mathcal F_X,\wedge^{k+1}\mathcal F_X\bigr),
& \text{if \(k\) is even},\\[6pt]
\mathcal T_{X_{\mathrm{red}}}\otimes_{\mathcal O_{X_{\mathrm{red}}}}
\wedge^{k+1}\mathcal F_X,
& \text{if \(k\) is odd}.
\end{cases}
\end{equation}
\end{proposition}

\begin{proof}
Let \(U\subseteq X_{\mathrm{red}}\) be open, and let
\(
g\in \mathscr K_X^{(k+1)}(U).
\)
Since \(g\) induces the identity modulo degree \(k+1\), we may write
\(
g=\mathrm{id}+\delta,
\)
where
\[
\delta:
\wedge^{\leq k+1}\mathcal F_X|_U
\longrightarrow
\wedge^{k+1}\mathcal F_X|_U.
\]
Because \((\wedge^{k+1}\mathcal F_X)^2=0\) in the truncated exterior algebra,
the condition that \(g\) be an algebra automorphism is equivalent to the Leibniz rule.
Hence \(\delta\) is an even derivation with values in \(\wedge^{k+1}\mathcal F_X|_U\).

Assume first that \(k\) is even. Then \(k+1\) is odd, so \(\wedge^{k+1}\mathcal F_X\)
is odd. Since \(\delta\) is even, it preserves parity, and therefore
\(
\delta(\mathcal O_{X_{\mathrm{red}}}|_U)=0.
\)
Thus \(\delta\) is determined by its restriction to degree \(1\),
\[
\delta|_{\mathcal F_X|_U}:
\mathcal F_X|_U\longrightarrow \wedge^{k+1}\mathcal F_X|_U,
\]
which is \(\mathcal O_U\)-linear. Conversely, every
\(
\alpha\in
\mathcal H om_{\mathcal O_U}
\bigl(\mathcal F_X|_U,\wedge^{k+1}\mathcal F_X|_U\bigr)
\)
extends uniquely to such a derivation \(\delta_\alpha\), and
\(
g_\alpha \defeq \mathrm{id}+\delta_\alpha
\)
lies in \(\mathscr K_X^{(k+1)}(U)\). Hence
\[
\mathscr K_X^{(k+1)}(U)
\cong
\mathcal H om_{\mathcal O_U}
\bigl(\mathcal F_X|_U,\wedge^{k+1}\mathcal F_X|_U\bigr).
\]

Assume now that \(k\) is odd. Then \(k+1\) is even, so \(\wedge^{k+1}\mathcal F_X\)
is even. Since \(\delta\) is even and \(\mathcal F_X\) is odd, one must have
\(
\delta(\mathcal F_X|_U)=0.
\)
Thus \(\delta\) is determined by its restriction to degree \(0\),
\[
\delta|_{\mathcal O_U}:
\mathcal O_U\longrightarrow \wedge^{k+1}\mathcal F_X|_U,
\]
which is a \(\mathbb C\)-derivation. Therefore
\[
\delta|_{\mathcal O_U}\in
\mathcal D er_{\mathbb C}
\bigl(\mathcal O_U,\wedge^{k+1}\mathcal F_X|_U\bigr)
\cong
\mathcal T_{X_{\mathrm{red}}}|_U
\otimes_{\mathcal O_U}
\wedge^{k+1}\mathcal F_X|_U.
\]
Conversely, every such derivation extends uniquely to an even derivation of
\(\wedge^{\leq k+1}\mathcal F_X|_U\) vanishing on \(\mathcal F_X|_U\), and again
\(\mathrm{id}+\delta\) lies in \(\mathscr K_X^{(k+1)}(U)\).

If \(g_a=\mathrm{id}+\delta_a\) for \(a=1,2\), then \(\delta_1\delta_2=0\) in the
truncated algebra, and therefore
\[
g_1\circ g_2
=
\mathrm{id}+\delta_1+\delta_2.
\]
Thus, under the above identifications, composition in \(\mathscr K_X^{(k+1)}\)
corresponds to addition in \(\mathcal Q_X^{(k+1)}\). These identifications are
compatible with restrictions, hence globalize to the claimed isomorphism.
\end{proof}

The sheaves $\mathcal{Q}^{(k+1)}$ are the classical \emph{Green obstruction sheaves} controlling the extension of a \(k\)-splitting
to order \(k+1\).

\begin{definition}[Set of $(k+1)$-splittings]
Let
\(
\sigma^{(k)}:
\wedge^{\leq k}\mathcal F_X
\xrightarrow{\sim}
\mathcal O_X/\mathcal J_X^{k+1}
\)
be a \(k\)-splitting. For every open set \(U\subseteq X_{\mathrm{red}}\), we denote by
\begin{equation}
\mathscr L_X^{(k+1)}(\sigma^{(k)})(U)
\end{equation}
the set of all \((k+1)\)-splittings of \(X|_U\) lifting \(\sigma^{(k)}|_U\).
\end{definition}

\begin{proposition}
If \(\mathscr L_X^{(k+1)}(\sigma^{(k)})(U)\) is non-empty, then it is a torsor under
\(\mathcal Q_X^{(k+1)}|_U\).

\smallskip

\noindent  More precisely, if
\( 
\sigma^{(k+1)}_U,\tau^{(k+1)}_U
\in
\mathscr L_X^{(k+1)}(\sigma^{(k)})(U),
\) 
then
\begin{equation}
g_U\defeq(\tau^{(k+1)}_U)^{-1}\circ \sigma^{(k+1)}_U
\in \mathscr K_X^{(k+1)}(U)
\cong \mathcal Q_X^{(k+1)}(U),
\end{equation}
and this identifies \(\mathscr L_X^{(k+1)}(\sigma^{(k)})(U)\) with a principal
homogeneous space under \(\mathcal Q_X^{(k+1)}(U)\).
\end{proposition}

\begin{proof}
This is immediate from the previous proposition. The composite
\(
(\tau^{(k+1)}_U)^{-1}\circ \sigma^{(k+1)}_U
\)
is an automorphism of \(\wedge^{\leq k+1}\mathcal F_X|_U\) which is the identity after
truncation to order \(k\), is compatible with the augmentations, and induces the
identity in degree \(1\). Hence it lies in \(\mathscr K_X^{(k+1)}(U)\). Conversely, any
section of \(\mathscr K_X^{(k+1)}(U)\cong \mathcal Q_X^{(k+1)}(U)\) acts on a given lift
by composition.
\end{proof}

\begin{lemma}[Local existence of liftings] \label{lem:local-liftings}
Let
\[
\sigma^{(k)}:
\wedge^{\leq k}\mathcal F_X
\xrightarrow{\sim}
\mathcal O_X/\mathcal J_X^{k+1}
\]
be a \(k\)-splitting. Then every point of \(X_{\mathrm{red}}\) has an open neighbourhood
\(U\) such that \(\sigma^{(k)}|_U\) admits a lifting to a \((k+1)\)-splitting.
\end{lemma}

\begin{proof}
The assertion is local on \(X_{\mathrm{red}}\).  Fix a point
\(p\in X_{\mathrm{red}}\).  After shrinking to a sufficiently small neighbourhood
\(U\ni p\), we may assume that \(U\) is a coordinate polydisc, that
\(\mathcal F_X|_U\) is free, and that \(X|_U\) is split.  Choose a local splitting
\(
\rho_U:
\wedge^\bullet \mathcal F_X|_U
\xrightarrow{\sim}
\mathcal O_X|_U .
\)
It induces compatible truncated splittings
\[
\rho_U^{(k)}:
\wedge^{\le k}\mathcal F_X|_U
\xrightarrow{\sim}
\mathcal O_X|_U/\mathcal J_X^{k+1},
\qquad
\rho_U^{(k+1)}:
\wedge^{\le k+1}\mathcal F_X|_U
\xrightarrow{\sim}
\mathcal O_X|_U/\mathcal J_X^{k+2}.
\]
The composite
\[
a^{(k)}
\defeq
(\rho_U^{(k)})^{-1}\circ \sigma^{(k)}|_U
\]
is an augmentation-preserving filtered automorphism of
\(
A_k\defeq \wedge^{\le k}\mathcal F_X|_U
\)
which induces the identity on the reduced quotient \(\mathcal O_U\) and on the first
graded piece \(\mathcal F_X|_U\).

We now lift \(a^{(k)}\) to a filtered automorphism of
\(
A_{k+1}\defeq \wedge^{\le k+1}\mathcal F_X|_U .
\)
Choose holomorphic coordinates \(z^1,\ldots,z^m\) on \(U\), and choose a local frame
\(\theta^1,\ldots,\theta^n\) of \(\mathcal F_X|_U\).  Let
\(
q:A_{k+1}\rightarrow A_k
\)
be the natural quotient.  Since \(a^{(k)}\) is augmentation-preserving and filtered, we
may choose lifts
\(
Z^a\in (A_{k+1})_{\bar0},
\) and \( \Theta^\alpha\in (A_{k+1})_{\bar1}
\)
such that
\[
q(Z^a)=a^{(k)}(z^a),
\qquad
q(\Theta^\alpha)=a^{(k)}(\theta^\alpha),
\]
and such that
\[
Z^a\equiv z^a \pmod{J_{A_{k+1}}},
\qquad
\Theta^\alpha\equiv \theta^\alpha \pmod{J_{A_{k+1}}^2}.
\]
Here \(J_{A_{k+1}}\subset A_{k+1}\) denotes the ideal generated by the odd variables.

The assignment \(z^a\mapsto Z^a\) determines a morphism of sheaves of holomorphic
algebras
\(
\varphi:\mathcal O_U\rightarrow (A_{k+1})_{\bar0}
\)
by nilpotent Taylor expansion.  Namely, for a holomorphic function
\(f=f(z^1,\ldots,z^m)\) on an open subset \(V\subseteq U\), write
\[
N^a\defeq Z^a-z^a\in J_{A_{k+1}}(V).
\]
Since \(J_{A_{k+1}}^{k+2}=0\), the finite expression
\[
\varphi(f)
=
\sum_{|I|\le k+1}
\frac{1}{I!}
\frac{\partial^{|I|}f}{\partial z^I}
\,
N^I
\]
is well-defined and is compatible with restrictions.  The usual Taylor formula in a
nilpotent thickening shows that \(\varphi\) is a morphism of sheaves of
\(\mathbb C\)-algebras.  Moreover, by construction, \(q\circ\varphi\) agrees with the
\(\mathcal O_U\)-part of \(a^{(k)}\).

Using the \(\mathcal O_U\)-algebra structure on \(A_{k+1}\) defined by \(\varphi\), the
odd sections \(\Theta^\alpha\) determine a unique morphism of sheaves of
supercommutative algebras
\[
a^{(k+1)}:
A_{k+1}
=
\wedge^{\le k+1}\mathcal F_X|_U
\longrightarrow
A_{k+1}
\]
by the rule
\[
z^a\longmapsto Z^a,
\qquad
\theta^\alpha\longmapsto \Theta^\alpha .
\]
Indeed, the elements \(\Theta^\alpha\) are odd, hence they satisfy the required
super-anticommutation relations in the supercommutative algebra \(A_{k+1}\).  Thus the
universal property of the exterior algebra applies.

By construction,
\(
q\circ a^{(k+1)}=a^{(k)}\circ q.
\)
Hence \(a^{(k+1)}\) lifts \(a^{(k)}\).  It remains only to check that
\(a^{(k+1)}\) is an automorphism.  It is filtered, and it induces the identity on
\(
\operatorname{gr}^0 A_{k+1}\cong\mathcal O_U
\) and \( 
\operatorname{gr}^1 A_{k+1}\cong\mathcal F_X|_U .
\)
Since the associated graded algebra
\[
\operatorname{gr} A_{k+1}
\cong
\wedge^{\le k+1}\mathcal F_X|_U
\]
is generated as an \(\mathcal O_U\)-algebra by its degree-one part, it follows that
\(\operatorname{gr}(a^{(k+1)})\) is the identity in every degree.  A filtered
endomorphism of a finite filtered sheaf whose associated graded morphism is an
isomorphism is itself an isomorphism.  Therefore \(a^{(k+1)}\) is a filtered
automorphism of \(A_{k+1}\).

Finally define
\[
\sigma_U^{(k+1)}
\defeq
\rho_U^{(k+1)}\circ a^{(k+1)}:
\wedge^{\le k+1}\mathcal F_X|_U
\longrightarrow
\mathcal O_X|_U/\mathcal J_X^{k+2}.
\]
This is a \((k+1)\)-splitting.  Since \(a^{(k+1)}\) lifts \(a^{(k)}\), and since
\(\rho_U^{(k+1)}\) lifts \(\rho_U^{(k)}\), its truncation to order \(k\) is exactly
\[
\rho_U^{(k)}\circ a^{(k)}
=
\sigma^{(k)}|_U.
\]
Thus \(\sigma_U^{(k+1)}\) is a local lifting of \(\sigma^{(k)}|_U\), as required.
\end{proof}

\emph{Obstruction classes} to the existence of liftings are introduced in the following proposition and subsequent corollary.

\begin{proposition}
Let
\(
\sigma^{(k)}:
\wedge^{\leq k}\mathcal F_X
\xrightarrow{\sim}
\mathcal O_X/\mathcal J_X^{k+1}
\)
be a \(k\)-splitting. Then there is a canonically defined obstruction class
\[
\omega_X^{(k+1)}\bigl(\sigma^{(k)}\bigr)
\in
H^1\bigl(X_{\mathrm{red}},\mathcal Q_X^{(k+1)}\bigr)
\]
whose vanishing is equivalent to the existence of a \((k+1)\)-splitting lifting
\(\sigma^{(k)}\).
\end{proposition}

\begin{proof}
By Lemma~\ref{lem:local-liftings}, choose an open cover \(\{U_i\}\) of
\(X_{\mathrm{red}}\) such that on each \(U_i\) there exists a local \((k+1)\)-splitting
\[
\sigma_i^{(k+1)}\in
\mathscr L_X^{(k+1)}(\sigma^{(k)})(U_i).
\]
On overlaps \(U_{ij}\defeq U_i\cap U_j\), set
\[
g_{ij}\defeq
(\sigma_j^{(k+1)})^{-1}\circ \sigma_i^{(k+1)}
\in
\mathscr K_X^{(k+1)}(U_{ij})
\cong
\mathcal Q_X^{(k+1)}(U_{ij}).
\]
Let
\(
\omega_{ij}^{(k+1)}\in \mathcal Q_X^{(k+1)}(U_{ij})
\)
denote the corresponding section. On triple overlaps \(U_{ijk} \defeq U_i \cap U_j \cap U_k\) one has
\[
g_{jk}\circ g_{ij}=g_{ik}.
\]
Since composition in \(\mathscr K_X^{(k+1)}\) corresponds to addition in
\(\mathcal Q_X^{(k+1)}\), this is equivalent to
\[
\omega_{ij}^{(k+1)}+\omega_{jk}^{(k+1)}=\omega_{ik}^{(k+1)}.
\]
Therefore \(\{\omega_{ij}^{(k+1)}\}\) is a \v{C}ech \(1\)-cocycle with values in
\(\mathcal Q_X^{(k+1)}\).

If one chooses another family of local lifts \(\tau_i^{(k+1)}\), then on each \(U_i\)
there exists
\(
\eta_i\in \mathcal Q_X^{(k+1)}(U_i)
\)
such that, under the torsor action,
\[
\tau_i^{(k+1)}=\sigma_i^{(k+1)}\circ(\mathrm{id}+\eta_i).
\]
The corresponding cocycle changes by the \v{C}ech coboundary of the \(0\)-cochain
\(\{\eta_i\}\). Hence the cohomology class
\[
\omega_X^{(k+1)}\bigl(\sigma^{(k)}\bigr)
\defeq
[\{\omega_{ij}^{(k+1)}\}]
\in
H^1\bigl(X_{\mathrm{red}},\mathcal Q_X^{(k+1)}\bigr)
\]
is well defined.

Finally, this class vanishes if and only if, after modifying the local lifts by a
\(0\)-cochain in \(\mathcal Q_X^{(k+1)}\), the transition cocycle becomes trivial. In that
case the local \((k+1)\)-splittings glue to a global one. Conversely, if a global lift
exists, one may take its restrictions as local lifts, and then the cocycle is trivial.
\end{proof}

\begin{corollary} \label{cor:split2}
A \(k\)-splitting
\(
\sigma^{(k)}:
\wedge^{\leq k}\mathcal F_X
\xrightarrow{\sim}
\mathcal O_X/\mathcal J_X^{k+1}
\)
extends to a \((k+1)\)-splitting if and only if
\[
\omega_X^{(k+1)}\bigl(\sigma^{(k)}\bigr)=0
\quad\text{in}\quad
H^1\bigl(X_{\mathrm{red}},\mathcal Q_X^{(k+1)}\bigr).
\]
Whenever this class vanishes, the set of global \((k+1)\)-splittings extending
\(\sigma^{(k)}\) is a torsor under
\(
H^0\bigl(X_{\mathrm{red}},\mathcal Q_X^{(k+1)}\bigr).
\)
\end{corollary}

\begin{remark}
Corollary~\ref{cor:split2} should be read in two different ways. First, for \(k=1\), one starts from the canonical first-order splitting
\(\sigma^{(1)}_{\mathrm{can}}\) of Lemma~\ref{lem:canonical-first-splitting}. This gives the primary -- or fundamental -- obstruction
\[
\omega_X^{(2)}
\defeq
\omega_X^{(2)}\bigl(\sigma^{(1)}_{\mathrm{can}}\bigr)
\in
H^1\!\bigl(X_{\mathrm{red}},
\mathcal T_{X_{\mathrm{red}}}\otimes \wedge^2\mathcal F_X\bigr),
\]
which is intrinsically attached to the supermanifold \(X\). Its vanishing is the necessary and
sufficient condition for the existence of a \(2\)-splitting.

Second, for \(k\ge2\), the obstruction class
\[
\omega_X^{(k+1)}\!\bigl(\sigma^{(k)}\bigr)\in H^1\!\bigl(X_{\mathrm{red}},\mathcal{Q}_X^{(k+1)}\bigr)
\]
is attached not to \(X\) alone, but to the chosen partial splitting tower up to order \(k\). In
particular, beyond the primary level, the higher obstruction classes are \emph{not} a priori intrinsic
invariants of the underlying supermanifold: they depend on the previously chosen liftings.
This is the source of the classical subtlety emphasized in the literature on higher obstructions,
and it is precisely this dependence on choices that will later motivate the filtered Maurer--Cartan
and minimal \(L_\infty\)-reformulations developed in Sections~4 and~5.

Thus the sheaf-theoretic obstruction tower is finite and effective, but from the secondary stage
onward it is inherently relative to a chosen partial splitting data.
\end{remark}

\begin{remark}[Postnikov-style view]
The tower
\[
X_{\mathrm{red}}=X^{(0)} \subset X^{(1)} \subset X^{(2)} \subset \cdots \subset X
\]
is formally analogous to a Postnikov tower, \emph{i.e.} to a tower built by successive extensions whose
obstructions are encoded by \(k\)-invariants. In the present situation, a \(k\)-splitting identifies
\(X^{(k)}\) with the \(k\)-th split truncation \(\widehat X^{(k)}\), and the class
\(
\omega_X^{(k+1)}\!\bigl(\sigma^{(k)}\bigr)
\)
plays the role of the corresponding obstruction to lifting this identification to the next stage.
This analogy is only heuristic, but it captures the fact that the splitting problem is organized as
a finite sequence of successive lifting problems.
\end{remark}

The following sections reinterpret this tower in three increasingly intrinsic ways.
Section~4 realizes the tower as a filtration problem for a Dolbeault dg Lie algebra.
Section~5 transfers this dg Lie algebra to a minimal filtered \(L_\infty\)-model on
cohomology. Sections~6 and~7 then relate the same data to the affine Atiyah class and
to explicit \v{C}ech representatives.

\section{The splitting problem as a formal moduli problem}

In this section, we provide a reinterpretation of the splitting problem for a fixed holomorphic
supermanifold \(X\) in Dolbeault--Maurer--Cartan terms. As reviewed in the previous
section, the classical approach, going back to Green and Palamodov, describes the failure
of splitting by a finite sequence of lifting problems for the truncations \(X^{(k)}\),
governed by the obstruction sheaves \(\mathcal Q_X^{(k)}\) and their cohomology classes
\(\omega_X^{(k)}\) \cite{BettadapuraHigherObstructions, Green1982, PalamodovInvariants}.
Our aim is to package the same finite obstruction theory into a single filtered dg Lie
algebra, so that the successive splitting obstructions arise as the graded shadows of one
Maurer--Cartan deformation.

\smallskip

Let \(X\) be a complex supermanifold of dimension \(m|n\). Throughout this section we set
\begin{equation}
\widehat{\mathcal O}_X
\defeq
\bigwedge\nolimits^\bullet_{\mathcal O_{\Xred}}\mathcal F_X,
\qquad
\widehat{\mathcal J}_X
\defeq
\bigoplus_{r\geq 1}\wedge^r_{\mathcal O_{\Xred}}\mathcal F_X,
\end{equation}
and we write
\(
\widehat X\defeq \bigl(\Xred,\widehat{\mathcal O}_X\bigr)
\)
for the split model associated with \(X\). For every \(k\geq 0\), we denote by
\begin{equation}
\widehat X^{(k)}
\defeq
\Bigl(\Xred,\bigwedge\nolimits^{\leq k}_{\mathcal O_{\Xred}}\mathcal F_X\Bigr),
\qquad
X^{(k)}
\defeq
\bigl(\Xred,\mathcal O_X/\mathcal J_X^{k+1}\bigr)
\end{equation}
the \(k\)-th truncations of the split model and of \(X\), respectively. Thus the
splitting problem is the problem of constructing a compatible tower of isomorphisms
\begin{equation}
\sigma^{(1)}\leadsto \sigma^{(2)}\leadsto \cdots \leadsto \sigma^{(n)},
\qquad
\sigma^{(k)}:\widehat X^{(k)}\xrightarrow{\sim}X^{(k)},
\end{equation}
where \(\sigma^{(1)}\) is the canonical first-order splitting and each
\(\sigma^{(k+1)}\) lifts \(\sigma^{(k)}\). As we observed, since \(\wedge^r\mathcal F_X=0\) for
\(r>n\), this tower is finite.

The construction below uses the \emph{smooth} split model
\[
\widehat X^\infty
=
\bigl(\Xred,\widehat{\mathcal O}_X^\infty\bigr),
\qquad
\widehat{\mathcal O}_X^\infty
=
\bigwedge\nolimits^\bullet_{\mathcal C^\infty_{\Xred}}F_X^\infty,
\]
where \(F_X^\infty\) is the underlying \(C^\infty\) complex vector bundle of
\(\mathcal F_X\). By the smooth splitting theorem -- or Batchelor's theorem \cite{BatchelorStructure} --, the underlying \(C^\infty\)
supermanifold of \(X\) is isomorphic to \(\widehat X^\infty\). Such a smooth splitting
is not canonical. After choosing one, the holomorphic structure of \(X\) can be
transported to \(\widehat X^\infty\) and represented by a Maurer--Cartan element
\(\mu_X\). A different smooth splitting changes \(\mu_X\) by gauge equivalence. Hence
the intrinsic datum associated with \(X\) will be the gauge class
\[
[\mu_X]\in MC(L_X)/\mathrm{gauge},
\]
not a distinguished representative.

The filtered dg Lie algebra \(L_X\) introduced below governs parity-preserving
integrable deformations of the split Dolbeault operator which are trivial on
\(\Xred\) and on the first odd neighbourhood. The actual splitting problem for the
given supermanifold \(X\) is then the problem of deciding how deeply the specific gauge
class \([\mu_X]\) can be moved into the natural filtration of \(L_X\). The leading terms
of such filtration-normalized representatives reproduce the classical obstruction classes
\(\omega_X^{(2j)}\) and \(\omega_X^{(2j+1)}\).

\subsection{The smooth split model and its graded tangent Lie superalgebra}

Let \(F_X^\infty\) be the underlying \(C^\infty\) complex vector bundle of
\(\mathcal F_X\), and define
\begin{equation}
\widehat{\mathcal O}_X^\infty
\defeq
\bigwedge\nolimits^\bullet_{\mathcal C^\infty_{\Xred}}F_X^\infty,
\qquad
\widehat{\mathcal J}_X^\infty
\defeq
\bigoplus_{r\geq 1}\wedge^r_{\mathcal C^\infty_{\Xred}}F_X^\infty.
\end{equation}
The split holomorphic structure on \(\widehat X\) is encoded by the Dolbeault operator
\begin{equation}
\bar\partial_0:
\widehat{\mathcal O}_X^\infty
\longrightarrow
\mathcal A_{\Xred}^{0,1}\otimes_{\mathcal C^\infty_{\Xred}}\widehat{\mathcal O}_X^\infty,
\end{equation}
obtained by extending the Dolbeault operators of \(\mathcal O_{\Xred}\) and
\(\mathcal F_X\) as a derivation.

We shall use the following convention. The smooth tangent sheaf of the split model means
the sheaf of \(\CC\)-linear derivations whose restriction to
\(\mathcal C^\infty_{\Xred}\) is of type \((1,0)\). Thus it is the smooth bundle
underlying the holomorphic tangent sheaf of the split model.

\begin{definition}
Let
\(
\mathcal T_{\widehat X^\infty}
\defeq
\Der^{1,0}_{\CC}\!\bigl(\widehat{\mathcal O}_X^\infty\bigr)
\)
be the sheaf of \((1,0)\)-derivations of the \(C^\infty\) split model. For every
integer \(p\geq -1\), let
\(
\mathcal T_{\widehat X^\infty}^{\langle p\rangle}
\subset
\mathcal T_{\widehat X^\infty}
\)
denote the subsheaf of derivations of homogeneous exterior weight \(p\), namely
\begin{equation}
\mathcal T_{\widehat X^\infty}^{\langle p\rangle}(U)
\defeq
\Bigl\{
\nu\in \Der^{1,0}_{\CC}\!\bigl(\widehat{\mathcal O}_X^\infty(U)\bigr)
\;\Big|\;
\nu\bigl(\wedge^rF_X^\infty(U)\bigr)
\subset
\wedge^{r+p}F_X^\infty(U)
\text{ for all }r\geq 0
\Bigr\}.
\end{equation}
\end{definition}

Thus \(\mathcal T_{\widehat X^\infty}^{\langle p\rangle}\) consists of derivations which raise
exterior degree by exactly \(p\). We set \(\wedge^rF_X^\infty=0\) for \(r<0\) and for
\(r>n\). In particular, \(\mathcal T_{\widehat X^\infty}^{\langle p\rangle}=0\) for
\(p<-1\). Since \(\widehat{\mathcal O}_X^\infty\) is \(\mathbb Z\)-graded, one has a
direct sum decomposition
\begin{equation}
\mathcal T_{\widehat X^\infty}
=
\bigoplus_{p\geq -1}\mathcal T_{\widehat X^\infty}^{\langle p\rangle}.
\end{equation}
Moreover, \(\mathcal T_{\widehat X^\infty}^{\langle p\rangle}\) has parity
\(p\bmod 2\), and the supercommutator satisfies
\begin{equation}
\bigl[
\mathcal T_{\widehat X^\infty}^{\langle p\rangle},
\mathcal T_{\widehat X^\infty}^{\langle q\rangle}
\bigr]
\subset
\mathcal T_{\widehat X^\infty}^{\langle p+q\rangle}.
\end{equation}
The degree \(-1\) piece is
\(
\mathcal T_{\widehat X^\infty}^{\langle -1\rangle}
\cong
(F_X^\infty)^\vee,
\)
and corresponds locally to the odd derivations \(\partial/\partial\theta^\alpha\). The
splitting problem itself will involve only the parity-even pieces of exterior weight at
least \(2\).

\begin{proposition}\label{splitTp}
For every \(p\geq 0\), there is a natural short exact sequence of sheaves of
\(\mathcal C^\infty_{\Xred}\)-modules
\begin{equation}
0
\longrightarrow
\Hom_{\mathcal C^\infty_{\Xred}}
\!\bigl(F_X^\infty,\wedge^{p+1}F_X^\infty\bigr)
\longrightarrow
\mathcal T_{\widehat X^\infty}^{\langle p\rangle}
\overset{\rho_p}{\longrightarrow}
T_{\Xred}^\infty\otimes_{\mathcal C^\infty_{\Xred}}\wedge^pF_X^\infty
\longrightarrow
0,
\end{equation}
where \(T_{\Xred}^\infty\) denotes the \(C^\infty\) complex vector bundle underlying
\(\mathcal T_{\Xred}\).
\end{proposition}

\begin{proof}
Let \(U\subseteq\Xred\) be open, and let
\(\nu\in \mathcal T_{\widehat X^\infty}^{\langle p\rangle}(U)\). Since
\(\widehat{\mathcal O}_X^\infty(U)\) is generated by
\(\mathcal C^\infty_{\Xred}(U)\) and \(F_X^\infty(U)\), the derivation \(\nu\) is
determined by its restrictions to these two pieces.

Its restriction to functions is a \(\CC\)-derivation of type \((1,0)\)
\begin{equation}
\nu|_{\mathcal C^\infty_{\Xred}(U)}:
\mathcal C^\infty_{\Xred}(U)
\longrightarrow
\wedge^pF_X^\infty(U),
\end{equation}
hence corresponds canonically to a section of
\(T_{\Xred}^\infty|_U\otimes \wedge^pF_X^\infty|_U\). This defines \(\rho_p\).

The kernel of \(\rho_p\) consists exactly of those derivations vanishing on functions.
For such a derivation, the restriction
\begin{equation}
\nu|_{F_X^\infty(U)}:
F_X^\infty(U)
\longrightarrow
\wedge^{p+1}F_X^\infty(U)
\end{equation}
is \(\mathcal C^\infty_{\Xred}(U)\)-linear by the Leibniz rule. Hence it defines a
section of
\(
\Hom_{\mathcal C^\infty_{\Xred}}
\!\bigl(F_X^\infty,\wedge^{p+1}F_X^\infty\bigr)(U).
\) 

Conversely, every pair consisting of a type \((1,0)\) derivation
\(\mathcal C^\infty_{\Xred}(U)\to\wedge^pF_X^\infty(U)\) and a
\(\mathcal C^\infty_{\Xred}(U)\)-linear morphism
\(F_X^\infty(U)\to\wedge^{p+1}F_X^\infty(U)\) extends uniquely to a derivation of
\(\widehat{\mathcal O}_X^\infty(U)\) of weight \(p\). This gives exactness.
\end{proof}

The operator \(\bar\partial_0\) preserves each homogeneous piece
\(\mathcal T_{\widehat X^\infty}^{\langle p\rangle}\). We denote by
\(\mathcal T_{\widehat X}^{\langle p\rangle}\) the corresponding holomorphic vector
bundle. Since the maps in Proposition~\ref{splitTp} commute with \(\bar\partial_0\), they
are morphisms of holomorphic vector bundles.

\begin{corollary}
For \(j\geq 1\), set
\begin{equation}
\mathcal Q_X^{(2j)}
\defeq
\mathcal T_{\Xred}\otimes_{\mathcal O_{\Xred}}\wedge^{2j}\mathcal F_X,
\qquad
\mathcal Q_X^{(2j+1)}
\defeq
\mathcal H om_{\mathcal O_{\Xred}}
\!\bigl(\mathcal F_X,\wedge^{2j+1}\mathcal F_X\bigr).
\end{equation}
Then Proposition~\ref{splitTp} gives, for every \(j\geq 1\), an exact sequence
\begin{equation}\label{eq:tangent-obstruction-sequence}
0
\longrightarrow
\mathcal Q_X^{(2j+1)}
\longrightarrow
\mathcal T_{\widehat X}^{\langle 2j\rangle}
\overset{\rho_{2j}}{\longrightarrow}
\mathcal Q_X^{(2j)}
\longrightarrow
0.
\end{equation}
\end{corollary}

Thus the obstruction sheaves \(\mathcal Q_X^{(k)}\) arise from the weight-
\(2j\) tangent piece of the split model: \(\mathcal Q_X^{(2j)}\) is the quotient part,
while \(\mathcal Q_X^{(2j+1)}\) is the kernel part.

\subsection{The filtered dg Lie algebra governing the splitting problem}

The full tangent sheaf is the natural ambient graded Lie superalgebra. The splitting
problem, however, concerns parity-preserving deformations of the split holomorphic
structure which are trivial on \(\Xred\) and on \(\mathcal F_X\). Hence only even exterior
weights \(\geq 2\) are relevant.

\begin{definition}
For \(j\geq 1\), define the filtered tails
\begin{equation}
\mathcal T_{\widehat X^\infty,\bar 0}^{[\geq 2j]}
\defeq
\bigoplus_{q\geq j}\mathcal T_{\widehat X^\infty}^{\langle 2q\rangle}.
\end{equation}
In particular,
\begin{equation}
\mathcal T_{\widehat X^\infty,\bar 0}^{[\geq 2]}
=
\bigoplus_{q\geq 1}\mathcal T_{\widehat X^\infty}^{\langle 2q\rangle}.
\end{equation}
\end{definition}

The filtration is finite:
\begin{equation}
\mathcal T_{\widehat X^\infty,\bar 0}^{[\geq 2]}
\supset
\mathcal T_{\widehat X^\infty,\bar 0}^{[\geq 4]}
\supset
\cdots
\supset
\mathcal T_{\widehat X^\infty,\bar 0}^{[\geq 2N]}
\supset 0,
\qquad
N\defeq\left\lfloor\frac n2\right\rfloor.
\end{equation}
Indeed, \(F_X\) has rank \(n\), so \(\wedge^rF_X=0\) for \(r>n\).

For a holomorphic vector bundle \(E\) on \(\Xred\), we write
\begin{equation}
A^{0,\bullet}(\Xred,E)
\defeq
\Gamma\!\Bigl(\Xred,
\mathcal A_{\Xred}^{0,\bullet}\otimes_{\mathcal C^\infty_{\Xred}}E^\infty
\Bigr)
\end{equation}
for its Dolbeault complex. The differential is the Dolbeault operator induced by the
holomorphic structure of \(E\).

\begin{definition}
Define
\begin{equation}
L_X
\defeq
A^{0,\bullet}\!\bigl(\Xred,\mathcal T_{\widehat X,\bar 0}^{[\geq 2]}\bigr),
\end{equation}
and, for every \(j\geq 1\),
\begin{equation}
F^jL_X
\defeq
A^{0,\bullet}\!\bigl(\Xred,\mathcal T_{\widehat X,\bar 0}^{[\geq 2j]}\bigr).
\end{equation}
The cohomological degree on \(L_X\) is the Dolbeault form degree. The differential
\(d\) is induced by \(\bar\partial_0\), and the bracket is the graded commutator of
\((0,\bullet)\)-forms with values in derivations.
\end{definition}

Thus
\[
L_X=F^1L_X\supset F^2L_X\supset\cdots\supset F^{N+1}L_X=0,
\]
and
\begin{equation}
F^jL_X/F^{j+1}L_X
\cong
A^{0,\bullet}\!\bigl(\Xred,\mathcal T_{\widehat X}^{\langle 2j\rangle}\bigr).
\end{equation}

\begin{proposition}\label{L_X}
The triple \((L_X,d,[\ ,\ ])\) is a finite-step filtered dg Lie algebra. More precisely,
\begin{equation}
d(F^jL_X)\subset F^jL_X,
\qquad
[F^iL_X,F^jL_X]\subset F^{i+j}L_X
\end{equation}
for all \(i,j\geq 1\). In particular, \(L_X\) is nilpotent, hence pronilpotent.
\end{proposition}

\begin{proof}
The differential \(d=\bar\partial_0\) preserves each homogeneous piece
\(\mathcal T_{\widehat X^\infty}^{\langle 2j\rangle}\), hence preserves every filtration term
\(F^jL_X\). The bracket is additive in exterior weight:
\[
\bigl[
\mathcal T_{\widehat X^\infty}^{\langle 2i\rangle},
\mathcal T_{\widehat X^\infty}^{\langle 2j\rangle}
\bigr]
\subset
\mathcal T_{\widehat X^\infty}^{\langle 2i+2j\rangle}.
\]
Therefore \([F^iL_X,F^jL_X]\subset F^{i+j}L_X\). Finally, if
\(N=\lfloor n/2\rfloor\), then \(2(N+1)>n\), so no nonzero even derivation of weight
\(\geq 2(N+1)\) occurs. Hence \(F^{N+1}L_X=0\).
\end{proof}

We call \(L_X\) the \emph{Dolbeault dg Lie algebra} of the splitting problem of \(X\).

\subsection{Maurer--Cartan elements and gauge transformations}

We now explain how Maurer--Cartan elements of \(L_X\) encode holomorphic
supermanifold structures on the fixed \(C^\infty\) split model.

\begin{definition}
A \emph{Maurer--Cartan element} of \(L_X\) is an element \(\mu\in L_X^1\) satisfying
\begin{equation}
d\mu+\frac12[\mu,\mu]=0.
\end{equation}
We denote the set of Maurer--Cartan elements by \(MC(L_X)\).
\end{definition}

An element
\(
\mu\in L_X^1
\)
is a \((0,1)\)-form on \(\Xred\) with values in even derivations of
\(\widehat{\mathcal O}_X^\infty\) of exterior weight at least \(2\). Therefore
\begin{equation}
\bar\partial_\mu\defeq \bar\partial_0+\mu
\end{equation}
is again a parity-preserving \((0,1)\)-operator on \(\widehat{\mathcal O}_X^\infty\)
satisfying the Leibniz rule. Since \(\mu\) has exterior weight at least \(2\), it acts
trivially on
\begin{equation}
\widehat{\mathcal O}_X^\infty/(\widehat{\mathcal J}_X^\infty)^2
\cong
\mathcal C^\infty_{\Xred}\oplus F_X^\infty.
\end{equation}
Thus it preserves the fixed holomorphic structures on \(\Xred\) and on
\(\mathcal F_X\). The Maurer--Cartan equation is precisely the integrability condition
\(\bar\partial_\mu^2=0\), since
\begin{equation}
\bar\partial_\mu^2
=
(\bar\partial_0+\mu)^2
=
d\mu+\frac12[\mu,\mu].
\end{equation}

\begin{proposition}\label{prop:bijection_mu}
The assignment
\[
\mu\longmapsto \bar\partial_\mu=\bar\partial_0+\mu
\]
induces a bijection between:

\begin{enumerate}
\item Maurer--Cartan elements \(\mu\in MC(L_X)\);

\item integrable Dolbeault superalgebra structures on
\(\widehat{\mathcal O}_X^\infty\) which induce the fixed holomorphic structures on
\(\Xred\) and \(\mathcal F_X\).
\end{enumerate}
Equivalently, these are holomorphic supermanifold structures on the fixed smooth split
model \(\widehat X^\infty\) with prescribed reduced space and prescribed odd conormal
bundle.
\end{proposition}

\begin{proof}
If \(\mu\in MC(L_X)\), then \(\bar\partial_\mu\) is an integrable
\(\CC\)-antilinear derivation of \(\widehat{\mathcal O}_X^\infty\). Let
\[
\mathcal O_\mu\defeq \ker(\bar\partial_\mu).
\]
Since \(\bar\partial_\mu\) is a derivation, \(\mathcal O_\mu\) is a sheaf of
\(\CC\)-superalgebras. Moreover, \(\bar\partial_\mu\) preserves the
\(\widehat{\mathcal J}_X^\infty\)-adic filtration, and the induced Dolbeault operator on
\((\widehat{\mathcal J}_X^\infty)^r/(\widehat{\mathcal J}_X^\infty)^{r+1}\) is the
split one. Hence the associated graded holomorphic superalgebra is
\(\wedge^\bullet\mathcal F_X\), and the first truncation is
\(\mathcal O_{\Xred}\oplus\mathcal F_X\).

It remains only to check local splitness. Let \(U\subset\Xred\) be a sufficiently small
coordinate ball, and choose \(\bar\partial_0\)-holomorphic split generators
\(
z^1,\ldots,z^m | \theta^1,\ldots,\theta^n
\)
of \(\widehat{\mathcal O}_X^\infty|_U\). Since \(\mu\) has weight at least \(2\),
\[
\bar\partial_\mu z^a\in A^{0,1}\!\bigl(U,(\widehat{\mathcal J}_X^\infty)^2\bigr),
\qquad
\bar\partial_\mu\theta^\alpha
\in A^{0,1}\!\bigl(U,(\widehat{\mathcal J}_X^\infty)^3\bigr).
\]
Using the local \(\bar\partial\)-Poincar\'e lemma for holomorphic vector bundles and
arguing recursively in the finite \(\widehat{\mathcal J}_X^\infty\)-adic filtration, one
corrects these generators to \(\bar\partial_\mu\)-holomorphic generators
\(
\widetilde z^1,\ldots,\widetilde z^m |
\widetilde\theta^1,\ldots,\widetilde\theta^n,
\)
with
\[
\widetilde z^a\equiv z^a\mod (\widehat{\mathcal J}_X^\infty)^2,
\qquad
\widetilde\theta^\alpha\equiv \theta^\alpha\mod (\widehat{\mathcal J}_X^\infty)^3.
\]
The morphism
\begin{equation}
\mathcal O_U\otimes_\CC\wedge^\bullet(\xi^1,\ldots,\xi^n)
\longrightarrow
\mathcal O_\mu|_U,
\qquad
z^a\mapsto\widetilde z^a,
\quad
\xi^\alpha\mapsto\widetilde\theta^\alpha,
\end{equation}
is filtration-preserving and induces the identity on the associated graded. Since the
filtration is finite, it is an isomorphism. Therefore \(\mathcal O_\mu\) defines a
holomorphic supermanifold structure on \(\widehat X^\infty\).

Conversely, let \(\bar\partial'\) be an integrable Dolbeault superalgebra structure on
\(\widehat{\mathcal O}_X^\infty\) inducing the fixed structures on \(\Xred\) and on
\(\mathcal F_X\). Then
\[
\mu\defeq \bar\partial'-\bar\partial_0
\]
is a \((0,1)\)-form with values in even derivations. Since \(\bar\partial'\) and
\(\bar\partial_0\) agree on
\(\widehat{\mathcal O}_X^\infty/(\widehat{\mathcal J}_X^\infty)^2\), the difference
\(\mu\) has exterior weight at least \(2\), so \(\mu\in L_X^1\). The equality
\((\bar\partial')^2=0\) is precisely
\[
d\mu+\frac12[\mu,\mu]=0.
\]
The two constructions are inverse to each other.
\end{proof}

\begin{definition}[Unipotent gauge group of the split model]
Let \(G_X^{\geq 2}\) be the group of parity-preserving automorphisms of the smooth
split algebra \(\widehat{\mathcal O}_X^\infty\) which are congruent to the identity
modulo \((\widehat{\mathcal J}_X^\infty)^2\) and whose logarithm is a \((1,0)\)-derivation.
Equivalently,
\begin{equation}
G_X^{\geq 2}\defeq \exp(L_X^0).
\end{equation}
We call \(G_X^{\geq 2}\) the \emph{unipotent gauge group of the split model}.
\end{definition}

\begin{lemma}[Exponential coordinates for the unipotent gauge group]
The exponential map induces a bijection
\begin{equation}
\exp:L_X^0\xrightarrow{\sim}G_X^{\geq 2},
\end{equation}
with inverse given by the finite logarithm series
\begin{equation}
\log(g)=\sum_{r\geq 1}\frac{(-1)^{r+1}}{r}(g-\id)^r.
\end{equation}
In particular, every \(e^u\in G_X^{\geq 2}\) is the identity on
\(
\widehat{\mathcal O}_X^\infty/(\widehat{\mathcal J}_X^\infty)^2.
\)
\end{lemma}

\begin{proof}
The \(\widehat{\mathcal J}_X^\infty\)-adic filtration is finite, and every
\(u\in L_X^0\) raises it by at least two exterior degrees. Hence the exponential series is
finite and defines a parity-preserving algebra automorphism acting trivially modulo
\((\widehat{\mathcal J}_X^\infty)^2\). Conversely, by definition, an element of
\(G_X^{\geq 2}\) has logarithm in \(L_X^0\), and the usual exponential-logarithm
identities hold because the relevant endomorphisms are nilpotent.
\end{proof}

\begin{definition}[Gauge action]
The \emph{gauge action} of \(u\in L_X^0\) on \(\mu\in MC(L_X)\) is defined by
\begin{equation}
\bar\partial_{e^u*\mu}
=
e^{-u}\circ(\bar\partial_0+\mu)\circ e^u.
\end{equation}
Equivalently, \(e^u*\mu\) is the unique Maurer--Cartan element whose Dolbeault
operator satisfies the displayed identity.
\end{definition}

\begin{proposition}\label{prop:MCgauge}
Let \(\mu,\nu\in MC(L_X)\). Then the following are equivalent:

\begin{enumerate}
\item \(\nu=e^u*\mu\) for some \(u\in L_X^0\);

\item the holomorphic supermanifold structures corresponding to \(\mu\) and \(\nu\) are
isomorphic by an element of the unipotent gauge group \(G_X^{\geq 2}\).
\end{enumerate}
\end{proposition}

\begin{proof}
If \(\nu=e^u*\mu\), then
\[
\bar\partial_\nu=e^{-u}\circ\bar\partial_\mu\circ e^u,
\]
so \(e^u\) intertwines the two Dolbeault operators and therefore restricts to an
isomorphism \(\mathcal O_\nu\cong\mathcal O_\mu\). Since \(e^u\in G_X^{\geq 2}\), this
isomorphism is the identity on the first truncation.

Conversely, let \(g\in G_X^{\geq 2}\) be an isomorphism between the two holomorphic
structures. Then \(g=e^u\) for a unique \(u\in L_X^0\). The fact that \(g\) intertwines
the Dolbeault operators is exactly the identity defining \(\nu=e^u*\mu\).
\end{proof}

Consequently, the set of gauge-equivalence classes of holomorphic supermanifold structures
on the fixed \(C^\infty\) split model, preserving \(\Xred\) and \(\mathcal F_X\), is
naturally identified with
\begin{equation}
MC(L_X)/\mathrm{gauge}.
\end{equation}
More precisely, the {derived deformation object} is the Deligne--Hinich--Getzler
(Maurer--Cartan) \(\infty\)-groupoid of \(L_X\): we will now expand on this point of view.

\subsection{Formal moduli problem and splitting}

Since \(L_X\) is a finite-step nilpotent dg Lie algebra, for every local Artin dg
\(\CC\)-algebra \(A\) with maximal ideal \(\mathfrak m_A\), the dg Lie algebra
\(L_X\otimes\mathfrak m_A\) is nilpotent. Hence its Maurer--Cartan simplicial set is
well-defined. We denote it by
\[
MC_\infty(L_X\otimes\mathfrak m_A),
\]
and regard it as the Deligne--Hinich--Getzler \(\infty\)-groupoid in the $\infty$-category of \(\infty\)-groupoids, that we denote \( \mathbf{Spc}\); see
\cite{GetzlerLieTheory, HinichDGCoalgebras} and \cite[\S1]{CalaqueGrivauxFMP}.

More precisely, we use cohomological grading conventions and let
\(\mathbf{dgArt}^{\mathrm{loc}}\) denote the category of augmented local Artin
commutative dg \(\CC\)-algebras concentrated in non-positive degrees, with nilpotent
augmentation ideal. We write \(\mathfrak m_A\) for the augmentation ideal of
\(A\).
\begin{definition}[The functor $\Def_X$]
Define
\(
\Def_X:\mathbf{dgArt}^{\mathrm{loc}} \rightarrow \Spc
\)
by
\begin{equation}
\Def_X(A)
\defeq
MC_\infty(L_X\otimes\mathfrak m_A).
\end{equation}
For every \(j\geq 1\), define
\begin{equation}
\Def_X^{\geq j}(A)
\defeq
MC_\infty(F^jL_X\otimes\mathfrak m_A).
\end{equation}
\end{definition}

The functor \(\Def_X\) is the {formal moduli problem} governed by \(L_X\), while
\(\Def_X^{\geq j}\) is the formal moduli problem governed by the filtered dg Lie
subalgebra \(F^jL_X\). We shall not need to interpret \(\Def_X^{\geq j}\) as a literal
subobject of \(\Def_X\); it is the formal moduli problem induced by the inclusion of dg
Lie algebras \(F^jL_X\hookrightarrow L_X\).

\begin{proposition}
The tangent complex of \(\Def_X\) at the split point is canonically
\begin{equation}
T_0\Def_X\simeq L_X[1].
\end{equation}
In particular,
\begin{equation}
H^0(T_0\Def_X)\cong H^1(L_X),
\qquad
H^{-1}(T_0\Def_X)\cong H^0(L_X).
\end{equation}
Thus \(H^1(L_X)\) is the space of first-order deformations of the split holomorphic
structure with fixed \(\Xred\) and fixed \(\mathcal F_X\), while \(H^0(L_X)\) is the
space of infinitesimal automorphisms. Obstructions to extending first-order
deformations lie in \(H^2(L_X)\).
\end{proposition}

\begin{proof}
This is the standard tangent--obstruction theory of formal moduli problems governed by
dg Lie algebras in characteristic zero; see \cite{LurieDAGX},
\cite[Th.~4.1]{PridhamUDDT}, and \cite[\S1]{CalaqueGrivauxFMP}.
\end{proof}

\begin{theorem}[Splitting as a filtered formal moduli problem]\label{thm:form_mod}
The parity-preserving split deformation problem on the fixed \(C^\infty\) split model
\(
\widehat X^\infty=\bigl(\Xred,\widehat{\mathcal O}_X^\infty\bigr)
\)
is a finite filtered formal moduli problem governed by
\(
L_X=A^{0,\bullet}\!\bigl(\Xred,\mathcal T_{\widehat X,\bar 0}^{[\geq 2]}\bigr).
\)
More precisely:

\begin{enumerate}
\item The functor
\[
A\longmapsto \Def_X(A)=MC_\infty(L_X\otimes\mathfrak m_A)
\]
is the formal moduli problem associated with \(L_X\).

\item The filtration \(F^\bullet L_X\) induces a finite tower of formal moduli problems
\begin{equation}
\Def_X=\Def_X^{\geq 1}\longleftarrow \Def_X^{\geq 2}\longleftarrow\cdots
\longleftarrow \Def_X^{\geq N}\longleftarrow \Def_X^{\geq N+1}=*,
\end{equation}
where
\[
\Def_X^{\geq j}(A)=MC_\infty(F^jL_X\otimes\mathfrak m_A),
\qquad
N=\left\lfloor\frac n2\right\rfloor.
\]

\item For every ordinary local Artin \(\CC\)-algebra \(A\),
\begin{equation}
\pi_0\Def_X(A)
\cong
MC(L_X\otimes\mathfrak m_A)/\mathrm{gauge},
\end{equation}
and similarly for \(\Def_X^{\geq j}\).

\item The tangent and obstruction theory of \(\Def_X\) is controlled by the cohomology
of \(L_X\), while that of \(\Def_X^{\geq j}\) is controlled by the cohomology of
\(F^jL_X\).
\end{enumerate}
\end{theorem}

\begin{proof}
By Proposition~\ref{L_X}, \(L_X\) is a finite-step nilpotent dg Lie algebra and each
\(F^jL_X\) is a dg Lie subalgebra. Hence all Maurer--Cartan \(\infty\)-groupoids above
are well-defined. The fact that these functors are formal moduli problems is the
Lurie--Pridham equivalence between dg Lie algebras and formal moduli problems in
characteristic zero. The identification of \(\pi_0\) with Maurer--Cartan elements modulo
gauge for ordinary Artin algebras is standard for nilpotent dg Lie algebras. The final
statement is the usual cohomological tangent--obstruction theory of the same
correspondence.
\end{proof}

\subsection{The Maurer--Cartan point attached to \(X\)}

We now return to the fixed holomorphic supermanifold \(X\). Choose a smooth splitting
\begin{equation}\label{eq:smooth-splitting-choice}
\tau^\infty:\widehat X^\infty\xrightarrow{\sim}X^\infty
\end{equation}
which induces the identity on \(\Xred\) and on \(\mathcal F_X\). Transporting the
holomorphic structure of \(X\) along \(\tau^\infty\), we obtain an integrable Dolbeault
operator on \(\widehat{\mathcal O}_X^\infty\) of the form
\begin{equation}
\bar\partial_{\mu_X}=\bar\partial_0+\mu_X,
\qquad
\mu_X\in A^{0,1}\!\bigl(\Xred,\mathcal T_{\widehat X,\bar 0}^{[\geq 2]}\bigr).
\end{equation}
Thus \(\mu_X\in MC(L_X)\). If a different smooth splitting is chosen, the two splittings
differ by a smooth parity-preserving automorphism of \(\widehat X^\infty\) which is the
identity modulo \((\widehat{\mathcal J}_X^\infty)^2\). By Proposition~\ref{prop:MCgauge},
the corresponding Maurer--Cartan elements are gauge-equivalent. Therefore \(X\)
determines an intrinsic gauge class
\begin{equation}
[\mu_X]\in MC(L_X)/\mathrm{gauge}.
\end{equation}

The dg Lie algebra \(L_X\) governs the full split deformation problem. The splitting
problem for the given supermanifold \(X\) is the problem of understanding the position of
\([\mu_X]\) inside the filtration
\[
L_X=F^1L_X\supset F^2L_X\supset\cdots\supset F^{N+1}L_X=0.
\]
The following proposition is the precise link with partial splittings.

\begin{proposition}\label{muXgauge}
Let \([\mu_X]\in MC(L_X)/\mathrm{gauge}\) be the Maurer--Cartan gauge class attached
to \(X\). For every \(j\geq 1\), the following are equivalent:

\begin{enumerate}
\item \(X\) admits a \((2j-1)\)-splitting;

\item the gauge class \([\mu_X]\) contains a representative
\[
\mu^{(j)}\in MC(L_X)\cap F^jL_X^1.
\]
Equivalently, \([\mu_X]\cap F^jL_X^1\neq\varnothing\).
\end{enumerate}
\end{proposition}

\begin{proof}
Assume first that \(X\) admits a \((2j-1)\)-splitting
\(
\sigma^{(2j-1)}:\widehat X^{(2j-1)}\xrightarrow{\sim}X^{(2j-1)}.
\)
Relative to the chosen smooth splitting \(\tau^\infty\), this truncated holomorphic
isomorphism is represented by a smooth parity-preserving automorphism of the
\((2j-1)\)-truncation of \(\widehat X^\infty\), trivial on the first truncation. Since the
\(C^\infty\)-sheaves involved are fine and the filtration is finite, this truncated
automorphism may be lifted to a smooth parity-preserving automorphism of the full split
\(C^\infty\)-algebra, still trivial modulo \((\widehat{\mathcal J}_X^\infty)^2\). By the
exponential lemma above, this automorphism is \(e^u\) for some \(u\in L_X^0\). Replacing
\(\mu_X\) by \(e^u*\mu_X\), the transported Dolbeault operator coincides with
\(\bar\partial_0\) modulo \((\widehat{\mathcal J}_X^\infty)^{2j}\). Equivalently,
\(e^u*\mu_X\in F^jL_X^1\).

Conversely, suppose that \([\mu_X]\) contains a representative
\(\mu^{(j)}\in MC(L_X)\cap F^jL_X^1\). Then
\(\bar\partial_{\mu^{(j)}}=\bar\partial_0+\mu^{(j)}\) coincides with \(\bar\partial_0\) on
\(
\widehat{\mathcal O}_X^\infty/(\widehat{\mathcal J}_X^\infty)^{2j}.
\)
Thus the corresponding \((2j-1)\)-truncation is holomorphically split. Since
\(\mu^{(j)}\) is gauge-equivalent to \(\mu_X\), Proposition~\ref{prop:MCgauge} identifies
this truncation with \(X^{(2j-1)}\). Hence \(X\) admits a \((2j-1)\)-splitting.
\end{proof}

\subsection{Splittings and adapted Maurer--Cartan representatives}

Fix, for the rest of this section, a representative
\(
\mu_X\in MC(L_X)
\) 
of the distinguished gauge class attached to \(X\), obtained from a chosen smooth
splitting as above.

\begin{definition}[Adapted Maurer--Cartan representative] \label{def:adapted-MC-representative}
Let
\(
\sigma^{(2j-1)}:\widehat X^{(2j-1)}\xrightarrow{\sim}X^{(2j-1)}
\)
be a \((2j-1)\)-splitting. A Maurer--Cartan representative
\[
\mu^{(j)}\in MC(L_X)\cap F^jL_X^1
\]
is called \emph{adapted} to \(\sigma^{(2j-1)}\) if there exists \(u^{(j)}\in L_X^0\)
such that
\[
\mu^{(j)}=e^{u^{(j)}}*\mu_X
\]
and the induced isomorphism of \((2j-1)\)-truncations
\[
\widehat X^{(2j-1)}
\xrightarrow{\ e^{u^{(j)}}\ }
(\widehat X,\bar\partial_{\mu_X})^{(2j-1)}
\xrightarrow{\ \tau^\infty\ }
X^{(2j-1)}
\]
agrees with \(\sigma^{(2j-1)}\).
\end{definition}

Thus adaptation, as defined above, records not only the filtration
condition \(\mu^{(j)}\in F^jL_X\), but also the chosen partial splitting that this
normalization realizes.

\begin{definition}[Leading Maurer--Cartan term]
Let
\[
\mu^{(j)}\in MC(L_X)\cap F^jL_X^1
\]
be adapted to a \((2j-1)\)-splitting. The \emph{leading Maurer--Cartan term} of
\(\mu^{(j)}\) is
\begin{equation}
\eta_{2j}
\defeq
\operatorname{pr}_{2j}(\mu^{(j)})
\in
F^jL_X^1/F^{j+1}L_X^1
\cong
A^{0,1}\!\bigl(\Xred,\mathcal T_{\widehat X}^{\langle 2j\rangle}\bigr).
\end{equation}
\end{definition}

\begin{proposition}\label{prop:leadingeta}
The leading term \(\eta_{2j}\) is \(\bar\partial\)-closed. Hence it defines a Dolbeault
cohomology class
\begin{equation}
[\eta_{2j}]_{\bar\partial}
\in
H_{\bar\partial}^{0,1}\!\bigl(\Xred,\mathcal T_{\widehat X}^{\langle 2j\rangle}\bigr)
\cong
H^1\!\bigl(\Xred,\mathcal T_{\widehat X}^{\langle 2j\rangle}\bigr).
\end{equation}
\end{proposition}

\begin{proof}
Since \(\mu^{(j)}\) satisfies the Maurer--Cartan equation,
\[
d\mu^{(j)}+\frac12[\mu^{(j)},\mu^{(j)}]=0.
\]
Because \(\mu^{(j)}\in F^jL_X^1\), we have
\[
[\mu^{(j)},\mu^{(j)}]\in F^{2j}L_X^2\subset F^{j+1}L_X^2,
\]
as \(2j\geq j+1\) for \(j\geq 1\). Reducing the Maurer--Cartan equation modulo
\(F^{j+1}L_X\), we obtain \(d\eta_{2j}=0\). On the quotient
\(F^jL_X/F^{j+1}L_X\), the differential is the Dolbeault operator on
\(A^{0,\bullet}(\Xred,\mathcal T_{\widehat X}^{\langle 2j\rangle})\). Hence
\(\bar\partial\eta_{2j}=0\).
\end{proof}

We shall compare \([\eta_{2j}]_{\bar\partial}\) with the classical obstruction cocycles
through the exact sequence \eqref{eq:tangent-obstruction-sequence}. The following lemma
makes the \v{C}ech--Dolbeault comparison explicit.

\begin{lemma}[\v{C}ech--Dolbeault comparison for the leading term]
\label{lem:CechDolbeaultLeading}
Let
\(
\mu^{(j)}\in MC(L_X)\cap F^jL_X^1
\)
be a Maurer--Cartan representative adapted to a \((2j-1)\)-splitting
\(\sigma^{(2j-1)}\). Write
\[
\eta_{2j}
\defeq
\operatorname{pr}_{2j}(\mu^{(j)})
\in
A^{0,1}\bigl(\Xred,\mathcal T_{\widehat X}^{\langle 2j\rangle}\bigr)
\]
for its leading component. Then \(\eta_{2j}\) is \(\bar\partial\)-closed.

Let \(\mathfrak U=\{U_i\}\) be a sufficiently fine Stein cover of \(\Xred\). Choose
sections
\(
u_i\in
A^{0,0}\bigl(U_i,\mathcal T_{\widehat X}^{\langle 2j\rangle}\bigr)
\)
such that
\(
\bar\partial u_i=-\eta_{2j}|_{U_i}.
\)
Then the local gauge transformations \(\exp(u_i)\) normalize
\(\mu^{(j)}\) modulo \(F^{j+1}L_X\), and the differences
\[
a_{ij}\defeq u_i-u_j
\in
\Gamma\bigl(U_{ij},\mathcal T_{\widehat X}^{\langle 2j\rangle}\bigr)
\]
form a holomorphic \v{C}ech \(1\)-cocycle. In particular, its \v{C}ech cohomology class satisfies
\(
[a_{ij}]
=
[\eta_{2j}]_{\bar\partial}
\)
under the \v{C}ech--Dolbeault isomorphism
\(
H^1\bigl(\Xred,\mathcal T_{\widehat X}^{\langle 2j\rangle}\bigr)
\cong
H^{0,1}_{\bar\partial}
\bigl(\Xred,\mathcal T_{\widehat X}^{\langle 2j\rangle}\bigr).
\)\\
Moreover, applying
\(
\rho_{2j}:
\mathcal T_{\widehat X}^{\langle 2j\rangle}
\rightarrow
\mathcal Q_X^{(2j)}
\)
to the cocycle \(\{a_{ij}\}\) gives the classical \v{C}ech cocycle defining
\(
\omega_X^{(2j)}(\sigma^{(2j-1)})
\in
H^1\bigl(\Xred,\mathcal Q_X^{(2j)}\bigr).
\)
\end{lemma}

\begin{proof}
The Maurer--Cartan equation for \(\mu^{(j)}\) is
\[
\bar\partial\mu^{(j)}
+
\frac12[\mu^{(j)},\mu^{(j)}]=0.
\]
Since \(\mu^{(j)}\in F^jL_X\), the bracket term lies in \(F^{2j}L_X\subseteq
F^{j+1}L_X\). Taking the component of exterior weight \(2j\), one obtains
\(
\bar\partial\eta_{2j}=0.
\)

On a sufficiently fine Stein cover, the Dolbeault--Poincar\'{e} lemma gives sections
\(u_i\) with
\(
\bar\partial u_i=-\eta_{2j}|_{U_i}.
\)
At the level of the associated graded piece \(F^jL_X/F^{j+1}L_X\), the gauge action
of \(\exp(u_i)\) changes the leading term by
\(
\eta_{2j}\mapsto \eta_{2j}+\bar\partial u_i.
\)
Hence \(\exp(u_i)\) makes the weight-\(2j\) component vanish on \(U_i\), or
equivalently normalizes the representative modulo \(F^{j+1}L_X\).

On overlaps \(U_{ij}\), the differences
\(
a_{ij}=u_i-u_j
\)
are \(\bar\partial\)-closed, hence holomorphic sections of
\(\mathcal T_{\widehat X}^{\langle 2j\rangle}\). They satisfy the \v{C}ech cocycle
identity
\(
a_{ij}+a_{jk}=a_{ik},
\)
and therefore define a class in
\(
H^1\bigl(\Xred,\mathcal T_{\widehat X}^{\langle 2j\rangle}\bigr).
\)
By construction of the \v{C}ech--Dolbeault comparison map, this class is precisely the
Dolbeault class of \(\eta_{2j}\).

Finally, the local gauges \(\exp(u_i)\) give local liftings of the fixed
\((2j-1)\)-splitting to order \(2j\). Their transition functions have leading
component \(a_{ij}\). After projecting by
\(
\rho_{2j}:
\mathcal T_{\widehat X}^{\langle 2j\rangle}
\to
\mathcal Q_X^{(2j)},
\)
one obtains exactly the \v{C}ech cocycle measuring the failure of these local liftings to
glue to a global \(2j\)-splitting. This is Green's obstruction cocycle defining
\(\omega_X^{(2j)}(\sigma^{(2j-1)})\).
\end{proof}

\begin{theorem}\label{thm:sheafMC}
Let
\(
\sigma^{(2j-1)}:\widehat X^{(2j-1)}\xrightarrow{\sim}X^{(2j-1)}
\)
be a \((2j-1)\)-splitting of \(X\), and let
\[
\mu^{(j)}\in MC(L_X)\cap F^jL_X^1
\]
be an adapted Maurer--Cartan representative with leading term \(\eta_{2j}\). Then:

\begin{enumerate}
\item The image of \([\eta_{2j}]_{\bar\partial}\) under
\(
H^1\!\bigl(\Xred,\mathcal T_{\widehat X}^{\langle 2j\rangle}\bigr)
\rightarrow
H^1\!\bigl(\Xred,\mathcal Q_X^{(2j)}\bigr)
\)
is the sheaf-theoretic obstruction class
\[
\omega_X^{(2j)}\bigl(\sigma^{(2j-1)}\bigr)
\in
H^1\!\bigl(\Xred,\mathcal Q_X^{(2j)}\bigr)
\]
to lifting \(\sigma^{(2j-1)}\) to a \(2j\)-splitting.

\item Suppose that \(\omega_X^{(2j)}(\sigma^{(2j-1)})=0\), and fix a
\(2j\)-splitting
\(
\sigma^{(2j)}:\widehat X^{(2j)}\xrightarrow{\sim}X^{(2j)}
\)
lifting \(\sigma^{(2j-1)}\). Then there exists a gauge-equivalent representative
\[
\widetilde\mu^{(j)}\in MC(L_X)\cap F^jL_X^1
\]
adapted to the chosen \(2j\)-splitting, whose leading term
\(
\widetilde\eta_{2j}
\in
A^{0,1}\!\bigl(\Xred,\mathcal T_{\widehat X}^{\langle 2j\rangle}\bigr)
\)
takes values in
\(
A^{0,1}\!\bigl(\Xred,\mathcal Q_X^{(2j+1)}\bigr)
\subset
A^{0,1}\!\bigl(\Xred,\mathcal T_{\widehat X}^{\langle 2j\rangle}\bigr).
\)

\item The Dolbeault class
\[
[\widetilde\eta_{2j}]_{\bar\partial}
\in
H^1\!\bigl(\Xred,\mathcal Q_X^{(2j+1)}\bigr)
\]
is the sheaf-theoretic obstruction class
\(
\omega_X^{(2j+1)}\bigl(\sigma^{(2j)}\bigr)
\)
to lifting the chosen \(2j\)-splitting to a \((2j+1)\)-splitting.
\end{enumerate}
\end{theorem}

\begin{proof}
Part (1) is precisely Lemma~\ref{lem:CechDolbeaultLeading} after applying the quotient
morphism \(\rho_{2j}\) in the exact sequence
\[
0\to\mathcal Q_X^{(2j+1)}\to\mathcal T_{\widehat X}^{\langle 2j\rangle}
\overset{\rho_{2j}}{\longrightarrow}\mathcal Q_X^{(2j)}\to0.
\]

Assume now that the projected obstruction vanishes and that a lift
\(\sigma^{(2j)}\) has been fixed. By Part (1), the class of
\(\rho_{2j}(\eta_{2j})\) vanishes in \(H^1(\Xred,\mathcal Q_X^{(2j)})\). Equivalently,
a global smooth normalization may be chosen whose quotient component realizes the chosen
\(2j\)-splitting. Lifting this smooth quotient normalization through the surjection of
smooth vector bundles
\(
\mathcal T_{\widehat X^\infty}^{\langle 2j\rangle}
\rightarrow
(\mathcal Q_X^{(2j)})^\infty,
\)
we obtain \(u\in F^jL_X^0\) such that, modulo \(F^{j+1}L_X\), the gauge transform
\(e^u*\mu^{(j)}\) has leading term whose \(\rho_{2j}\)-image is zero. Replacing
\(\mu^{(j)}\) by this representative gives \(\widetilde\mu^{(j)}\), adapted to the chosen
\(\sigma^{(2j)}\), and with
\[
\widetilde\eta_{2j}\in
A^{0,1}\!\bigl(\Xred,\ker\rho_{2j}\bigr)
=
A^{0,1}\!\bigl(\Xred,\mathcal Q_X^{(2j+1)}\bigr).
\]
This proves (2).

For (3), apply Lemma~\ref{lem:CechDolbeaultLeading} to the normalized representative
\(\widetilde\mu^{(j)}\). Since its leading term lies in the kernel
\(\mathcal Q_X^{(2j+1)}\), the resulting \v{C}ech cocycle is the residual difference between
local \((2j+1)\)-lifts of the chosen \(2j\)-splitting. This is exactly the classical
obstruction cocycle defining
\(\omega_X^{(2j+1)}(\sigma^{(2j)})\).
\end{proof}

The preceding theorem shows that one leading Maurer--Cartan class controls two
successive layers of Green's obstruction tower. Its quotient part gives the obstruction in
\(\mathcal Q_X^{(2j)}\). Once this class vanishes and a \(2j\)-splitting is chosen, the
remaining kernel part gives the obstruction in \(\mathcal Q_X^{(2j+1)}\).

\begin{definition}\label{def:terminal-adapted-MC-sequence}
Let \(X\) be a complex supermanifold of odd dimension \(n\), and set
\(
N\defeq \left\lfloor\frac n2\right\rfloor .
\)
A \emph{terminal compatible sequence of adapted Maurer--Cartan representatives}
for \(X\) is a sequence
\[
\mu^{(1)},\mu^{(2)},\ldots,\mu^{(N)},\mu^{(N+1)}
\]
with the following properties.

\begin{enumerate}
\item For every \(1\leq j\leq N+1\), one has
\(
\mu^{(j)}\in MC(L_X)\cap F^jL_X^1.
\)
Since \(F^{N+1}L_X=0\), this condition implies
\(
\mu^{(N+1)}=0.
\)

\item For every \(1\leq j\leq N+1\), the Maurer--Cartan element
\(\mu^{(j)}\) represents the distinguished gauge class attached to \(X\):
\(
[\mu^{(j)}]_{\mathrm{gauge}}=[\mu_X]_{\mathrm{gauge}}.
\)
In particular, the terminal element \(0=\mu^{(N+1)}\) is required to lie in the
same gauge orbit as \(\mu_X\).

\item For every \(1\leq j\leq N\), the passage
\(
\mu^{(j)}\rightsquigarrow \mu^{(j+1)}
\)
is obtained by the two-step normalization of Theorem~\ref{thm:sheafMC}. More
explicitly, the projected leading class in \(\mathcal Q_X^{(2j)}\) is first killed,
thereby choosing a \(2j\)-splitting; after this choice, the residual kernel class in
\(\mathcal Q_X^{(2j+1)}\) is killed, thereby obtaining a \((2j+1)\)-splitting and
a representative lying in \(F^{j+1}L_X\). If \(2j+1>n\), the residual odd step is
understood to be void.
\end{enumerate}

Equivalently, a terminal compatible sequence is a chain of successive adapted
normalizations
\[
\mu_X
\sim_{\mathrm{gauge}}
\mu^{(1)}
\sim_{\mathrm{gauge}}
\mu^{(2)}
\sim_{\mathrm{gauge}}
\cdots
\sim_{\mathrm{gauge}}
\mu^{(N)}
\sim_{\mathrm{gauge}}
\mu^{(N+1)}=0,
\]
where each arrow is realized by the Green normalization procedure at the
corresponding filtration level. Thus the terminal vanishing is not an additional
formal convention: it asserts that the successive Green normalizations have reached
the zero Maurer--Cartan representative in the gauge class of \(X\).
\end{definition}

\begin{theorem}\label{thm:terminal-sequence-splitting-tower}
The supermanifold \(X\) admits a full splitting tower
\[
\sigma^{(1)}
\rightsquigarrow
\sigma^{(2)}
\rightsquigarrow
\cdots
\rightsquigarrow
\sigma^{(n)}
\]
if and only if it admits a terminal compatible sequence of adapted Maurer--Cartan
representatives
\[
\mu^{(1)},\mu^{(2)},\ldots,\mu^{(N)},\mu^{(N+1)}=0.
\]
Under this correspondence, for every \(1\leq j\leq N\), the leading term of
\(\mu^{(j)}\) recovers, through Theorem~\ref{thm:sheafMC}, the two consecutive
Green obstruction classes
\(
\omega_X^{(2j)}\bigl(\sigma^{(2j-1)}\bigr),
\omega_X^{(2j+1)}\bigl(\sigma^{(2j)}\bigr),
\)
with the convention that the second class is absent when \(2j+1>n\).
\end{theorem}

\begin{proof}
Assume first that a full splitting tower is given. At every odd stage
\((2j-1)\), Proposition~\ref{muXgauge} gives a Maurer--Cartan representative
\(
\mu^{(j)}\in MC(L_X)\cap F^jL_X^1
\)
adapted to the chosen \((2j-1)\)-splitting. Theorem~\ref{thm:sheafMC} identifies
the projected leading class of \(\mu^{(j)}\) with the obstruction to lifting the
\((2j-1)\)-splitting to a \(2j\)-splitting. Once this lift has been chosen, the same
theorem identifies the residual kernel class with the obstruction to lifting the
\(2j\)-splitting to a \((2j+1)\)-splitting. Since the splitting tower is assumed to
exist up to order \(n\), all these obstruction classes vanish in the prescribed order.
Inductively, the two-step normalizations produce
\(
\mu^{(1)},\mu^{(2)},\ldots,\mu^{(N)}.
\)
Finally, because the tower is full, the holomorphic structure of \(X\) is globally
identified with the split holomorphic structure of \(\widehat X\). Equivalently, the
distinguished Maurer--Cartan gauge class contains the zero representative. Hence
we set
\(
\mu^{(N+1)}=0\in F^{N+1}L_X,
\)
and this zero representative is gauge-equivalent to \(\mu_X\). Thus the sequence is
terminal in the sense of Definition~\ref{def:terminal-adapted-MC-sequence}.

Conversely, suppose that a terminal compatible sequence is given. By definition,
for every \(1\leq j\leq N\), the passage from \(\mu^{(j)}\) to \(\mu^{(j+1)}\) is
obtained by killing, in order, the projected class in \(\mathcal Q_X^{(2j)}\) and the
residual kernel class in \(\mathcal Q_X^{(2j+1)}\). By Theorem~\ref{thm:sheafMC},
these are precisely the Green obstruction classes controlling the successive lifts
\(
(2j-1)\rightarrow 2j
\text{ and }
2j\rightarrow 2j+1.
\)
Therefore the required splittings exist successively, starting from the canonical
first-order splitting.

The terminal condition gives
\[
\mu^{(N+1)}=0
\qquad\text{and}\qquad
[\mu^{(N+1)}]_{\mathrm{gauge}}=[\mu_X]_{\mathrm{gauge}}.
\]
Thus the zero Maurer--Cartan element lies in the distinguished gauge class of \(X\).
Equivalently, the transported Dolbeault operator is gauge-equivalent to the split
Dolbeault operator \(\bar\partial_0\). Hence the last normalized representative gives
a splitting beyond the final nonzero Green filtration level. Since \(2N+1\geq n\),
and since \(\mathcal J_X^{n+1}=0\), this is a full splitting of \(X\).
\end{proof}

The essential point is that the filtered Maurer--Cartan formulation we introduced in this section does not replace the
classical obstruction tower by a different theory. It packages the same finite sequence of
obstructions into the filtration of one dg Lie algebra. A representative in \(F^jL_X\) is the
Dolbeault avatar of a \((2j-1)\)-splitting, while its leading term in
\(F^jL_X/F^{j+1}L_X\) contains precisely the next even and odd obstruction classes.


\section{The minimal \(L_\infty\)-model of the splitting problem}

Section~4 reformulated the splitting problem for a fixed holomorphic supermanifold
\(X\) as a filtered Maurer--Cartan problem for the dg Lie algebra \(L_X\). The
filtration of \(L_X\) records the successive orders of the Green splitting tower, and
adapted Maurer--Cartan representatives recover the classical obstruction classes through
their leading terms. The purpose of the present section is to pass from this dg Lie
model to a cohomological model. More precisely, after choosing a filtered contraction of
\(L_X\) onto its cohomology, the homotopy transfer theorem gives a minimal filtered
\(L_\infty\)-algebra on
\[
H_X\defeq H^\bullet(L_X).
\]
This minimal model governs the same pointed formal moduli problem as \(L_X\), but its
equations are written entirely on cohomology.

There is one point of interpretation that will be important throughout the section. The
transferred minimal \(L_\infty\)-structure depends on the chosen contraction, although its
filtered \(L_\infty\)-quasi-isomorphism class does not. Thus the minimal model is
canonical only up to filtered \(L_\infty\)-isomorphism. Similarly, the higher leading
classes that recover Green's obstruction classes depend, from order \(3\) onwards, on
the chosen partial splitting tower. The canonical object is not an individual higher
coordinate, but the filtered Maurer--Cartan deformation problem together with its
successive obstruction maps.

The homotopy transfer technology used below is standard; for background on transferred
\(L_\infty\)-structures and their Maurer--Cartan deformation theory, we refer to
\cite{GetzlerLieTheory,LodayVallette}.

\subsection{The cohomology of \(L_X\)}

Recall from Section~4 that
\[
F^jL_X/F^{j+1}L_X
\cong
A^{0,\bullet}\!\bigl(\Xred,\mathcal T_{\widehat X,\bar 0}^{\langle 2j\rangle}\bigr),
\]
and that the filtration on \(L_X\) is finite. We define
\begin{equation}
H_X\defeq H^\bullet(L_X,d),
\qquad d=\bar\partial_0.
\end{equation}
Since \(d\) preserves the homogeneous exterior weight, the cohomology of \(L_X\)
decomposes as
\begin{equation}\label{eq:HX-weight-decomp}
H_X
\cong
\bigoplus_{j=1}^N H_X^{\langle 2j\rangle},
\qquad
H_X^{\langle 2j\rangle}
\defeq
H^{0,\bullet}_{\bar\partial}\!\bigl(\Xred,\mathcal T_{\widehat X,\bar 0}^{\langle 2j\rangle}\bigr),
\end{equation}
where
\[
N\defeq \Bigl\lfloor\frac n2\Bigr\rfloor,
\qquad n=\rk(\FF_X).
\]
For every cohomological degree \(q\), we write
\begin{equation} \label{eq:HX-mixed-notation}
H_X^{q,\langle 2j\rangle}
\defeq
\left(H_X^{\langle 2j\rangle}\right)^q
=
H^{0,q}_{\bar\partial}
\bigl(\Xred,\mathcal T_{\widehat X,\bar 0}^{\langle 2j\rangle}\bigr).
\end{equation}
Thus
\begin{equation}\label{eq:HX-degree-q}
H_X^q
=
\bigoplus_{j=1}^N H_X^{q,\langle 2j\rangle}.
\end{equation}
In particular,
\begin{equation}\label{eq:HX-deg-1-2}
H_X^1
=
\bigoplus_{j=1}^N
H^{0,1}_{\bar\partial}\!\bigl(\Xred,\mathcal T_{\widehat X,\bar 0}^{\langle 2j\rangle}\bigr),
\qquad
H_X^2
=
\bigoplus_{j=1}^N
H^{0,2}_{\bar\partial}\!\bigl(\Xred,\mathcal T_{\widehat X,\bar 0}^{\langle 2j\rangle}\bigr).
\end{equation}

The weight decomposition induces a finite decreasing filtration on \(H_X\):
\begin{equation}\label{eq:HX-filtration}
F^jH_X
\defeq
\bigoplus_{r\geq j}H_X^{\langle 2r\rangle}.
\end{equation}
Equivalently,
\[
F^jH_X^q
=
F^jH_X\cap H_X^q
=
\bigoplus_{r\geq j}H_X^{q,\langle 2r\rangle}.
\]
Thus \(F^1H_X=H_X\) and \(F^{N+1}H_X=0\). The degree-one part \(H_X^1\) will provide
cohomological coordinates for minimal Maurer--Cartan elements, while \(H_X^2\) will
contain the corresponding Kuranishi obstruction equations.

For even weights we shall sometimes omit the parity subscript and write
\(
\mathcal T_{\widehat X}^{\langle 2j\rangle}
\)
for
\(
\mathcal T_{\widehat X,\bar 0}^{\langle 2j\rangle}
\).

\subsection{Homotopy transfer}

Before applying homotopy transfer, one needs a contraction of the filtered dg Lie algebra
\((L_X,d)\) onto its cohomology. In the present situation such a contraction exists for a
purely linear reason. The weight decomposition of \(L_X\) is finite, and each homogeneous
piece
\(
A^{0,\bullet}\!\bigl(\Xred,\mathcal T_{\widehat X,\bar 0}^{\langle 2j\rangle}\bigr)
\)
is a complex of vector spaces over the ground field. Hence one may choose, noncanonically,
a splitting into boundaries, cohomology representatives, and a complementary subspace.
Performing this construction weight by weight gives a filtered contraction.

\begin{proposition}\label{prop:contr}
There exist filtered linear maps
\[
i:H_X\longrightarrow L_X,
\qquad
p:L_X\longrightarrow H_X,
\qquad
h:L_X\longrightarrow L_X[-1],
\]
such that
\[
p\circ i=\id_{H_X},
\qquad
\id_{L_X}-i\circ p = dh+hd,
\qquad
p\circ h=0,
\qquad
h\circ i=0,
\qquad
h^2=0.
\]
Equivalently, \((L_X,d)\) contracts onto its cohomology \((H_X,0)\) by filtered
homotopy data.
\end{proposition}

\begin{proof}
For each \(j=1,\dots,N\), set
\[
L_X^{\langle 2j\rangle,\bullet}
\defeq
A^{0,\bullet}\!\bigl(\Xred,\mathcal T_{\widehat X,\bar 0}^{\langle 2j\rangle}\bigr),
\]
with differential \(d=\bar\partial_0\). Choose graded vector space splittings
\[
Z_j^\bullet = B_j^\bullet \oplus H_j^\bullet,
\qquad
L_X^{\langle 2j\rangle,\bullet}=Z_j^\bullet \oplus C_j^\bullet,
\]
where
\(
Z_j^\bullet=\ker d,
\) and \(
B_j^\bullet=\operatorname{im}d,
\)
and where \(H_j^\bullet\subset Z_j^\bullet\) is a choice of representatives for
\[
H^\bullet\!\bigl(L_X^{\langle 2j\rangle,\bullet},d\bigr)
=
H_X^{\langle 2j\rangle}.
\]
Since
\(
d:C_j^\bullet\rightarrow B_j^{\bullet+1}
\)
is an isomorphism, these splittings determine a contraction
\[
(H_j^\bullet,0)
\xrightarrow{i_j}
L_X^{\langle 2j\rangle,\bullet}
\xrightarrow{p_j}
(H_j^\bullet,0),
\qquad
h_j:L_X^{\langle 2j\rangle,\bullet}\longrightarrow L_X^{\langle 2j\rangle,\bullet-1},
\]
satisfying the required side conditions. Summing over the finitely many even weights gives
\[
i=\bigoplus_{j=1}^N i_j,
\qquad
p=\bigoplus_{j=1}^N p_j,
\qquad
h=\bigoplus_{j=1}^N h_j.
\]
The resulting maps preserve the homogeneous weights, hence also the induced filtration.
\end{proof}

The contraction is not canonical. The following theorem should therefore be read as
constructing a cohomological model associated with a chosen contraction; different choices
produce filtered \(L_\infty\)-isomorphic minimal models.

\begin{theorem}[Homotopy transfer]
Let
\[
(H_X,0)\xrightarrow{i}(L_X,d)\xrightarrow{p}(H_X,0),
\qquad
h:L_X\longrightarrow L_X[-1],
\]
be a filtered contraction as in Proposition~\ref{prop:contr}. Then the homotopy transfer
theorem endows \(H_X\) with a minimal filtered \(L_\infty\)-structure
\(
\bigl(H_X,\{\ell_r\}_{r\geq 2}\bigr),
\)
together with a filtered \(L_\infty\)-quasi-isomorphism
\[
I:\bigl(H_X,\{\ell_r\}_{r\geq 2}\bigr)
\rightsquigarrow
(L_X,d,[\, ,\, ]),
\]
whose linear term is \(I_1=i\).

Moreover, the transferred brackets preserve homogeneous weights additively: if
\(
u_a\in H_X^{\langle 2j_a\rangle}\) for \( a=1,\dots,r,
\)
then
\[
\ell_r(u_1,\dots,u_r)
\in
H_X^{\langle 2(j_1+\cdots+j_r)\rangle}.
\]
In particular,
\[
\ell_r\bigl(F^{j_1}H_X,\dots,F^{j_r}H_X\bigr)
\subset
F^{j_1+\cdots+j_r}H_X.
\]
\end{theorem}

\begin{proof}
The existence of the transferred minimal \(L_\infty\)-structure and of the
\(L_\infty\)-quasi-isomorphism \(I\) is the standard homotopy transfer theorem applied
to the contraction constructed in Proposition~\ref{prop:contr}; see, for instance,
\cite[Thm.~10.3.3]{LodayVallette}.

The compatibility with the homogeneous weights follows from the transfer formulas. The
differential \(d\), the bracket \([\, ,\, ]\), and the contraction maps \(i\), \(p\), and
\(h\) all preserve the weight decomposition, while the bracket adds weights. Therefore
every rooted-tree expression appearing in the transferred brackets has total weight equal
to the sum of the input weights. This proves both assertions.
\end{proof}

\begin{remark}\label{rem:linf}
The transferred \(L_\infty\)-algebra is minimal, since \(\ell_1=0\). It is also finite
and nilpotent, because the filtration \(F^\bullet H_X\) is finite and the brackets satisfy
\[
\ell_r\bigl(F^{j_1}H_X,\dots,F^{j_r}H_X\bigr)
\subset
F^{j_1+\cdots+j_r}H_X.
\]

The filtered \(L_\infty\)-quasi-isomorphism \(I\) induces, for every local Artin dg
algebra \(A\) over the ground field with maximal ideal \(\mathfrak m_A\), an equivalence
of Maurer--Cartan \(\infty\)-groupoids
\[
MC_\infty(H_X\otimes\mathfrak m_A)
\simeq
MC_\infty(L_X\otimes\mathfrak m_A).
\]
Thus the transferred minimal filtered \(L_\infty\)-algebra governs the same pointed
formal moduli problem as \(L_X\), but in a cohomological model. As we stressed above, this model is canonical
only up to filtered \(L_\infty\)-isomorphism, because the contraction used to construct it
is not canonical.

For ordinary local Artin algebras, passing to connected components recovers the usual
Maurer--Cartan classes modulo gauge:
\[
\pi_0MC_\infty(L_X\otimes\mathfrak m_A)
\cong
MC(L_X\otimes\mathfrak m_A)/\mathrm{gauge}.
\]
See again \cite{GetzlerLieTheory} for the Maurer--Cartan \(\infty\)-groupoid, and
\cite[\S1]{CalaqueGrivauxFMP}, \cite{LurieDAGX}, \cite{PridhamUDDT}, for the formal
moduli interpretation.
\end{remark}

\subsection{Kuranishi lifts and the obstruction map}

Although the higher brackets may be written explicitly by rooted-tree formulas, for our
purposes it is useful to describe the transferred structure by means of the Kuranishi
slice associated with the chosen contraction. 

Let
\(
\alpha\in H_X^1.
\)
We want to construct a distinguished lift
\(
\tau(\alpha)\in L_X^1
\)
which projects to \(\alpha\) and satisfies the gauge-fixing condition \(h(\tau(\alpha))=0\).

\begin{definition}[Kuranishi lift]
Let \(\alpha\in H_X^1\). A \emph{Kuranishi lift} of \(\alpha\) is an element
\(\tau\in L_X^1\) satisfying
\begin{equation}\label{eqn:Kuranishi}
\tau
=
i(\alpha)-\frac12\,h[\tau,\tau].
\end{equation}
When it exists, it is denoted by \(\tau(\alpha)\).
\end{definition}

The equation is a finite nonlinear fixed-point equation with respect to the nilpotent
filtration. It can therefore be solved weight by weight.

\begin{proposition}\label{prop:KuranishiLift}
For every \(\alpha\in H_X^1\), the Kuranishi equation \eqref{eqn:Kuranishi} has a unique
solution
\(
\tau(\alpha)\in L_X^1.
\)
Moreover,
\[
p(\tau(\alpha))=\alpha,
\qquad
h(\tau(\alpha))=0.
\]
\end{proposition}

\begin{proof}
Write
\(
\alpha=\alpha_1+\cdots+\alpha_N,
\) with \(
\alpha_j\in H_X^{1,\langle 2j\rangle},
\)
and set
\(
i(\alpha_j)=a_j
\in
A^{0,1}\!\bigl(\Xred,\mathcal T_{\widehat X,\bar 0}^{\langle 2j\rangle}\bigr).
\)
We look for
\[
\tau=\tau_1+\cdots+\tau_N,
\qquad
\tau_j\in
A^{0,1}\!\bigl(\Xred,\mathcal T_{\widehat X,\bar 0}^{\langle 2j\rangle}\bigr).
\]
Taking the component of weight \(2j\) in \eqref{eqn:Kuranishi} gives
\[
\tau_j
=
a_j-\frac12\sum_{a+b=j}h[\tau_a,\tau_b].
\]
For \(j=1\) the sum is empty, so \(\tau_1=a_1\). If
\(\tau_1,\dots,\tau_{j-1}\) have already been determined, the right-hand side of the
weight-\(2j\) equation depends only on the previously constructed terms and on \(a_j\).
Thus \(\tau_j\) is uniquely determined. Since there are only finitely many weights, this
constructs a unique solution.

Applying \(p\) to \eqref{eqn:Kuranishi} and using \(p i=\id\) and \(p h=0\), one obtains
\(p(\tau(\alpha))=\alpha\). Applying \(h\) and using \(h i=0\) and \(h^2=0\), one obtains
\(h(\tau(\alpha))=0\).
\end{proof}

The Kuranishi lift need not be Maurer--Cartan. Its failure to satisfy the
Maurer--Cartan equation is measured by the following curvature.

\begin{definition}[Kuranishi map]
For \(\alpha\in H_X^1\), define
\[
F(\alpha)
\defeq
d\tau(\alpha)+\frac12[\tau(\alpha),\tau(\alpha)]
\in L_X^2.
\]
The \emph{Kuranishi map} associated with the chosen contraction is
\(
\kappa:H_X^1\rightarrow H_X^2,
\)
given by
\begin{equation}\label{eqn:KuranishiMap}
\kappa(\alpha)
\defeq
p\!\left(F(\alpha)\right)
=
p\!\left(d\tau(\alpha)+\frac12[\tau(\alpha),\tau(\alpha)]\right).
\end{equation}
\end{definition}

Since the filtration is finite, \(\kappa\) is a finite polynomial map in the filtered
cohomological coordinates. Its zero locus is the Kuranishi presentation of the
Maurer--Cartan functor in the chosen slice.

\begin{proposition}\label{prop:KuranishiMC}
For every \(\alpha\in H_X^1\), the following are equivalent:
\begin{enumerate}
\item \(\tau(\alpha)\in MC(L_X)\);
\item \(\kappa(\alpha)=0\).
\end{enumerate}
Moreover, the Maurer--Cartan equation in the transferred minimal \(L_\infty\)-algebra
on \(H_X\) is
\begin{equation}\label{eqn:minimalMC}
\kappa(\alpha)
=
\sum_{r\geq 2}\frac1{r!}\,\ell_r(\alpha,\dots,\alpha)
=
0.
\end{equation}
\end{proposition}

\begin{proof}
If \(\tau(\alpha)\) is Maurer--Cartan, then \(F(\alpha)=0\), hence
\(\kappa(\alpha)=0\).

Conversely, assume \(\kappa(\alpha)=0\), so \(p(F(\alpha))=0\). We first observe that
\(h(F(\alpha))=0\). Indeed, from the contraction identity and the relations
\(p(\tau(\alpha))=\alpha\), \(h(\tau(\alpha))=0\), we obtain
\[
\tau(\alpha)-i(\alpha)=h d\tau(\alpha).
\]
On the other hand, the Kuranishi equation gives
\[
\tau(\alpha)-i(\alpha)=-\frac12h[\tau(\alpha),\tau(\alpha)].
\]
Therefore
\[
hd\tau(\alpha)
=-\frac12h[\tau(\alpha),\tau(\alpha)],
\]
and hence
\[
h(F(\alpha))
=
hd\tau(\alpha)+\frac12h[\tau(\alpha),\tau(\alpha)]
=0.
\]

Applying the contraction identity to \(F(\alpha)\) gives
\[
F(\alpha)-ip(F(\alpha))
=
dh(F(\alpha))+hd(F(\alpha)).
\]
Since \(p(F(\alpha))=0\) and \(h(F(\alpha))=0\), this becomes
\[
F(\alpha)=h d(F(\alpha)).
\]
The Bianchi identity
\[
d(F(\alpha))+[\tau(\alpha),F(\alpha)]=0
\]
then implies
\[
F(\alpha)=-h[\tau(\alpha),F(\alpha)].
\]
Because \(\tau(\alpha)\in F^1L_X^1\), the bracket is filtration-additive, and \(h\)
preserves the filtration, the last identity shows that if \(F(\alpha)\in F^mL_X^2\),
then \(F(\alpha)\in F^{m+1}L_X^2\). Starting from
\(F(\alpha)\in F^1L_X^2\) and iterating, one obtains
\(F(\alpha)\in F^mL_X^2\) for every \(m\). Since the filtration is finite, this forces
\(F(\alpha)=0\). Thus \(\tau(\alpha)\) is Maurer--Cartan.

The formula \eqref{eqn:minimalMC} is the standard Kuranishi form of the
Maurer--Cartan equation for the transferred minimal \(L_\infty\)-structure \cite{LodayVallette}.
\end{proof}

Thus the degree-one cohomology \(H_X^1\) gives the cohomological deformation
coordinates in the chosen Kuranishi slice, and the map
\(
\kappa:H_X^1\rightarrow H_X^2
\)
gives the corresponding obstruction equations. This is a cohomological presentation of
the same pointed formal moduli problem governed by \(L_X\).

\subsection{Weight decomposition of the minimal Maurer--Cartan equation}

Let
\[
\alpha=
\alpha_1+\cdots+\alpha_N\in H_X^1,
\qquad
\alpha_j\in H_X^{1,\langle 2j\rangle}.
\]
Because the transferred brackets preserve total weight, the minimal Maurer--Cartan
equation \eqref{eqn:minimalMC} splits into homogeneous weight components. The system is
triangular: the equation in weight \(2j\) depends only on
\(\alpha_1,\dots,\alpha_{j-1}\).

For every \(2\leq j\leq N\), define
\begin{equation}\label{eqn:partialMC}
\kappa_j(\alpha_1,\dots,\alpha_{j-1})
\defeq
\sum_{r\geq 2}\frac1{r!}
\sum_{\substack{i_1+\cdots+i_r=j\\ i_a\geq 1}}
\ell_r(\alpha_{i_1},\dots,\alpha_{i_r})
\in H_X^{2,\langle 2j\rangle}.
\end{equation}
Since \(r\geq 2\), every index \(i_a\) occurring in the inner sum is strictly smaller
than \(j\). Thus \(\kappa_j\) depends only on lower-weight coordinates.

For example, the first few equations are
\[
\kappa_2(\alpha_1)
=
\frac12\ell_2(\alpha_1,\alpha_1),
\]
\[
\kappa_3(\alpha_1,\alpha_2)
=
\ell_2(\alpha_1,\alpha_2)
+
\frac16\ell_3(\alpha_1,\alpha_1,\alpha_1),
\]
and
\[
\kappa_4(\alpha_1,\alpha_2,\alpha_3)
=
\ell_2(\alpha_1,\alpha_3)
+
\frac12\ell_2(\alpha_2,\alpha_2)
+
\frac12\ell_3(\alpha_1,\alpha_1,\alpha_2)
+
\frac1{24}\ell_4(\alpha_1,\alpha_1,\alpha_1,\alpha_1).
\]

\begin{proposition}[Truncated minimal Maurer--Cartan equation]\label{prop:weightMC}
Let \(1\leq m\leq N\), and set
\[
\alpha_{\leq m}
\defeq
\alpha_1+\cdots+\alpha_m,
\qquad
\alpha_j\in H_X^{1,\langle 2j\rangle}.
\]
Then \(\alpha_{\leq m}\) satisfies the minimal Maurer--Cartan equation modulo
\(F^{m+1}H_X^2\) if and only if
\[
\kappa_2(\alpha_1)=0,
\qquad
\kappa_3(\alpha_1,\alpha_2)=0,
\qquad
\dots,
\qquad
\kappa_m(\alpha_1,
\dots,
\alpha_{m-1})=0.
\]
Equivalently, these equations are the homogeneous components of the minimal
Maurer--Cartan equation of weights \(4,6,\dots,2m\).
\end{proposition}

\begin{proof}
The component of \(\kappa(\alpha)\) in weight \(2j\) is exactly
\(\kappa_j(\alpha_1,
\dots,
\alpha_{j-1})\). Therefore the minimal Maurer--Cartan equation modulo
\(F^{m+1}H_X^2\) is equivalent to the vanishing of all components of weights
\(4,6,\dots,2m\), namely the displayed equations. The case \(m=1\) is empty and says
that every degree-one class satisfies the equation modulo \(F^2H_X^2\).
\end{proof}

In particular, a full minimal Maurer--Cartan element
\(
\alpha=\alpha_1+\cdots+\alpha_N
\)
is equivalent to a finite system of equations
\begin{equation}\label{eqn:full-weight-MC}
\kappa_j(\alpha_1,
\dots,
\alpha_{j-1})=0,
\qquad
2\leq j\leq N.
\end{equation}
Thus the higher brackets of the minimal model do not themselves produce the Green
obstruction classes. Rather, they produce the intrinsic \(H_X^2\)-valued compatibility
equations that the degree-one cohomological coordinates must satisfy in order to arise
from an actual Maurer--Cartan element.

\subsection{Comparison with the sheaf-theoretic obstruction tower}

We now compare the cohomological minimal model with the sheaf-theoretic obstruction
tower. Clearly, the comparison must keep track of the choices already present in the Green tower.
The primary leading class is intrinsic, because the first-order splitting is canonical. At
higher order, the leading class depends on the chosen partial splitting.

Let \(1\leq j\leq N\), and let
\(
\sigma^{(2j-1)}:\widehat X^{(2j-1)}\xrightarrow{\sim}X^{(2j-1)}
\)
be a chosen \((2j-1)\)-splitting. Choose a Maurer--Cartan representative
\[
\mu^{(j)}\in MC(L_X)\cap F^jL_X^1
\]
adapted to \(\sigma^{(2j-1)}\) in the sense of Definition~\ref{def:adapted-MC-representative}, and write
\[
\eta_{2j}
\defeq
\operatorname{pr}_{2j}(\mu^{(j)})
\in
A^{0,1}\!\bigl(\Xred,
\mathcal T_{\widehat X}^{\langle 2j\rangle}\bigr)
\]
for its leading term. By Proposition~\ref{prop:leadingeta}, \(\eta_{2j}\) is
\(\bar\partial\)-closed.

\begin{definition}\label{def:adapted-leading-class}
The \emph{adapted leading class} associated with the chosen \((2j-1)\)-splitting is
\[
\beta_{X,j}\bigl(\sigma^{(2j-1)}\bigr)
\defeq
[\eta_{2j}]_{\bar\partial}
\in
H^1\!\bigl(\Xred,\mathcal T_{\widehat X}^{\langle 2j\rangle}\bigr)
=
H_X^{1,\langle 2j\rangle}.
\]
\end{definition}

\begin{lemma}\label{lem:adapted-leading-well-defined}
The class
\(
\beta_{X,j}\bigl(\sigma^{(2j-1)}\bigr)
\)
is independent of the adapted Maurer--Cartan representative \(\mu^{(j)}\) used to define
it. In general, for \(j\geq 2\), it depends on the chosen \((2j-1)\)-splitting.
\end{lemma}

\begin{proof}
Let \(\mu^{(j)}\) and \(\widetilde\mu^{(j)}\) be two representatives adapted to the same
\((2j-1)\)-splitting. They differ by a gauge transformation whose generator lies in
\(F^jL_X^0\), since both representatives induce the same truncation up to order
\(2j-1\). Modulo \(F^{j+1}L_X\), the gauge action is linear:
\[
e^u*\mu^{(j)}
\equiv
\mu^{(j)}+\bar\partial u
\pmod{F^{j+1}L_X^1}.
\]
Taking the component of weight \(2j\), we obtain
\[
\widetilde\eta_{2j}-\eta_{2j}
=
\bar\partial\operatorname{pr}_{2j}(u).
\]
Hence the two leading terms determine the same Dolbeault cohomology class.

The dependence on the partial splitting for \(j\geq 2\) is the same dependence already
present in the higher Green obstruction classes: changing the preceding partial splitting
changes the normalization condition imposed on the representative.
\end{proof}

The adapted leading class recovers two consecutive levels of Green's obstruction tower:
first its quotient part gives the even obstruction, and, after choosing a further
splitting, its residual kernel part gives the next odd obstruction.

\begin{theorem}[Adapted leading classes and Green obstructions]\label{thm:minimal-green-comparison}
Let \(1\leq j\leq N\), and let
\(\sigma^{(2j-1)}\) be a chosen \((2j-1)\)-splitting. Then:
\begin{enumerate}
\item The image of the adapted leading class under
\[
H^1(\rho_{2j}):
H^1\!\bigl(\Xred,\mathcal T_{\widehat X}^{\langle 2j\rangle}\bigr)
\longrightarrow
H^1\!\bigl(\Xred,\mathcal Q_X^{(2j)}\bigr)
\]
is the even Green obstruction class
\[
H^1(\rho_{2j})
\bigl(\beta_{X,j}(\sigma^{(2j-1)})\bigr)
=
\omega_X^{(2j)}\bigl(\sigma^{(2j-1)}\bigr).
\]

\item Suppose that this class vanishes and that a \(2j\)-splitting
\(
\sigma^{(2j)}:\widehat X^{(2j)}\xrightarrow{\sim}X^{(2j)}
\)
lifting \(\sigma^{(2j-1)}\) has been fixed. Then one may choose an adapted representative
normalized by \(\sigma^{(2j)}\), so that the quotient component of its leading term
vanishes. \\
The remaining kernel-valued class lies in
\(
H^1\!\bigl(\Xred,\mathcal Q_X^{(2j+1)}\bigr),
\)
and, if \(2j+1\leq n\), it is
\(
\omega_X^{(2j+1)}\bigl(\sigma^{(2j)}\bigr).
\)\\
If \(2j+1>n\), then \(\mathcal Q_X^{(2j+1)}=0\) and there is no further odd obstruction
at this step.
\end{enumerate}
\end{theorem}

\begin{proof}
The first assertion is exactly the comparison established in Theorem~\ref{thm:sheafMC}:
under the \v{C}ech--Dolbeault comparison of Lemma~\ref{lem:CechDolbeaultLeading},
the projected leading term
\(
\rho_{2j}(\eta_{2j})
\)
is the classical \v{C}ech cocycle measuring the failure of local \(2j\)-lifts of
\(\sigma^{(2j-1)}\) to glue.

For the second assertion, the vanishing of the projected class is equivalent to the
existence of a global \(2j\)-splitting lifting \(\sigma^{(2j-1)}\). Once such a splitting
\(\sigma^{(2j)}\) is fixed, Theorem~\ref{thm:sheafMC} allows the adapted representative
to be normalized so that its weight-\(2j\) component has zero image under \(\rho_{2j}\).
It therefore takes values in the kernel of \(\rho_{2j}\), which is
\(\mathcal Q_X^{(2j+1)}\) by the exact sequence
\[
0
\longrightarrow
\mathcal Q_X^{(2j+1)}
\longrightarrow
\mathcal T_{\widehat X}^{\langle 2j\rangle}
\overset{\rho_{2j}}{\longrightarrow}
\mathcal Q_X^{(2j)}
\longrightarrow
0.
\]
The resulting kernel-valued cohomology class is exactly the Green obstruction to lifting
\(\sigma^{(2j)}\) to order \(2j+1\), again by Theorem~\ref{thm:sheafMC}. This proves the
claim.
\end{proof}

We now relate this comparison to the transferred minimal model. Let
\(
[\mu_X]\in \pi_0MC_\infty(L_X)
\)
be the connected component determined by the holomorphic supermanifold \(X\), after the
choice of a smooth splitting as in Section~4. Through the filtered \(L_\infty\)-quasi-
isomorphism of Remark~\ref{rem:linf}, this component corresponds to a component of
\(MC_\infty(H_X)\). A representative of this component is a minimal Maurer--Cartan
coordinate
\[
\alpha_X=
\alpha_{X,1}+\cdots+\alpha_{X,N}
\in H_X^1,
\qquad
\alpha_{X,j}\in H_X^{1,\langle 2j\rangle}.
\]
Its components satisfy the finite Kuranishi equations
\begin{equation}\label{eqn:X-minimal-equations}
\kappa_j(\alpha_{X,1},\dots,\alpha_{X,j-1})=0,
\qquad
2\leq j\leq N.
\end{equation}

The components \(\alpha_{X,j}\) should not be identified with the adapted leading
classes \(\beta_{X,j}\) without a compatibility choice. If the Kuranishi representative is
chosen compatibly with a given adapted splitting tower, then its homogeneous coordinates
are represented by the corresponding adapted leading classes. For a different contraction,
Kuranishi slice, or adapted tower, the coordinates are changed by a filtered
\(L_\infty\)-coordinate transformation. What remains invariant is the induced formal moduli
problem and the comparison, through Theorem~\ref{thm:minimal-green-comparison}, with the
Green obstruction classes.

\begin{remark}
The content of Sections~4 and~5 may be summarized as follows. The dg Lie algebra \(L_X\)
governs the pointed formal moduli problem of holomorphic supermanifold structures on the
fixed \(C^\infty\)-split model \(\widehat X^\infty\), with fixed reduced space and fixed odd
conormal bundle. The given holomorphic supermanifold \(X\) determines a distinguished
Maurer--Cartan gauge class \([\mu_X]\). The filtration on \(L_X\) measures how far this
gauge class can be moved toward the split point, and its adapted leading terms recover the
classical Green obstruction tower.

After choosing a filtered contraction of \(L_X\) onto cohomology, the same deformation
problem is represented by a minimal filtered \(L_\infty\)-algebra on \(H_X\). Its degree-one
coordinates give a cohomological presentation of Maurer--Cartan classes, while its higher
brackets give the \(H_X^2\)-valued Kuranishi compatibility equations. Thus the classical
splitting tower, the filtered Maurer--Cartan theory of \(L_X\), and the minimal filtered
\(L_\infty\)-model encode the same deformation problem at three complementary levels:
sheaf-theoretic, dg Lie, and cohomological.
\end{remark}

\section{The affine Atiyah class and the obstruction tower}
\label{sec:affine-atiyah-tower}

Sections~4 and~5 reformulated the Green splitting tower in terms of adapted
Maurer--Cartan representatives, their leading cohomology classes, and a transferred
minimal filtered \(L_\infty\)-model. In those formulations the basic degree-one
classes are
\[
\beta_{X,j}\bigl(\sigma^{(2j-1)}\bigr)
\in
H^1\!\bigl(\Xred,\mathcal T_{\widehat X}^{\langle 2j\rangle}\bigr),
\qquad
1\leq j\leq N,
\qquad
N\defeq \left\lfloor\frac{\rk(\FF_X)}2\right\rfloor .
\]
They depend, for \(j\geq 2\), on the chosen partial splitting tower. Their quotient
parts are the even Green obstruction classes, and their normalized residual kernel parts
are the subsequent odd Green obstruction classes.

The purpose of this section is to explain how the same data are seen from the affine
Atiyah class. The key point is that the affine Atiyah cocycle is not itself
\(\mathcal T_{\widehat X}^{\langle 2j\rangle}\)-valued: it lives in a sheaf of
symmetric-two-form-valued tangent fields. Thus the correct comparison is obtained by
passing to the \(\mathcal J_X\)-adic filtration and then applying a (Donagi--Witten type)
symbol projection. At the primary level this recovers the classical Donagi--Witten
component of the restricted super Atiyah class \cite{DonagiWittenSuperAtiyah}. At higher levels the same construction is
relative to a chosen lower-order splitting.

\subsection{Atiyah and affine Atiyah classes}

Let \(\mathcal E\) be a locally free \(\mathcal O_X\)-module on a complex supermanifold
\(X\). Its first jet sheaf fits into the Atiyah extension
\[
0
\longrightarrow
\Omega_X^1\otimes\mathcal E
\longrightarrow
J^1(\mathcal E)
\longrightarrow
\mathcal E
\longrightarrow
0.
\]
The corresponding extension class
\[
\At(\mathcal E)
\in
\Ext^1_X\!\bigl(\mathcal E,\Omega_X^1\otimes\mathcal E\bigr)
\cong
H^1\!\bigl(X,\Omega_X^1\otimes\mathcal{E}nd_X(\mathcal E)\bigr)
\]
is the Atiyah class of \(\mathcal E\). It is the obstruction to the existence of a
global holomorphic connection on \(\mathcal E\), see the original \cite{Atiyah1957}, or the recent \cite{BettadapuraKoszulSplitting, NojaBVSupermanifolds} specifically for the super case.

We shall use the tangent case. A holomorphic affine connection on \(X\) is a holomorphic
connection on \(\mathcal T_X\). The difference of two torsion-free affine connections is
a section of
\[
\mathfrak A_X
\defeq
\left(
\Sym^2_{\mathrm{gr}}\Omega_X^1\otimes_{\mathcal O_X}\mathcal T_X
\right)_{\bar 0}
\]
where \(\Sym^2_{\mathrm{gr}}\) denotes the graded-symmetric square. Hence the sheaf of
local torsion-free holomorphic affine connections is a torsor under \(\mathfrak A_X\).
Its \v{C}ech class is the \emph{affine Atiyah class}
\[
\at_X
\in
H^1\!\bigl(X,\mathfrak A_X\bigr)
=
H^1\!\bigl(X,\Sym^2_{\mathrm{gr}}\Omega_X^1\otimes\mathcal T_X\bigr).
\]
Equivalently, if \(\{\nabla_i\}\) is a family of local torsion-free holomorphic affine
connections on an open cover \(\{V_i\}\) of \(X\), then
\[
A_{ij}\defeq \nabla_j-\nabla_i
\in
\Gamma\!\bigl(V_{ij},\mathfrak A_X\bigr)
\]
forms a \v{C}ech \(1\)-cocycle representing \(\at_X\).

\begin{remark}[Affine connections and splitting]
The vanishing of \(\at_X\) is equivalent to the existence of a global torsion-free
holomorphic affine connection on \(X\). By Koszul's splitting theorem \cite{Koszul}, the existence of a
global even holomorphic connection forces a complex supermanifold to be split. Thus the
affine Atiyah class is a natural receptacle for splitting-theoretic information, as recently observed in \cite{BettadapuraKoszulSplitting, NojaBVSupermanifolds}. Donagi and
Witten made this precise at the primary level by identifying the first splitting obstruction
with a distinguished component of the restriction of the super Atiyah class to the reduced
space; see \cite{DonagiWittenSuperAtiyah}. 
\end{remark}

\subsection{Restriction to \(\Xred\) and the Donagi--Witten component}

Let
\(
\iota:\Xred\hookrightarrow X
\)
be the canonical embedding. The restriction of \(\at_X\) to \(\Xred\) decomposes, after
using the natural filtration induced by the odd ideal, into the components described by
Donagi and Witten in \cite{DonagiWittenSuperAtiyah}. \\
In particular, the restricted class has ordinary components
corresponding to the affine Atiyah class of \(\Xred\) and to the Atiyah class of the odd
tangent bundle \(\FF_X^\vee\), and it has a further component
\[
\Pi^{\mathrm{DW}}_2(\iota^*\at_X)
\in
H^1\!\bigl(\Xred,\mathcal T_{\Xred}\otimes\wedge^2\FF_X\bigr).
\]
Here \(\Pi^{\mathrm{DW}}_2\) denotes the projection, in the Donagi--Witten
decomposition of \(\iota^*\at_X\), onto the summand
\(
\mathcal T_{\Xred}\otimes\wedge^2\mathcal F_X.
\)
The Donagi--Witten decomposition identifies this component with the primary splitting
obstruction:
\begin{equation}\label{eq:DW-primary-component}
\Pi^{\mathrm{DW}}_2(\iota^*\at_X)
=
\omega_X^{(2)}
\in
H^1\!\bigl(\Xred,\mathcal Q_X^{(2)}\bigr),
\qquad
\mathcal Q_X^{(2)}=
\mathcal T_{\Xred}\otimes\wedge^2\FF_X.
\end{equation}

The higher construction below is a filtered version of \eqref{eq:DW-primary-component}.
The affine Atiyah class will not be decomposed into a direct sum of higher obstruction
classes. Instead, after a lower-order splitting has been fixed, the first non-split symbol
of its affine Atiyah cocycle produces an obstruction class. Its projected part is the next
even obstruction, and, if that obstruction vanishes, the normalized residual part is the
next odd obstruction.

\subsection{The pure odd Hessian symbol}

The higher obstruction classes are not obtained by projecting the affine Atiyah cocycle
directly to \(\mathcal T_{\widehat X}^{\langle 2j\rangle}\). The affine Atiyah cocycle
is a section of
\(
\mathfrak A_X
=
\left(
\Sym^2_{\mathrm{gr}}\Omega_X^1\otimes_{\mathcal O_X}\mathcal T_X
\right)_{\bar 0},
\)
whereas the Green obstruction classes live in the associated quotient and kernel sheaves
of the homogeneous tangent piece
\[
0
\longrightarrow
\mathcal Q_X^{(2j+1)}
\longrightarrow
\mathcal T_{\widehat X}^{\langle 2j\rangle}
\overset{\rho_{2j}}{\longrightarrow}
\mathcal Q_X^{(2j)}
\longrightarrow
0.
\]
The bridge between these two objects is the \emph{pure odd Hessian symbol}: it is the component
of the affine Atiyah tensor in which the two covariant arguments are taken in the odd
directions.

\smallskip

We first fix the normalization intrinsically. For every \(r\geq 0\), let
\[
m_{r,2}:
\wedge^r\FF_X\otimes\wedge^2\FF_X
\longrightarrow
\wedge^{r+2}\FF_X
\]
be exterior multiplication. We denote by
\[
\Delta_{r,2}:
\wedge^{r+2}\FF_X
\longrightarrow
\wedge^r\FF_X\otimes\wedge^2\FF_X
\]
the normalized \((r,2)\)-shuffle coproduct, characterized locally as follows. If
\(s_1,\ldots,s_{r+2}\) are local sections of \(\FF_X\), then
\[
\Delta_{r,2}(s_1\wedge\cdots\wedge s_{r+2})
=
\frac{1}{\binom{r+2}{2}}
\sum_{\substack{I\sqcup J=\{1,\ldots,r+2\}\\ |I|=r,\ |J|=2}}
\varepsilon(I,J)\,
s_I\otimes s_J,
\]
where \(s_I\) and \(s_J\) denote the exterior products of the sections indexed by
\(I\) and \(J\), written in increasing order, and \(\varepsilon(I,J)\) is the Koszul sign
of the corresponding shuffle. Equivalently,
\[
m_{r,2}\circ \Delta_{r,2}
=
\operatorname{id}_{\wedge^{r+2}\FF_X}.
\]
This normalization is canonical over \(\CC\).

\smallskip

For \(1\leq j\leq N\), define
\[
\mathcal H_{\widehat X,\mathrm{ev}}^{\langle 2j\rangle}
\defeq
\mathcal T_{\Xred}\otimes\wedge^{2j-2}\FF_X\otimes\wedge^2\FF_X,
\qquad \mathcal H_{\widehat X,\mathrm{odd}}^{\langle 2j\rangle}
\defeq
\FF_X^\vee\otimes\wedge^{2j-1}\FF_X\otimes\wedge^2\FF_X.
\]
We set
\[
\mathcal H_{\widehat X}^{\langle 2j\rangle}
\defeq
\mathcal H_{\widehat X,\mathrm{ev}}^{\langle 2j\rangle}
\oplus
\mathcal H_{\widehat X,\mathrm{odd}}^{\langle 2j\rangle}.
\]
The even summand records pure odd Hessian symbols with value in
\(\mathcal T_{\Xred}\), while the odd summand records pure odd Hessian symbols with
value in the odd tangent bundle \(\FF_X^\vee\).

\smallskip

Exterior multiplication in the last two factors gives the natural projection
\begin{equation}\label{eq:Pi-even-symbol}
\Pi_{2j}^{\mathrm{ev}}:
\mathcal H_{\widehat X}^{\langle 2j\rangle}
\longrightarrow
\mathcal Q_X^{(2j)}
=
\mathcal T_{\Xred}\otimes\wedge^{2j}\FF_X,
\end{equation}
defined as \(m_{2j-2,2}\) on the even summand and as zero on the odd summand.
Similarly, multiplication in the last two factors gives
\begin{equation}\label{eq:Pi-odd-symbol}
\Pi_{2j}^{\mathrm{odd}}:
\mathcal H_{\widehat X,\mathrm{odd}}^{\langle 2j\rangle}
\longrightarrow
\mathcal Q_X^{(2j+1)}
=
\mathcal H om_{\mathcal O_{\Xred}}
\!\bigl(\FF_X,\wedge^{2j+1}\FF_X\bigr),
\end{equation}
namely
\[
\Pi_{2j}^{\mathrm{odd}}
=
\operatorname{id}_{\FF_X^\vee}\otimes m_{2j-1,2}.
\]
If \(2j+1>\rk(\FF_X)\), then \(\mathcal Q_X^{(2j+1)}=0\), and this map is the zero
map.

\smallskip

The normalized coproducts give canonical right inverses to these projections on the
relevant Green pieces. Namely, define
\[
 h_{2j}^{\mathrm{ev}}:
\mathcal Q_X^{(2j)}
=
\mathcal T_{\Xred}\otimes\wedge^{2j}\FF_X
\longrightarrow
\mathcal H_{\widehat X,\mathrm{ev}}^{\langle 2j\rangle}
\]
by
\[
 h_{2j}^{\mathrm{ev}}
\defeq
\operatorname{id}_{\mathcal T_{\Xred}}\otimes\Delta_{2j-2,2},
\]
and define
\[
 h_{2j}^{\mathrm{odd}}:
\mathcal Q_X^{(2j+1)}
=
\FF_X^\vee\otimes\wedge^{2j+1}\FF_X
\longrightarrow
\mathcal H_{\widehat X,\mathrm{odd}}^{\langle 2j\rangle}
\]
by
\[
 h_{2j}^{\mathrm{odd}}
\defeq
\operatorname{id}_{\FF_X^\vee}\otimes\Delta_{2j-1,2}.
\]
By construction,
\[
\Pi_{2j}^{\mathrm{ev}}\circ h_{2j}^{\mathrm{ev}}
=
\operatorname{id}_{\mathcal Q_X^{(2j)}},
\qquad
\Pi_{2j}^{\mathrm{odd}}\circ h_{2j}^{\mathrm{odd}}
=
\operatorname{id}_{\mathcal Q_X^{(2j+1)}}.
\]

We now record the corresponding local formula. Let \(e_1,\ldots,e_n\) be a local frame
of \(\FF_X\), let \(e^1,\ldots,e^n\) be the dual frame of \(\FF_X^\vee\), and let
\(\theta^1,\ldots,\theta^n\) be the corresponding odd coordinates. Let
\(z^1,\ldots,z^m\) be holomorphic coordinates on \(\Xred\). A homogeneous derivation
\(
v\in
\Gamma\!\bigl(U,\mathcal T_{\widehat X}^{\langle 2j\rangle}\bigr)
\)
can be written locally as
\[
v
=
\sum_a f^a(\theta)\frac{\partial}{\partial z^a}
+
\sum_\gamma g^\gamma(\theta)\frac{\partial}{\partial\theta^\gamma},
\]
with
\(
f^a\in\Gamma(U,\wedge^{2j}\FF_X),
g^\gamma\in\Gamma(U,\wedge^{2j+1}\FF_X).
\)
In this local trivialization, the corresponding pure odd Hessian representative is written as
\[
\operatorname{Hess}_{2j}^{\mathrm{loc}}(v)
=
\sum_a
\frac{\partial}{\partial z^a}\otimes
\Delta_{2j-2,2}(f^a)
+
\sum_\gamma
e^\gamma\otimes
\Delta_{2j-1,2}(g^\gamma).
\]
This formula is not a choice of two distinguished odd variables: the coproducts
\(\Delta_{2j-2,2}\) and \(\Delta_{2j-1,2}\) are the normalized shuffle coproducts on
the exterior coalgebra of \(\FF_X\).  However, the displayed expression should not be
interpreted as defining a canonical global sheaf morphism
\(
\mathcal T_{\widehat X}^{\langle 2j\rangle}
\rightarrow
\mathcal H_{\widehat X}^{\langle 2j\rangle}
\)
on the whole homogeneous tangent sheaf.  For a general
\(v\in\mathcal T_{\widehat X}^{\langle 2j\rangle}\), the invariant content of the
local expression is its projected even component.

The two invariant facts used below are the following.  First, for every local homogeneous
derivation \(v\), the even projection is canonical and satisfies
\[
\Pi_{2j}^{\mathrm{ev}}
\bigl(
\operatorname{Hess}_{2j}^{\mathrm{loc}}(v)
\bigr)
=
\rho_{2j}(v).
\]
Equivalently, this component is the canonical lift
\(
h_{2j}^{\mathrm{ev}}\bigl(\rho_{2j}(v)\bigr)
\in
\mathcal H_{\widehat X,\mathrm{ev}}^{\langle 2j\rangle}
\)
of the projected Green class.

Second, after the projected even obstruction has vanished and the representative has
been normalized into the kernel of \(\rho_{2j}\), the odd projection is canonical.  More
precisely, if
\[
v\in\ker(\rho_{2j})
=
\mathcal Q_X^{(2j+1)}
=
\FF_X^\vee\otimes\wedge^{2j+1}\FF_X,
\]
then
\[
\Pi_{2j}^{\mathrm{odd}}
\bigl(
\operatorname{Hess}_{2j}^{\mathrm{loc}}(v)
\bigr)
=
v.
\]
Equivalently, the relevant Hessian representative is the canonical lift
\(
h_{2j}^{\mathrm{odd}}(v)
\in
\mathcal H_{\widehat X,\mathrm{odd}}^{\langle 2j\rangle}.
\)

Thus the even projection of the pure odd Hessian representative recovers the projected
Green class, while, after the even obstruction has been killed and a \(2j\)-splitting has
been chosen, the odd projection recovers the residual Green class.

\begin{remark}[Invariant content of the pure odd Hessian symbol]
There is no canonical splitting of the tangent exact sequence
\[
0
\longrightarrow
\mathcal Q_X^{(2j+1)}
\longrightarrow
\mathcal T_{\widehat X}^{\langle 2j\rangle}
\overset{\rho_{2j}}{\longrightarrow}
\mathcal Q_X^{(2j)}
\longrightarrow
0.
\]
Accordingly, the construction above should not be read as a canonical decomposition of
\(\mathcal T_{\widehat X}^{\langle 2j\rangle}\) into even and odd Green pieces, nor as a
canonical global Hessian morphism defined on all of
\(\mathcal T_{\widehat X}^{\langle 2j\rangle}\).  The canonical structures used in the
comparison with the Green tower are instead:
\(
\Pi_{2j}^{\mathrm{ev}},
\Pi_{2j}^{\mathrm{odd}},
h_{2j}^{\mathrm{ev}},
h_{2j}^{\mathrm{odd}},
\)
together with the following two projected comparison rules:
\[
\Pi_{2j}^{\mathrm{ev}}
\bigl(
\operatorname{Hess}_{2j}^{\mathrm{loc}}(v)
\bigr)
=
\rho_{2j}(v),
\]
and, after normalization into \(\ker(\rho_{2j})\),
\[
\Pi_{2j}^{\mathrm{odd}}
\bigl(
\operatorname{Hess}_{2j}^{\mathrm{loc}}(v)
\bigr)
=
v.
\]
These projected identities are independent of the local coordinates and frames used to
write the displayed Hessian representative.
\end{remark}

\subsection{Filtered affine Atiyah symbols relative to partial splittings}

Fix \(1\leq j\leq N\), and suppose that a \((2j-1)\)-splitting
\(
\sigma^{(2j-1)}:
\widehat X^{(2j-1)}\xrightarrow{\sim}X^{(2j-1)}
\)
has been chosen. Let \(\mathfrak U=\{U_i\}\) be a sufficiently fine Stein cover of
\(\Xred\), and choose local coordinates on \(X|_{U_i}\) adapted to
\(\sigma^{(2j-1)}\). Thus the transition functions agree with those of the split model up
to order \(2j-1\), and the first possible non-split term has tangent weight \(2j\).

On each adapted coordinate chart, let \(\nabla_i\) be the torsion-free affine connection
with vanishing Christoffel symbols in those coordinates. Let \(\widehat\nabla_i\) be the
corresponding local flat connection on the split model, written in the same adapted
coordinates. The affine Atiyah cocycle of \(X\) and the affine Atiyah cocycle of the split
model are represented by
\[
A_{ij}\defeq\nabla_j-\nabla_i,
\qquad
\widehat A_{ij}\defeq\widehat\nabla_j-\widehat\nabla_i.
\]
Their difference
\[
B_{ij}\defeq A_{ij}-\widehat A_{ij}
\]
is a filtered affine-Atiyah-valued \v{C}ech \(1\)-cocycle. Since the coordinates are
adapted to \(\sigma^{(2j-1)}\), all non-split contributions of tangent weight lower than
\(2j\) vanish. We may therefore take the pure odd Hessian symbol of the first non-split
homogeneous component.

\begin{definition}[Filtered affine Atiyah symbol in weight \(2j\)]
\label{def:filatcoc}
The \emph{filtered affine Atiyah symbol in weight \(2j\)} relative to
\(\sigma^{(2j-1)}\) is the \v{C}ech \(1\)-cochain
\[
\mathfrak a_X^{\langle 2j\rangle}
=
\{\mathfrak a_{ij}^{\langle 2j\rangle}\}
\in
\check C^1\!\bigl(\mathfrak U,
\mathcal H_{\widehat X}^{\langle 2j\rangle}\bigr)
\]
obtained by projecting \(B_{ij}\) to the pure odd Hessian symbol sheaf
\(\mathcal H_{\widehat X}^{\langle 2j\rangle}\) in homogeneous tangent weight \(2j\).
\end{definition}

\begin{proposition}\label{prop:filtat1}
The cochain \(\mathfrak a_X^{\langle 2j\rangle}\) is a \v{C}ech \(1\)-cocycle. Its
cohomology class
\[
\mathfrak{At}^{\langle 2j\rangle}
\bigl(X;\sigma^{(2j-1)}\bigr)
\defeq
[\mathfrak a_X^{\langle 2j\rangle}]
\in
H^1\!\bigl(\Xred,\mathcal H_{\widehat X}^{\langle 2j\rangle}\bigr)
\]
is independent of the adapted cover, of the adapted coordinates, and of the local
torsion-free affine connections used to represent \(\at_X\). It depends, in general, on
the chosen \((2j-1)\)-splitting.

\smallskip

Moreover, if
\(
\mu^{(j)}\in MC(L_X)\cap F^jL_X^1
\)
is adapted to the same \((2j-1)\)-splitting and
\[
\eta_{2j}=\operatorname{pr}_{2j}(\mu^{(j)})
\in
A^{0,1}\!\bigl(\Xred,\mathcal T_{\widehat X}^{\langle 2j\rangle}\bigr)
\]
is its leading term, then the \v{C}ech representative of
\(\beta_{X,j}(\sigma^{(2j-1)})\) from Lemma~\ref{lem:CechDolbeaultLeading} is mapped,
under the normalized pure odd Hessian symbol, to
\(\mathfrak a_X^{\langle 2j\rangle}\).
\end{proposition}

\begin{proof}
The cocycles \(A=\{A_{ij}\}\) and \(\widehat A=\{\widehat A_{ij}\}\) both represent affine
connection torsors. Hence \(B=A-\widehat A\) is a \v{C}ech \(1\)-cocycle with values in
the filtered affine Atiyah sheaf. Projection to a homogeneous associated-graded symbol
commutes with the \v{C}ech differential, so
\(
\check\delta\mathfrak a_X^{\langle 2j\rangle}=0.
\)

Let \(g_{ij}\) be the transition automorphisms in the chosen adapted coordinates, and
write their leading non-split part as
\[
g_{ij}
=
\widehat g_{ij}\circ \exp\bigl(v_{ij}^{\langle 2j\rangle}\bigr)
\quad
\text{modulo higher weights},
\]
with
\(
v_{ij}^{\langle 2j\rangle}
\in
\Gamma\!\bigl(U_{ij},\mathcal T_{\widehat X}^{\langle 2j\rangle}\bigr).
\)
By Lemma~\ref{lem:CechDolbeaultLeading}, the cocycle
\(\{v_{ij}^{\langle 2j\rangle}\}\) represents
\(\beta_{X,j}(\sigma^{(2j-1)})\).

In local coordinates, the difference between the flat affine connections attached to
\(g_{ij}\) and to \(\widehat g_{ij}\) is the Hessian of the coordinate change. Since all
non-split terms of lower weight vanish, the homogeneous weight-\(2j\) part is linear in
\(v_{ij}^{\langle 2j\rangle}\) and is precisely
\[
\operatorname{Hess}_{2j}\bigl(v_{ij}^{\langle 2j\rangle}\bigr)
\in
\Gamma\!\bigl(U_{ij},\mathcal H_{\widehat X}^{\langle 2j\rangle}\bigr).
\]
This is \(\mathfrak a_{ij}^{\langle 2j\rangle}\) by the normalized shuffle-coproduct
definition of the pure odd Hessian symbol..

Changing the adapted coordinates without changing the fixed lower-order splitting changes
\(\{v_{ij}^{\langle 2j\rangle}\}\) by a \v{C}ech coboundary in
\(\mathcal T_{\widehat X}^{\langle 2j\rangle}\), and therefore changes its Hessian symbol by
the corresponding coboundary in \(\mathcal H_{\widehat X}^{\langle 2j\rangle}\). Changing
the local torsion-free connections also changes the affine Atiyah cocycle by a coboundary.
Thus the cohomology class is independent of all auxiliary adapted choices.
\end{proof}

\begin{definition}[Filtered affine Atiyah class in weight \(2j\)]
\label{def:filtAtCl}
Under the hypotheses above, the class
\[
\mathfrak{At}^{\langle 2j\rangle}
\bigl(X;\sigma^{(2j-1)}\bigr)
\in
H^1\!\bigl(\Xred,\mathcal H_{\widehat X}^{\langle 2j\rangle}\bigr)
\]
is called the \emph{filtered affine Atiyah class in weight \(2j\)} relative to the chosen
\((2j-1)\)-splitting.
\end{definition}

\begin{remark}
The filtered affine Atiyah class is not a tangent-valued class. It is a pure odd
Hessian-symbol class extracted from the affine Atiyah cocycle. The tangent-valued classes
of Sections~4 and~5 are recovered from the same local transition data before applying the
Hessian symbol; the Green obstruction classes are obtained from the Hessian-symbol class
by applying \(\Pi_{2j}^{\mathrm{ev}}\) and, after normalization, \(\Pi_{2j}^{\mathrm{odd}}\).
\end{remark}

\subsection{Projected defects and residual classes}

We now state the comparison with the Green obstruction tower. This is the higher-order
analogue of the Donagi--Witten primary decomposition.

\begin{theorem}[Filtered affine Atiyah symbols and Green obstructions]
\label{thm:filteredAtiyahObstructions}
\label{thm:fiteredAtThm}
Let \(1\leq j\leq N\), and fix a \((2j-1)\)-splitting
\(
\sigma^{(2j-1)}:
\widehat X^{(2j-1)}\xrightarrow{\sim}X^{(2j-1)}.
\)
Let
\[
\mathfrak{At}^{\langle 2j\rangle}
\bigl(X;\sigma^{(2j-1)}\bigr)
\in
H^1\!\bigl(\Xred,\mathcal H_{\widehat X}^{\langle 2j\rangle}\bigr)
\]
be the filtered affine Atiyah class in weight \(2j\). Then:

\begin{enumerate}
\item The projected even symbol is the Green obstruction to lifting the chosen
\((2j-1)\)-splitting to a \(2j\)-splitting:
\[
H^1(\Pi_{2j}^{\mathrm{ev}})
\Bigl(
\mathfrak{At}^{\langle 2j\rangle}
\bigl(X;\sigma^{(2j-1)}\bigr)
\Bigr)
=
\omega_X^{(2j)}\bigl(\sigma^{(2j-1)}\bigr)
\]
in
\(
H^1\!\bigl(\Xred,\mathcal Q_X^{(2j)}\bigr).
\)

\item Suppose that
\(
\omega_X^{(2j)}\bigl(\sigma^{(2j-1)}\bigr)=0
\)
and that a \(2j\)-splitting
\(
\sigma^{(2j)}:
\widehat X^{(2j)}\xrightarrow{\sim}X^{(2j)}
\)
lifting \(\sigma^{(2j-1)}\) has been fixed. Then the filtered affine Atiyah symbol may be
normalized so that its even projected component vanishes. Its remaining odd symbol
component defines a class
\[
\widetilde{\mathfrak{At}}^{\langle 2j\rangle}
\bigl(X;\sigma^{(2j)}\bigr)
\in
H^1\!\bigl(\Xred,
\mathcal H_{\widehat X,\mathrm{odd}}^{\langle 2j\rangle}\bigr).
\]
If \(2j+1\leq \rk(\FF_X)\), then
\[
H^1(\Pi_{2j}^{\mathrm{odd}})
\Bigl(
\widetilde{\mathfrak{At}}^{\langle 2j\rangle}
\bigl(X;\sigma^{(2j)}\bigr)
\Bigr)
=
\omega_X^{(2j+1)}\bigl(\sigma^{(2j)}\bigr)
\]
in
\(
H^1\!\bigl(\Xred,\mathcal Q_X^{(2j+1)}\bigr).
\)\\
If \(2j+1>\rk(\FF_X)\), then \(\mathcal Q_X^{(2j+1)}=0\), and there is no residual odd
obstruction at this step.
\end{enumerate}
\end{theorem}

\begin{proof}
Choose adapted local coordinates and let
\(
v_{ij}^{\langle 2j\rangle}
\in
\Gamma\!\bigl(U_{ij},\mathcal T_{\widehat X}^{\langle 2j\rangle}\bigr)
\)
be the leading logarithmic transition cocycle, as in the proof of
Proposition~\ref{prop:filtat1}. By Lemma~\ref{lem:CechDolbeaultLeading}, this cocycle
represents the adapted leading class
\[
\beta_{X,j}\bigl(\sigma^{(2j-1)}\bigr)
\in
H^1\!\bigl(\Xred,\mathcal T_{\widehat X}^{\langle 2j\rangle}\bigr).
\]
The filtered affine Atiyah symbol is its normalized pure odd Hessian symbol:
\(
\mathfrak a_{ij}^{\langle 2j\rangle}
=
\operatorname{Hess}_{2j}\bigl(v_{ij}^{\langle 2j\rangle}\bigr).
\)
By the normalization of the symbol projection,
\[
\Pi_{2j}^{\mathrm{ev}}
\bigl(\mathfrak a_{ij}^{\langle 2j\rangle}\bigr)
=
\rho_{2j}\bigl(v_{ij}^{\langle 2j\rangle}\bigr).
\]
Passing to cohomology and using Theorem~\ref{thm:minimal-green-comparison}, or
equivalently Theorem~\ref{thm:sheafMC}, gives
\[
H^1(\Pi_{2j}^{\mathrm{ev}})
\Bigl(
\mathfrak{At}^{\langle 2j\rangle}
\bigl(X;\sigma^{(2j-1)}\bigr)
\Bigr)
=
H^1(\rho_{2j})
\bigl(\beta_{X,j}(\sigma^{(2j-1)})\bigr)
=
\omega_X^{(2j)}\bigl(\sigma^{(2j-1)}\bigr).
\]
This proves the first assertion.

Assume now that this even obstruction vanishes and choose a \(2j\)-splitting
\(\sigma^{(2j)}\). By Theorem~\ref{thm:sheafMC}, the adapted representative can be
normalized so that the leading logarithmic transition cocycle takes values in
\[
\ker(\rho_{2j})
=
\mathcal Q_X^{(2j+1)}
\subset
\mathcal T_{\widehat X}^{\langle 2j\rangle}.
\]
Equivalently, the even projected component of the Hessian-symbol cocycle vanishes, and
only the odd Hessian symbol remains. This gives a class
\[
\widetilde{\mathfrak{At}}^{\langle 2j\rangle}
\bigl(X;\sigma^{(2j)}\bigr)
\in
H^1\!\bigl(\Xred,
\mathcal H_{\widehat X,\mathrm{odd}}^{\langle 2j\rangle}\bigr).
\]
Again by the normalization of the symbol projection,
\[
\Pi_{2j}^{\mathrm{odd}}
\bigl(\operatorname{Hess}_{2j}(v_{ij}^{\langle 2j\rangle})\bigr)
=
v_{ij}^{\langle 2j\rangle}
\]
for the kernel-valued representative. The cohomology class of this kernel-valued
representative is precisely \(\omega_X^{(2j+1)}(\sigma^{(2j)})\) by
Theorem~\ref{thm:sheafMC}. This proves the second assertion.
\end{proof}

\begin{definition}[Atiyah defects and residual classes]
\label{def:defect_rescl}
Let \(1\leq j\leq N\), and fix a \((2j-1)\)-splitting
\[
\sigma^{(2j-1)}:
\widehat X^{(2j-1)}\xrightarrow{\sim}X^{(2j-1)}.
\]
The \emph{Atiyah defect in weight \(2j\)} is the class
\[
\Def_{2j}\bigl(X;\sigma^{(2j-1)}\bigr)
\defeq
H^1(\Pi_{2j}^{\mathrm{ev}})
\Bigl(
\mathfrak{At}^{\langle 2j\rangle}
\bigl(X;\sigma^{(2j-1)}\bigr)
\Bigr)
\in
H^1\!\bigl(\Xred,\mathcal Q_X^{(2j)}\bigr).
\]
If this class vanishes and a \(2j\)-splitting \(\sigma^{(2j)}\) has been fixed, the
\emph{Atiyah residual class in weight \(2j+1\)} is
\[
\Res_{2j+1}\bigl(X;\sigma^{(2j)}\bigr)
\defeq
H^1(\Pi_{2j}^{\mathrm{odd}})
\Bigl(
\widetilde{\mathfrak{At}}^{\langle 2j\rangle}
\bigl(X;\sigma^{(2j)}\bigr)
\Bigr)
\in
H^1\!\bigl(\Xred,\mathcal Q_X^{(2j+1)}\bigr),
\]
whenever \(2j+1\leq \rk(\FF_X)\). If \(2j+1>\rk(\FF_X)\), the residual class is
understood to be zero.
\end{definition}

\begin{corollary}
With the notation above,
\[
\Def_{2j}\bigl(X;\sigma^{(2j-1)}\bigr)
=
\omega_X^{(2j)}\bigl(\sigma^{(2j-1)}\bigr),
\]
and, after choosing a \(2j\)-splitting and assuming \(2j+1\leq \rk(\FF_X)\),
\[
\Res_{2j+1}\bigl(X;\sigma^{(2j)}\bigr)
=
\omega_X^{(2j+1)}\bigl(\sigma^{(2j)}\bigr).
\]
\end{corollary}

\begin{remark}
The affine Atiyah class contains the splitting tower in a filtered and conditional sense.
The primary obstruction is an actual component of the restricted affine Atiyah class.
Higher obstructions are obtained only after choosing the previous partial splitting,
passing to the first non-split pure odd Hessian symbol, and applying the natural even or
odd symbol projection. Thus the higher classes are not independent direct summands of
\(\at_X\); they are successive projected defects and residual classes extracted from the
same filtered affine Atiyah cocycle.
\end{remark}

The present section has therefore supplied the Atiyah-theoretic realization of the higher splitting
data: after lower-order splittings are fixed, the filtered affine Atiyah symbols recover
the successive Green obstruction classes through their projected defects and normalized
residual classes. This completes the comparison between the sheaf-theoretic obstruction
tower, the filtered Maurer--Cartan formulation, the minimal \(L_\infty\)-model, and the
affine Atiyah class.


\section{\v{C}ech coordinates, Atiyah symbols, and Kuranishi compatibilities}
\label{sec:cech-kuranishi-at-symbols}

Sections~4--6 describe the splitting problem in three complementary languages. First,
Section~4 identifies the Green tower with the filtration of the dg Lie algebra
\(
L_X.
\) 
Second, Section~5 transfers this filtered dg Lie algebra to a minimal filtered
\(L_\infty\)-model on cohomology. Third, Section~6 explains how the affine Atiyah
class sees the same obstruction data through its filtered pure odd Hessian symbols.

The purpose of the present section is to make the comparison concrete on a cover. There
is a small technical point that must be handled carefully, though: instead of the raw \v{C}ech cochain complex
of a sheaf (of Lie algebras), we will use the Thom--Whitney totalization of the semicosimplicial dg Lie algebra associated with a
Cartan--Stein cover\footnote{The reason for using the Thom--Whitney model is that the ordinary \v{C}ech
complex is the correct complex for cohomology but not, in general, the most natural
strict dg Lie object. The Alexander--Whitney cup product on \v{C}ech cochains is
not graded-commutative; consequently, combining it with the pointwise Lie bracket
of a sheaf of Lie algebras does not yield a functorial strict dg Lie structure in the
required sense. The Thom--Whitney totalization is quasi-isomorphic to the ordinary
\v{C}ech complex but carries a strict dg Lie structure, and hence is the appropriate
model for Maurer--Cartan and gauge-theoretic constructions.}. For background material about the 
Thom--Whitney functor, we refer the reader to the classical \cite{NavarroAznarHodgeDeligne}. Its use in deformation theory via semicosimplicial
dg Lie algebras is standard; see, for example,
Fiorenza--Iacono--Martinengo \cite{FiorenzaIaconoMartinengoDGLA}. The descent point of view
goes back, in this context, to Hinich--Schechtman
\cite{HinichSchechtmanDescent}.

The outcome is the following picture. For a chosen partial splitting
\(
\sigma^{(2j-1)}:
\widehat X^{(2j-1)}\xrightarrow{\sim}X^{(2j-1)},
\)
the adapted leading class
\(
\beta_{X,j}\bigl(\sigma^{(2j-1)}\bigr)
\in
H^1\!\bigl(\Xred,\mathcal T_{\widehat X}^{\langle 2j\rangle}\bigr)
\)
has ordinary \v{C}ech representatives given by the first non-split logarithmic transition
terms. The normalized pure odd Hessian symbol of this class is the filtered affine Atiyah
class
\(
\mathfrak{At}^{\langle 2j\rangle}
\bigl(X;\sigma^{(2j-1)}\bigr)
\in
H^1\!\bigl(\Xred,\mathcal H_{\widehat X}^{\langle 2j\rangle}\bigr).
\)
Its projected even symbol is the Green obstruction
\(\omega_X^{(2j)}(\sigma^{(2j-1)})\), and, after a further normalization determined by a
\(2j\)-splitting, its residual odd symbol is
\(\omega_X^{(2j+1)}(\sigma^{(2j)})\). The Kuranishi equations of the minimal
\(L_\infty\)-model are the intrinsic compatibility equations among these degree-one data.

\smallskip

The constructions developed in Sections~3--6 may be organized into the
diagram in Figure \ref{fig:master-diagram}, which summarizes the relation between the splitting tower, the filtered
Maurer--Cartan model, the minimal filtered \(L_\infty\)-model, and the filtered affine
Atiyah symbols.

\begin{figure}[h]
\centering
\resizebox{0.94\linewidth}{!}{%
\begin{tikzpicture}[
  >=Latex,
  node distance=9mm and 16mm,
  box/.style={
    draw,
    rounded corners,
    align=center,
    inner sep=3.5pt,
    font=\small,
    fill=gray!6,
    text width=3.15cm
  },
  ibox/.style={
    draw,
    rounded corners,
    align=center,
    inner sep=3.5pt,
    font=\small,
    fill=blue!4,
    text width=2.95cm
  },
  arrow/.style={->, thick},
  darrow/.style={->, thick, dashed},
  lab/.style={
    font=\scriptsize,
    fill=white,
    inner sep=1.2pt,
    align=center
  }
]

\node[ibox] (X) {$X$\\complex supermanifold};

\node[ibox, right=24mm of X] (split)
{split model\\[1mm]
$\widehat X=(\Xred,\wedge^\bullet \FF_X)$};

\node[box, below=10mm of X] (tower)
{partial splitting tower\\[1mm]
$\sigma^{(1)} \leadsto \sigma^{(2)} \leadsto \cdots \leadsto \sigma^{(n)}$};

\node[box, below=10mm of split] (gauge)
{filtered gauge class\\[1mm]
$[\mu_X]\in MC(L_X)/\mathrm{gauge}$};

\node[box, below=10mm of gauge] (adapt)
{adapted Maurer--Cartan representative\\[1mm]
$\mu^{(j)}\in MC(L_X)\cap F^jL_X^1$};

\node[box, right=24mm of adapt] (minmodel)
{minimal filtered \(L_\infty\)-model\\[1mm]
$(H_X,\{\ell_n\}_{n\ge 2})$};

\node[box, below=10mm of adapt] (beta)
{leading cohomology class\\[1mm]
$\beta_{X,j}(\sigma^{(2j-1)})$};

\node[box, right=24mm of beta] (mincoord)
{minimal Maurer--Cartan coordinates\\[1mm]
$\alpha_X^{\min}$};

\node[box, below=10mm of beta] (atiyah)
{filtered affine Atiyah symbol\\[1mm]
$\mathfrak{At}^{\langle 2j\rangle}(X;\sigma^{(2j-1)})$};

\node[box, right=18mm of atiyah] (even)
{even Green obstruction\\[1mm]
$\omega_X^{(2j)}(\sigma^{(2j-1)})$};

\node[box, below=10mm of atiyah] (res)
{normalized residual symbol\\[1mm]
$\widetilde{\mathfrak{At}}^{\langle 2j\rangle}(X;\sigma^{(2j)})$};

\node[box, right=18mm of res] (odd)
{odd Green obstruction\\[1mm]
$\omega_X^{(2j+1)}(\sigma^{(2j)})$};

\draw[arrow] (X) -- (split)
  node[midway, above=2pt, lab] {associated object};

\draw[arrow] (X) -- (tower)
  node[midway, left=4pt, lab] {splitting\\data};

\draw[arrow] (split) -- (gauge)
  node[midway, right=4pt, lab] {choose a\\\(C^\infty\)-splitting};

\draw[arrow] (tower.east) -- (adapt.west)
  node[midway, above=2pt, lab] {choose \(\sigma^{(2j-1)}\)};

\draw[arrow] (gauge) -- (adapt)
  node[midway, left=4pt, lab] {represent};

\draw[darrow] (gauge.east) -- (minmodel.west)
  node[midway, above=2pt, lab] {homotopy transfer};

\draw[darrow] (adapt.east) -- (mincoord.west)
  node[midway, above=2pt, lab] {compatible choice};

\draw[arrow] (adapt) -- (beta)
  node[midway, left=4pt, lab] {leading term\\and class};

\draw[darrow] (minmodel) -- (mincoord)
  node[midway, right=4pt, lab] {Kuranishi\\slice};

\draw[arrow] (beta) -- (atiyah)
  node[midway, left=4pt, lab] {\(H^1(\operatorname{Hess}_{2j})\)};

\draw[arrow] (atiyah.east) -- (even.west)
  node[midway, above=2pt, lab] {\(H^1(\Pi_{2j}^{\mathrm{ev}})\)};

\draw[darrow] (atiyah) -- (res)
  node[midway, left=4pt, lab] {normalization};

\draw[arrow] (res.east) -- (odd.west)
  node[midway, above=2pt, lab] {\(H^1(\Pi_{2j}^{\mathrm{odd}})\)};

\draw[darrow] (even) -- (odd)
  node[midway, right=4pt, lab] {if \(\omega_X^{(2j)}=0\)\\choose \(\sigma^{(2j)}\)};

\end{tikzpicture}%
}
\caption{Schematic organization of the constructions developed in Sections~3--6.
The diagram relates the geometric splitting tower, the filtered Maurer--Cartan model,
the minimal filtered \(L_\infty\)-model, and the filtered affine Atiyah-symbol
formalism. Solid arrows denote constructions or projections; dashed arrows denote
choices, comparison procedures, or model-dependent realizations.}
\label{fig:master-diagram}
\end{figure}
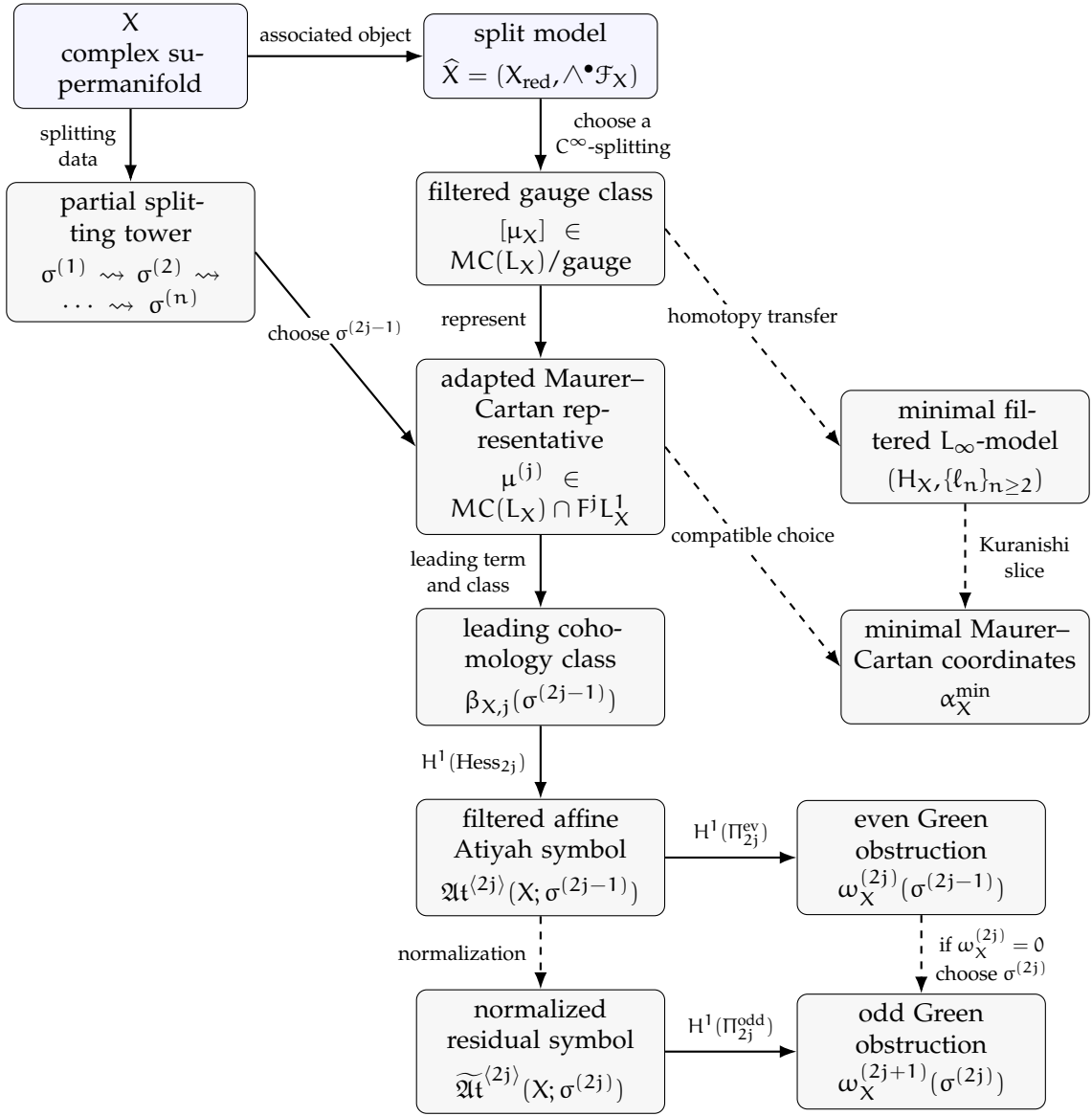

\subsection{The Thom--Whitney \v{C}ech model}

Let
\(
\mathfrak U=\{U_i\}_{i\in I}
\)
be a sufficiently fine Cartan--Stein cover of \(\Xred\), so every finite intersection
\(
U_{i_0\cdots i_p}\defeq U_{i_0}\cap\cdots\cap U_{i_p}
\)
is Stein. For every \(p\geq 0\), define
\[
\mathfrak L_X^p(\mathfrak U)
\defeq
\prod_{i_0<\cdots<i_p}
A^{0,\bullet}\!\bigl(
U_{i_0\cdots i_p},
\mathcal T_{\widehat X,\bar 0}^{[\geq 2]}
\bigr).
\]
The coface maps are the restrictions obtained by omitting one index; their alternating
sum gives the usual \v{C}ech differential after totalization. The Dolbeault differential and
the pointwise commutator bracket of derivations make
\(
[p]\mapsto \mathfrak L_X^p(\mathfrak U)
\)
a semicosimplicial filtered dg Lie algebra. The filtration is
\[
F^j\mathfrak L_X^p(\mathfrak U)
\defeq
\prod_{i_0<\cdots<i_p}
A^{0,\bullet}\!\bigl(
U_{i_0\cdots i_p},
\mathcal T_{\widehat X,\bar 0}^{[\geq 2j]}
\bigr).
\]

\begin{definition}[Thom--Whitney \v{C}ech controller]
\label{def:TW-cech-controller}
The \emph{Thom--Whitney \v{C}ech controller} of the splitting problem on
\(\mathfrak U\) is
\[
\check L_X^{\mathrm{TW}}(\mathfrak U)
\defeq
\operatorname{Tot}_{\mathrm{TW}}\bigl(\mathfrak L_X^\bullet(\mathfrak U)\bigr).
\]
It is endowed with the Thom--Whitney differential
\(
d_{\mathrm{TW}}
\)
and with the induced dg Lie bracket, denoted
\([\,\cdot,\cdot\,]_{\mathrm{TW}}\). Its filtration is the totalization of the
filtrations \(F^j\mathfrak L_X^p(\mathfrak U)\).
\end{definition}

\begin{proposition}\label{prop:TW-dgla-model}
The complex \(\check L_X^{\mathrm{TW}}(\mathfrak U)\) is a finite-step filtered dg Lie
algebra, and its filtration satisfies
\[
\bigl[
F^i\check L_X^{\mathrm{TW}}(\mathfrak U),
F^j\check L_X^{\mathrm{TW}}(\mathfrak U)
\bigr]_{\mathrm{TW}}
\subset
F^{i+j}\check L_X^{\mathrm{TW}}(\mathfrak U).
\]
Moreover,
\[
F^{N+1}\check L_X^{\mathrm{TW}}(\mathfrak U)=0,
\qquad
N\defeq \left\lfloor\frac{\rk(\FF_X)}2\right\rfloor .
\]
\end{proposition}

\begin{proof}
The Thom--Whitney totalization of a semicosimplicial dg Lie algebra is again a dg Lie
algebra. The filtration is finite because the odd rank of \(X\) is finite. The bracket
compatibility follows from the corresponding statement on the sheaf of homogeneous
derivations, that is
\[
\bigl[
\mathcal T_{\widehat X,\bar 0}^{[\geq 2i]},
\mathcal T_{\widehat X,\bar 0}^{[\geq 2j]}
\bigr]
\subset
\mathcal T_{\widehat X,\bar 0}^{[\geq 2(i+j)]}.
\]
Since there are no even tangent weights greater than \(2N\), the filtration vanishes at
\(F^{N+1}\).
\end{proof}

\begin{proposition}[Comparison with the Dolbeault controller]
\label{prop:TW-quasi-iso-LX}
There is a natural filtered quasi-isomorphism of dg Lie algebras
\[
\check L_X^{\mathrm{TW}}(\mathfrak U)
\simeq
L_X.
\]
Consequently \(\check L_X^{\mathrm{TW}}(\mathfrak U)\) and \(L_X\) govern equivalent
pointed filtered formal moduli problems.
\end{proposition}

\begin{proof}
For each homogeneous weight \(2j\), the sheaf
\(
\mathcal A_{\Xred}^{0,\bullet}\otimes
\mathcal T_{\widehat X,\bar 0}^{\langle 2j\rangle}
\)
is a fine Dolbeault resolution of the holomorphic vector bundle
\(\mathcal T_{\widehat X,\bar 0}^{\langle 2j\rangle}\). The Thom--Whitney totalization of
the \v{C}ech semicosimplicial complex of a fine resolution computes the same derived
global sections as the global Dolbeault complex. Hence, for each \(j\), we have a
quasi-isomorphism
\[
\operatorname{Tot}_{\mathrm{TW}}
\left(
[p]\longmapsto
\prod_{i_0<\cdots<i_p}
A^{0,\bullet}\!\bigl(U_{i_0\cdots i_p},
\mathcal T_{\widehat X,\bar 0}^{\langle 2j\rangle}\bigr)
\right)
\simeq
A^{0,\bullet}\!\bigl(\Xred,
\mathcal T_{\widehat X,\bar 0}^{\langle 2j\rangle}\bigr).
\]
Summing over the finitely many even weights gives a filtered quasi-isomorphism to
\(L_X\). Compatibility with the Lie bracket follows because both brackets are induced by
the same sheaf-level commutator bracket of derivations. The last assertion follows from
invariance of Maurer--Cartan \(\infty\)-groupoids under quasi-isomorphisms of nilpotent
dg Lie algebras.
\end{proof}

\begin{remark}[Ordinary \v{C}ech representatives]
\label{rem:ordinary-cech-reps}
Although the strict dg Lie algebra used below is the Thom--Whitney totalization, ordinary
holomorphic \v{C}ech cocycles still represent its cohomology. In particular, for every
weight \(2j\), the usual \v{C}ech--Dolbeault comparison identifies
\(
H^1\!\bigl(\Xred,
\mathcal T_{\widehat X}^{\langle 2j\rangle}\bigr)
\)
with Dolbeault and Thom--Whitney cohomology in degree one. Thus the cocycles appearing in
Lemma~\ref{lem:CechDolbeaultLeading} may be regarded as ordinary \v{C}ech representatives
of classes in the Thom--Whitney model.
\end{remark}

\subsection{Minimal coordinates in the \v{C}ech model}

By Proposition~\ref{prop:TW-quasi-iso-LX}, the Thom--Whitney controller has the same
cohomology as \(L_X\). We continue to denote this cohomology by
\[
H_X
=
H^\bullet(L_X)
\cong
H^\bullet\!\bigl(\check L_X^{\mathrm{TW}}(\mathfrak U)\bigr).
\]
It carries the weight decomposition and filtration described in Section~5:
\[
H_X
=
\bigoplus_{j=1}^N H_X^{\langle 2j\rangle},
\qquad
F^jH_X=
\bigoplus_{r\geq j}H_X^{\langle 2r\rangle}.
\]

\begin{proposition}[Filtered contraction in the \v{C}ech model]
\label{prop:cech-contraction}
There exist filtered linear maps
\[
i:H_X\longrightarrow \check L_X^{\mathrm{TW}}(\mathfrak U),
\qquad
p:\check L_X^{\mathrm{TW}}(\mathfrak U)\longrightarrow H_X,
\qquad
h:\check L_X^{\mathrm{TW}}(\mathfrak U)
\longrightarrow
\check L_X^{\mathrm{TW}}(\mathfrak U)[-1],
\]
satisfying
\[
p\circ i=\operatorname{id}_{H_X},
\qquad
\operatorname{id}-i\circ p=d_{\mathrm{TW}}h+h d_{\mathrm{TW}},
\qquad
ph=0,
\qquad
hi=0,
\qquad
h^2=0.
\]
\end{proposition}

\begin{proof}
The proof is the same as the proof of Proposition~\ref{prop:contr}. The Thom--Whitney
controller decomposes into finitely many homogeneous weight complexes, each of which is a
complex of vector spaces. Choose complements to boundaries and cycles weight by weight.
The resulting contraction preserves homogeneous weights and hence the filtration.
\end{proof}

A choice of contraction as in Proposition~\ref{prop:cech-contraction} transfers the dg Lie
structure to a minimal filtered \(L_\infty\)-structure
\(
\bigl(H_X,\{\ell_r\}_{r\geq 2}\bigr),
\)
together with a filtered \(L_\infty\)-quasi-isomorphism
\[
I:
\bigl(H_X,\{\ell_r\}_{r\geq 2}\bigr)
\rightsquigarrow
\check L_X^{\mathrm{TW}}(\mathfrak U).
\]
As before, this minimal model depends on the contraction, but its filtered
\(L_\infty\)-quasi-isomorphism type is independent of that choice.

Let
\(
[\mu_X]\in \pi_0MC_\infty(L_X)
\)
be the filtered Maurer--Cartan class determined by the holomorphic supermanifold \(X\),
after the choice of a smooth splitting as in Section~4. Via the equivalence
\[
MC_\infty\!\bigl(H_X\bigr)
\simeq
MC_\infty\!\bigl(L_X\bigr),
\]
this class is represented, after choosing the above contraction and Kuranishi slice, by a
minimal Maurer--Cartan element
\[
\alpha_X^{\min}
=
\alpha_{X,1}^{\min}+\cdots+\alpha_{X,N}^{\min}
\in H_X^1,
\qquad
\alpha_{X,j}^{\min}\in H_X^{1,\langle 2j\rangle}.
\]
This representative is not canonical. Changing the contraction or the Kuranishi slice
changes \(\alpha_X^{\min}\) by a filtered \(L_\infty\)-coordinate transformation. The
intrinsic object is the filtered Maurer--Cartan class \([\mu_X]\).

\subsection{Kuranishi lifts and \v{C}ech curvature cocycles}

Let
\(
\alpha=
\alpha_1+\cdots+\alpha_N\in H_X^1,
\) with \(
\alpha_j\in H_X^{1,\langle 2j\rangle}.
\)
The Kuranishi lift in the Thom--Whitney model is the unique solution
\[
\tau(\alpha)
=
i(\alpha)-\frac12 h[\tau(\alpha),\tau(\alpha)]_{\mathrm{TW}}
\]
in \(\check L_X^{\mathrm{TW}}(\mathfrak U)^1\). Writing
\[
\tau(\alpha)=\tau_1+\cdots+\tau_N,
\qquad
\tau_j\in
\check L_X^{\mathrm{TW}}(\mathfrak U)^{1,\langle 2j\rangle},
\]
one obtains the recursive formula
\begin{equation}\label{eq:cech-kuranishi-recursion}
\tau_j
=
i(\alpha_j)
-
\frac12
\sum_{a+b=j}h[\tau_a,\tau_b]_{\mathrm{TW}}.
\end{equation}
Thus \(\tau_j\) is the weight-\(2j\) cochain obtained from a representative of
\(\alpha_j\) after subtracting all nonlinear contributions forced by lower weights.

Define the Kuranishi curvature
\begin{equation}\label{eq:cech-kuranishi-curvature}
F(\alpha)
\defeq
d_{\mathrm{TW}}\tau(\alpha)
+
\frac12[\tau(\alpha),\tau(\alpha)]_{\mathrm{TW}}
\in
\check L_X^{\mathrm{TW}}(\mathfrak U)^2.
\end{equation}
Its weight-\(2j\) component is denoted
\[
F_j(\alpha_1,
\dots,
\alpha_{j-1})
\defeq
\operatorname{pr}^{\langle 2j\rangle}F(\alpha),
\qquad
2\leq j\leq N.
\]
Its cohomology class is the \(j\)-th Kuranishi obstruction
\[
\kappa_j(\alpha_1,
\dots,
\alpha_{j-1})
\in
H_X^{2,\langle 2j\rangle}.
\]
The minimal Maurer--Cartan equation is the finite system
\begin{equation}\label{eq:cech-minimal-equations}
\kappa_j(\alpha_1,
\dots,
\alpha_{j-1})=0,
\qquad
2\leq j\leq N.
\end{equation}

\begin{remark}[Ordinary \v{C}ech form of the first curvature cocycles]
\label{rem:cech-curvature-formulas}
Choose ordinary holomorphic \v{C}ech representatives
\(
a_j\in
\check Z^1\!\bigl(\mathfrak U,
\mathcal T_{\widehat X}^{\langle 2j\rangle}\bigr)
\)
for the classes \(\alpha_j\), and choose the contraction so that the linear term
\(i(\alpha_j)\) is represented by \(a_j\). In ordinary \v{C}ech representatives, suppressing
the standard comparison maps between the Thom--Whitney and \v{C}ech models, the first
curvature cocycles have the expected form
\[
F_2
=
\frac12[a_1,a_1],
\]
\[
F_3
=
[a_1,a_2]
-
\frac12\,\check\delta h[a_1,a_1],
\]
and higher \(F_j\)'s are obtained from iterated cup--Lie brackets of the lower
representatives, with the homotopy operator \(h\) inserted along internal edges of rooted
trees. These formulas should be understood as \v{C}ech representatives of the exact
Thom--Whitney Kuranishi operations above. They are useful for computations, whereas the
Thom--Whitney formulation supplies the strict dg Lie model.
\end{remark}

\subsection{Adapted leading classes and affine Atiyah symbols}

We now relate the \v{C}ech representatives of the previous subsection to the actual
supermanifold \(X\). Let \(1\leq j\leq N\), and suppose that a \((2j-1)\)-splitting
\(
\sigma^{(2j-1)}:
\widehat X^{(2j-1)}\xrightarrow{\sim}X^{(2j-1)}
\)
has been chosen. By Definition~\ref{def:adapted-leading-class}, this splitting determines
the adapted leading class
\(
\beta_{X,j}\bigl(\sigma^{(2j-1)}\bigr)
\in
H^1\!\bigl(\Xred,
\mathcal T_{\widehat X}^{\langle 2j\rangle}\bigr).
\)
By Lemma~\ref{lem:CechDolbeaultLeading}, on the cover \(\mathfrak U\) this class is
represented by the holomorphic \v{C}ech cocycle
\[
v_{ab}^{\langle 2j\rangle}
\in
\Gamma\!\bigl(U_{ab},
\mathcal T_{\widehat X}^{\langle 2j\rangle}\bigr)
\]
obtained from the first non-split logarithmic transition terms in coordinates adapted to
\(\sigma^{(2j-1)}\).

The affine Atiyah formalism of Section~6 assigns to the same partial splitting the
filtered affine Atiyah class
\[
\mathfrak{At}^{\langle 2j\rangle}
\bigl(X;\sigma^{(2j-1)}\bigr)
\in
H^1\!\bigl(\Xred,
\mathcal H_{\widehat X}^{\langle 2j\rangle}\bigr).
\]
It is the pure odd Hessian-symbol class of the same leading
logarithmic transition cocycle. In \v{C}ech representatives, this means that the cocycle
\begin{equation} \label{eq:hessian-beta-at}
\mathfrak a_{ab}^{\langle 2j\rangle}
=
\operatorname{Hess}_{2j}\bigl(v_{ab}^{\langle 2j\rangle}\bigr)
\end{equation}
represents
\( 
\mathfrak{At}^{\langle 2j\rangle}
\bigl(X;\sigma^{(2j-1)}\bigr).
\)
Here \(\operatorname{Hess}_{2j}\) denotes the normalized pure odd Hessian symbol from
Section~6, applied to the adapted logarithmic transition representative.

The Green obstruction classes are then recovered by the symbol projections:
\begin{equation}\label{eq:at-symbol-even-obstruction}
H^1(\Pi_{2j}^{\mathrm{ev}})
\left(
\mathfrak{At}^{\langle 2j\rangle}
\bigl(X;\sigma^{(2j-1)}\bigr)
\right)
=
\omega_X^{(2j)}\bigl(\sigma^{(2j-1)}\bigr).
\end{equation}
If this class vanishes and a \(2j\)-splitting
\(
\sigma^{(2j)}:
\widehat X^{(2j)}\xrightarrow{\sim}X^{(2j)}
\)
has been fixed, then the adapted representative may be normalized so that the even
projected component vanishes. The remaining odd Hessian-symbol class is
\[
\widetilde{\mathfrak{At}}^{\langle 2j\rangle}
\bigl(X;\sigma^{(2j)}\bigr)
\in
H^1\!\bigl(\Xred,
\mathcal H_{\widehat X,\mathrm{odd}}^{\langle 2j\rangle}\bigr),
\]
and, if \(2j+1\leq \rk(\FF_X)\),
\begin{equation}\label{eq:at-symbol-odd-obstruction}
H^1(\Pi_{2j}^{\mathrm{odd}})
\left(
\widetilde{\mathfrak{At}}^{\langle 2j\rangle}
\bigl(X;\sigma^{(2j)}\bigr)
\right)
=
\omega_X^{(2j+1)}\bigl(\sigma^{(2j)}\bigr).
\end{equation}
If \(2j+1>\rk(\FF_X)\), the odd obstruction sheaf is zero.

\begin{remark}
The representative relation \eqref{eq:hessian-beta-at}, together with
\eqref{eq:at-symbol-even-obstruction} and \eqref{eq:at-symbol-odd-obstruction}, is the
precise form of the slogan that the affine
Atiyah class contains the Green tower. 
After a
lower-order splitting has been chosen, the first non-split logarithmic transition cocycle
has a tangent-valued leading class \(\beta_{X,j}\). 
The affine Atiyah cocycle records the
pure odd Hessian symbol of that class. The Green obstructions are obtained only after
applying \(\Pi_{2j}^{\mathrm{ev}}\) and, when appropriate, \(\Pi_{2j}^{\mathrm{odd}}\).
\end{remark}

\subsection{The combined comparison theorem}

The preceding constructions can be summarized in one stagewise statement. The statement is
stagewise because the classes in weight \(2j\) are defined relative to a chosen
\((2j-1)\)-splitting. If such a partial splitting does not exist, the corresponding adapted
leading class and filtered affine Atiyah symbol are not defined at that stage.

Before stating the comparison theorem, we summarize the objects that enter the
construction. For \(1\le j\le N\), and after fixing the relevant lower-order splitting
data, the notation is as follows:

\smallskip

\[
\begin{array}{c|c|c}
\text{object} & \text{target} & \text{meaning} \\
\hline
\eta_{2j} &
A^{0,1}\bigl(\Xred,\mathcal T_{\widehat X}^{\langle 2j\rangle}\bigr) &
\text{leading term of an adapted Maurer--Cartan representative}
\\[3pt]
\beta_{X,j}(\sigma^{(2j-1)}) &
H^1\bigl(\Xred,\mathcal T_{\widehat X}^{\langle 2j\rangle}\bigr) &
\text{adapted leading cohomology class}
\\[3pt]
\mathfrak{At}^{\langle 2j\rangle}(X;\sigma^{(2j-1)}) &
H^1\bigl(\Xred,\mathcal H_{\widehat X}^{\langle 2j\rangle}\bigr) &
\text{pure odd Hessian symbol of the affine Atiyah class}
\\[3pt]
\omega_X^{(2j)}(\sigma^{(2j-1)}) &
H^1\bigl(\Xred,\mathcal Q_X^{(2j)}\bigr) &
\text{even Green obstruction}
\\[3pt]
\omega_X^{(2j+1)}(\sigma^{(2j)}) &
H^1\bigl(\Xred,\mathcal Q_X^{(2j+1)}\bigr) &
\text{residual odd Green obstruction}
\end{array}
\]

\smallskip


\begin{theorem}[\v{C}ech coordinates, Atiyah symbols, and Kuranishi compatibilities]
\label{thm:cech-at-symbol-kuranishi}
Let \(X\) be a complex supermanifold, and let
\(
N=\left\lfloor {\rk(\FF_X)}/2\right\rfloor .
\)
Fix a Cartan--Stein cover \(\mathfrak U\), a Thom--Whitney \v{C}ech controller
\(\check L_X^{\mathrm{TW}}(\mathfrak U)\), and a filtered contraction of this controller
onto \(H_X\). Then the following statements hold.

\begin{enumerate}
\item The filtered Maurer--Cartan class of \(X\),
\(
[\mu_X]\in \pi_0MC_\infty(L_X),
\)
is represented in the chosen minimal model by a minimal Maurer--Cartan element
\[
\alpha_X^{\min}
=
\alpha_{X,1}^{\min}+\cdots+\alpha_{X,N}^{\min}\in H_X^1.
\]
This representative depends on the chosen contraction and Kuranishi slice, but its
filtered \(L_\infty\)-equivalence class represents the intrinsic gauge class \([\mu_X]\).

\item For every \(1\leq j\leq N\) for which a \((2j-1)\)-splitting
\(\sigma^{(2j-1)}\) has been chosen, the adapted leading class
\(
\beta_{X,j}\bigl(\sigma^{(2j-1)}\bigr)
\in
H_X^{1,\langle 2j\rangle}
\)
is represented by the first non-split logarithmic \v{C}ech transition terms in coordinates
adapted to \(\sigma^{(2j-1)}\). \\
If the Kuranishi slice is chosen compatibly with these
representatives, then the weight-\(2j\) minimal coordinate is represented by this class.
For a different slice, it is transformed by a filtered \(L_\infty\)-coordinate change.

\item With the same choice of \(j\) and \(\sigma^{(2j-1)}\), the filtered affine Atiyah
symbol is represented by applying the normalized pure odd Hessian symbol to an adapted
\v{C}ech representative of the leading class
\(
\beta_{X,j}\bigl(\sigma^{(2j-1)}\bigr).
\)
Equivalently, if \(\{v_{ab}^{\langle 2j\rangle}\}\) represents this leading class, then
\(\{\operatorname{Hess}_{2j}(v_{ab}^{\langle 2j\rangle})\}\) represents
\(
\mathfrak{At}^{\langle 2j\rangle}
\bigl(X;\sigma^{(2j-1)}\bigr).
\)

\item With the same choice of \(j\) and \(\sigma^{(2j-1)}\), the even projected
symbol is the Green obstruction to lifting \(\sigma^{(2j-1)}\) to order \(2j\):
\[
H^1(\Pi_{2j}^{\mathrm{ev}})
\left(
\mathfrak{At}^{\langle 2j\rangle}
\bigl(X;\sigma^{(2j-1)}\bigr)
\right)
=
\omega_X^{(2j)}\bigl(\sigma^{(2j-1)}\bigr).
\]
If this obstruction vanishes and a \(2j\)-splitting \(\sigma^{(2j)}\) is fixed, then the
normalized residual odd symbol satisfies
\[
H^1(\Pi_{2j}^{\mathrm{odd}})
\left(
\widetilde{\mathfrak{At}}^{\langle 2j\rangle}
\bigl(X;\sigma^{(2j)}\bigr)
\right)
=
\omega_X^{(2j+1)}\bigl(\sigma^{(2j)}\bigr)
\]
whenever \(2j+1\leq \rk(\FF_X)\).

\item The Kuranishi equations
\[
\kappa_j(\alpha_1,
\dots,
\alpha_{j-1})=0,
\qquad
2\leq j\leq N,
\]
are the intrinsic \(H_X^2\)-valued compatibility equations for any system of degree-one
coordinates representing an actual Maurer--Cartan class. In \v{C}ech representatives they are
built from iterated brackets of the lower representatives, corrected by the homotopy
operator of the chosen contraction.
\end{enumerate}
\end{theorem}

\begin{proof}
The first assertion is the equivalence of Maurer--Cartan \(\infty\)-groupoids induced by the
filtered \(L_\infty\)-quasi-isomorphism from the minimal model to the Thom--Whitney
controller, together with Proposition~\ref{prop:TW-quasi-iso-LX}.

The second assertion follows from Lemma~\ref{lem:CechDolbeaultLeading} and
Definition~\ref{def:adapted-leading-class}: the ordinary \v{C}ech cocycle of leading
logarithmic transition terms represents the adapted leading class. Compatibility with a
chosen Kuranishi slice is simply the choice of cohomology representatives in the linear term
of the contraction.

The third assertion is Proposition~\ref{prop:filtat1}, rewritten in cohomological form. The
fourth assertion is Theorem~\ref{thm:filteredAtiyahObstructions}, equivalently the
definitions of the Atiyah defects and residual classes in Definition~\ref{def:defect_rescl}.

Finally, the fifth assertion is the weight decomposition of the minimal Maurer--Cartan
equation from Section~5, written in the Thom--Whitney \v{C}ech model by
\eqref{eq:cech-kuranishi-curvature}. The ordinary \v{C}ech formulas of
Remark~\ref{rem:cech-curvature-formulas} describe the corresponding representatives.
\end{proof}

Thus the three types of data are related by the chain
\[
\beta_{X,j}\bigl(\sigma^{(2j-1)}\bigr)
\xrightarrow{\operatorname{Hess}_{2j}}
\mathfrak{At}^{\langle 2j\rangle}
\bigl(X;\sigma^{(2j-1)}\bigr)
\xrightarrow{\Pi_{2j}^{\mathrm{ev}}}
\omega_X^{(2j)}\bigl(\sigma^{(2j-1)}\bigr),
\]
and, after the even obstruction vanishes and a \(2j\)-splitting is chosen,
\[
\widetilde{\mathfrak{At}}^{\langle 2j\rangle}
\bigl(X;\sigma^{(2j)}\bigr)
\xrightarrow{\Pi_{2j}^{\mathrm{odd}}}
\omega_X^{(2j+1)}\bigl(\sigma^{(2j)}\bigr).
\]
The higher brackets of the minimal \(L_\infty\)-model are not additional Green obstruction
classes; they are the compatibility equations ensuring that the degree-one data arise from a
single filtered Maurer--Cartan class.

\subsection{Exotic atlases and intrinsic obstruction data}

The preceding discussion also clarifies the role of exotic atlases, see \cite{BettadapuraHigherObstructions, DonagiWittenNotProjected}. In the classical Green
construction, once one passes beyond the primary obstruction, the \v{C}ech cocycles obtained
from a chosen partial splitting tower depend on the chosen tower. They may therefore fail to
be invariants of the underlying supermanifold by themselves.

The intrinsic object attached to \(X\) is the filtered Maurer--Cartan gauge class
\(
[\mu_X]\in \pi_0MC_\infty(L_X).
\)
After choosing a filtered minimal model and a Kuranishi slice, this class is represented by a
minimal Maurer--Cartan element
\(
\alpha_X^{\min}\in H_X^1.
\)
The adapted \v{C}ech cocycles, filtered affine Atiyah symbols, defects, residual classes,
and Kuranishi equations are coordinate expressions of this intrinsic gauge class relative to
chosen lower-order splitting data and a chosen cohomological slice.

If \(X\) is split, then \([\mu_X]=0\). Nevertheless, one may choose nontrivial higher
transition data representing the trivial filtered gauge class; such data can produce nonzero
naive higher \v{C}ech cocycles. This is the mechanism behind the exotic-atlas phenomenon.
It does not mean that the intrinsic obstruction theory fails. It means that the atlas-level
higher cocycles are not final invariants until they are interpreted modulo filtered gauge and
through the compatible obstruction maps described above.

Conversely, when a Green obstruction class obtained from the adapted tower is nonzero in the
appropriate cohomology group, it detects a genuine failure to continue that tower. The
filtered Maurer--Cartan and Atiyah-symbol formalisms keep track of which parts of the
\v{C}ech data are coordinate choices and which parts survive as intrinsic obstruction classes.


\section{Green--Onishchik moduli and formal splitting theory}
\label{sec:green-onishchik-fixed-retract-moduli}

The preceding sections developed the splitting problem for a fixed complex supermanifold
\(X\) in several equivalent languages.  Starting from Green's classical obstruction
tower, we passed to a filtered Dolbeault Maurer--Cartan model, then to a minimal
filtered \(L_\infty\)-model, and finally to the affine Atiyah-symbol interpretation of
the same obstruction data.  In all these constructions the reduced space \(\Xred\) and
the odd conormal bundle \(\FF_X\) were fixed, and the central question was whether the
filtered structure sheaf \(\mathcal O_X\) could be identified, compatibly with its
augmentation and filtration, with the split exterior algebra
\(
\bigwedge^\bullet \FF_X .
\)
Thus Sections~3--7 gave an absolute deformation-theoretic description of one
supermanifold relative to its split model.

The purpose of the present section is to place that absolute picture in families.  We
fix a smooth proper holomorphic family
\(
\pi:Y\rightarrow S
\)
and a locally free \(\mathcal O_Y\)-module \(F\), regarded as the prescribed odd conormal
bundle. The split relative model is
\(
A_{\spl}\defeq \bigwedge_{\mathcal O_Y}^{\bullet}F .
\)
We study filtered relative superalgebras whose associated graded algebra is identified
with \(A_{\spl}\).  Equivalently, we pass from the splitting theory of one
supermanifold to a \emph{relative moduli problem of supermanifold structures with fixed
reduced family and fixed odd conormal bundle}.

There is a basic subtlety in formulating this fixed-retract problem.  The structure
sheaf of a non-projected supermanifold is not naturally an algebra over the structure
sheaf of its reduced space.  Requiring an \(\mathcal O_Y\)-algebra structure would
therefore amount to choosing a projection onto the reduced family, and would remove
exactly the phenomena measured by the Green obstruction tower.  For this reason, the
objects below are filtered \(\pi^{-1}\mathcal O_S\)-superalgebras, not filtered
\(\mathcal O_Y\)-superalgebras.  They are required to have reduced quotient
\(\mathcal O_Y\) and associated graded algebra \(\bigwedge^\bullet F\), but no splitting
\(\mathcal O_Y\to\mathcal A\) is part of the data.

The terminology ``Green--Onishchik'' reflects the two classical sources behind this
construction.  Green's work on holomorphic graded manifolds introduced the
non-abelian cohomological description of supermanifolds with fixed retract and the
associated obstruction sheaves \cite{Green1982}.  Onishchik developed the
classification of complex analytic supermanifolds with fixed retract by means of
non-abelian cochain complexes and constructed finite-dimensional moduli varieties in
the compact case, in analogy with Kuranishi theory
\cite{OnishchikClassification}.\footnote{Related foundational viewpoints are also
present in Palamodov's work on invariants of analytic \(\mathbb Z_2\)-manifolds
\cite{PalamodovInvariants} and in the standard treatments of complex and algebraic supergeometry by
Manin and others \cite{ManinGaugeField,VoronovManinPenkovElements}.}
The relative stack introduced below is a formal, filtered, and relative version of this
Green--Onishchik fixed-retract classification problem.

The infinitesimal symmetries of the split relative model are encoded by the sheaf of
even relative derivations
\(
\mathcal T_{A_{\spl}/S,\bar0}.
\)
The Green filtration selects the nilpotent Lie subsheaf
\[
\mathcal T^{[\geq 2]}_{A_{\spl}/S,\bar0}
\subset
\mathcal T_{A_{\spl}/S,\bar0},
\]
consisting of even derivations which raise the odd Green filtration by at least two.
Its Dolbeault resolution is the relative Green controller
\[
\mathfrak t_{F/S}
\defeq
\mathcal A_Y^{0,\bullet}
\bigl(
\mathcal T^{[\geq 2]}_{A_{\spl}/S,\bar0}
\bigr).
\]
For \(S=\mathrm{pt}\), this is exactly the filtered dg Lie algebra used in
Section~4.  Thus the construction below is not a new obstruction theory, but the
relative form of the same filtered Maurer--Cartan mechanism.

The comparison with the classical Green tower is already visible at the level of the
homogeneous pieces of the relative derivation sheaf.  For every \(j\ge1\), there is a
natural short exact sequence
\[
0
\longrightarrow
\mathcal Q^{(2j+1)}_{F/S}
\longrightarrow
\mathcal T^{\langle 2j\rangle}_{A_{\spl}/S,\bar 0}
\overset{\rho_{2j}}{\longrightarrow}
\mathcal Q^{(2j)}_{F/S}
\longrightarrow
0,
\]
where
\[
\mathcal Q^{(r)}_{F/S}
=
\begin{cases}
\mathcal T_{Y/S}\otimes \bigwedge^r F, & r \text{ even},\\[3pt]
\mathcal H\!om_{\mathcal O_Y}(F,\bigwedge^r F), & r \text{ odd}.
\end{cases}
\]
Consequently, one homogeneous Green coordinate of even weight \(2j\) contains two
classical pieces: the projected obstruction in \(\mathcal Q^{(2j)}_{F/S}\), and, after
this obstruction has vanished and a \(2j\)-splitting has been chosen, the residual
obstruction in \(\mathcal Q^{(2j+1)}_{F/S}\).  This is the relative counterpart of the
absolute comparison developed in Sections~4--7.

The formal part of the theory is intrinsic.  Fibrewise, the derived formal
neighbourhood of the split object is governed by
\(
R\Gamma\bigl(Y_T,\mathfrak t_{F_T/T}\bigr)
\)
and therefore defines a formal moduli problem in the sense of derived deformation
theory.  We use the language of formal moduli problems and dg Lie controllers in the
sense of Lurie and Pridham \cite{LurieDAGX,PridhamUDDT}, as well as the relative and
stack-theoretic formulations developed by Hennion, Calaque--Grivaux, and Nuiten
\cite{HennionTangentLie,CalaqueGrivauxFMP,Nuiten}.  The general background on derived
stacks and geometric stacks is provided by the framework of To\"en--Vezzosi
\cite{ToenVezzosiHAGII}, see also \cite{VezzosiDZL}.  When one passes from formal
moduli problems to algebraized derived moduli spaces, further finiteness and
representability hypotheses enter; this is the same philosophy underlying the
relationship between derived enhancements and obstruction theories in the sense of
Behrend--Fantechi and its derived-geometric refinements
\cite{BehrendFantechiIntrinsicNormalCone,SchurgToenVezzosiDerived}.

For this reason we separate the logical status of the results in this section into
three levels.

\begin{enumerate}
\item[\textup{(F)}] \textbf{Fibrewise formal theory.}  For every test space
\(T\to S\), the formal neighbourhood of the split object is governed by the finite-step
filtered dg Lie algebra
\(
R\Gamma\bigl(Y_T,\mathfrak t_{F_T/T}\bigr).
\)
This is the intrinsic formal statement.  It does not require a relative base-change
theorem.

\item[\textup{(B)}] \textbf{Relative base-change form.}  If the derived pushforward
\(
\mathfrak g_{F/S}
\defeq
R\pi_*\mathfrak t_{F/S}
\)
commutes with the base changes under consideration, the same fibrewise formal moduli
problem can be written compactly as
\(
\Def_S(\mathfrak g_{F/S}).
\)
This notation is useful, but it is conditional on an explicit derived base-change
assumption.

\item[\textup{(K)}] \textbf{Kuranishi and algebraization criteria.}  Passing from the
formal object to local or global derived Kuranishi charts requires additional
finite-model, amplitude, algebraizability, and descent hypotheses.  These hypotheses
are stated as sufficient criteria.  We do not claim that a canonical global algebraic
or analytic Green--Onishchik derived stack exists without them.
\end{enumerate}

\subsection{Fixed-retract objects and the Green controller}
\label{subsec:GO-moduli-controller}

We begin by defining the classical fixed-retract stack.  The data fixed once and for
all are the reduced family \(Y/S\) and the odd conormal bundle \(F\).  The split
model is the exterior algebra \(\wedge^\bullet F\), and the objects to be classified
are filtered superalgebras locally isomorphic to this split model and with associated
graded algebra globally identified with it.

\begin{definition}[The relative split model]
\label{def:relative-split-model}
Let \(\pi:Y\to S\) be a proper holomorphic submersion and let \(F\) be a holomorphic
vector bundle of rank \(n\) on \(Y\).  The associated split relative supermanifold is
\[
\widehat Y_{\spl}\defeq (Y,A_{\spl}),
\qquad
A_{\spl}\defeq \bigwedge_{\mathcal O_Y}^{\bullet}F .
\]
Its odd ideal is
\(
\mathcal J_{\spl}\defeq \bigoplus_{i\ge 1}\wedge^i F,
\)
and the Green filtration on \(A_{\spl}\) is
\[
F^rA_{\spl}\defeq \mathcal J_{\spl}^r
=
\bigoplus_{i\ge r}\wedge^iF,
\qquad
F^{n+1}A_{\spl}=0.
\]
The quotient \(A_{\spl}/F^1A_{\spl}\) is canonically \(\mathcal O_Y\), and
\(
\operatorname{gr}_F A_{\spl}
\cong
\bigwedge_{\mathcal O_Y}^{\bullet}F
\)
as graded \(\mathcal O_Y\)-superalgebras.
\end{definition}

\begin{definition}[Filtered relative superalgebras with fixed retract]
\label{def:filtered-relative-superalgebra-fixed-retract}
Let \(T\to S\) be a complex analytic test space and set
\[
Y_T\defeq Y\times_S T,
\qquad
\pi_T:Y_T\to T,
\qquad
p_T:Y_T\to Y,
\qquad
F_T\defeq p_T^*F.
\]
A \emph{filtered relative superalgebra on \(Y_T/T\) with associated graded
\(\wedge^\bullet F_T\)} is a sheaf \(\mathcal A\) of
\(\pi_T^{-1}\mathcal O_T\)-supercommutative algebras on \(Y_T\), endowed with a
finite decreasing filtration by ideals
\[
\mathcal A=F^0\mathcal A\supset F^1\mathcal A\supset\cdots\supset F^{n+1}\mathcal A=0,
\]
such that:
\begin{enumerate}
\item \(F^r\mathcal A\cdot F^s\mathcal A\subset F^{r+s}\mathcal A\);
\item there is a fixed identification
\(
\mathcal A/F^1\mathcal A \cong \mathcal O_{Y_T}
\)
as \(\pi_T^{-1}\mathcal O_T\)-algebras;
\item there is an isomorphism of graded \(\mathcal O_{Y_T}\)-superalgebras
\(
\iota:\operatorname{gr}_F\mathcal A
\overset{\sim}{\longrightarrow}
\bigwedge_{\mathcal O_{Y_T}}^{\bullet}F_T.
\)
\end{enumerate}
No \(\mathcal O_{Y_T}\)-algebra structure on \(\mathcal A\) is part of the data.
\end{definition}

\begin{remark}[Why \(\mathcal O_{Y_T}\)-linearity is not imposed]
\label{rem:no-OY-linearity}
The quotient \(\mathcal A\to \mathcal A/F^1\mathcal A\cong\mathcal O_{Y_T}\) is
canonical in the fixed-retract problem.  A splitting of this quotient,
\(
\mathcal O_{Y_T}\rightarrow \mathcal A,
\)
would be precisely a projection of the corresponding supermanifold onto its reduced
space.  Since non-projected supermanifolds are one of the main objects of interest,
Definition~\ref{def:filtered-relative-superalgebra-fixed-retract} deliberately avoids
requiring \(\mathcal A\) to be an \(\mathcal O_{Y_T}\)-algebra.
\end{remark}

\begin{definition}[The classical Green--Onishchik stack]
\label{def:GO-stack}
The \emph{classical Green--Onishchik stack with fixed retract \((Y/S,F)\)} is the
stack in groupoids \(\Tw^{\cl}_{F/S}\) whose value on \(T\to S\) is the groupoid of
pairs \((\mathcal A,\iota)\), where \(\mathcal A\) is a filtered relative superalgebra on
\(Y_T/T\) with associated graded \(\wedge^\bullet F_T\), and \(\iota\) is the chosen
identification
\[
\operatorname{gr}_F\mathcal A\cong \wedge^\bullet F_T.
\]
Morphisms are isomorphisms of filtered
\(\pi_T^{-1}\mathcal O_T\)-superalgebras compatible with the quotient to
\(\mathcal O_{Y_T}\) and with the associated graded identification.
\end{definition}

\begin{definition}[The split section]
\label{def:split-section-GO}
The split model defines a distinguished object over every \(T\to S\), namely
\[
s_{\spl}(T)
\defeq
\bigl(\wedge^\bullet F_T,\operatorname{id}_{\wedge^\bullet F_T}\bigr).
\]
These objects are compatible with pullback and define a section
\[
s_{\spl}:S\longrightarrow \Tw^{\cl}_{F/S},
\]
called the \emph{split section}.
\end{definition}

We next identify the automorphism sheaf whose non-abelian cohomology describes the
same fixed-retract problem.  Since the objects are not \(\mathcal O_Y\)-algebras, the
automorphisms must not be required to be \(\mathcal O_Y\)-linear.  They are relative
over the base only.

\begin{definition}[The Green automorphism sheaf]
\label{def:Green-automorphism-sheaf}
Let \(U\subset Y\) be open.  The \emph{Green automorphism sheaf} of the split model is
the sheaf of groups \(G^{(2)}_{F/S}\) defined by
\[
G^{(2)}_{F/S}(U)
\defeq
\left\{
\varphi\in
\Aut_{\pi^{-1}\mathcal O_S,\bar 0}
\bigl(A_{\spl}|_U\bigr)
\;\middle|\;
\varphi(F^rA_{\spl}|_U)\subset F^rA_{\spl}|_U
\text{ for all }r,
\quad
\operatorname{gr}_F(\varphi)=\operatorname{id}
\right\}.
\]
Equivalently, \(\varphi\in G^{(2)}_{F/S}(U)\) if and only if it is an even
\(\pi^{-1}\mathcal O_S\)-algebra automorphism of \(A_{\spl}|_U\) satisfying
\(
\varphi\equiv \operatorname{id}
\pmod{\mathcal J_{\spl,U}^2}.
\)
The equivalence follows because a filtered algebra automorphism is determined on
functions and on \(F=\mathcal J_{\spl}/\mathcal J_{\spl}^2\), and the associated graded
algebra is generated by these pieces.
\end{definition}

\begin{remark}[Green cocycles]
\label{rem:Green-cocycles}
After choosing local identifications of an object of \(\Tw^{\cl}_{F/S}\) with the split
model, its transition functions are sections of \(G^{(2)}_{F/S}\).  Conversely, a
\(G^{(2)}_{F/S}\)-valued non-abelian \v{C}ech cocycle glues the local split models to a
filtered relative superalgebra with associated graded \(\wedge^\bullet F\).  Thus
\(\Tw^{\cl}_{F/S}\) is equivalently the stack of Green torsors, or split forms, of the
fixed split model.
\end{remark}

The infinitesimal object associated with \(G^{(2)}_{F/S}\) is the sheaf of relative
filtered derivations.  This is the sheaf-theoretic precursor of the dg Lie algebra
\(L_X\) of the previous sections.

\begin{definition}[Relative derivations of the split model]
\label{def:relative-derivation-sheaf}
Let
\[
\mathcal T_{A_{\spl}/S,\bar 0}
\defeq
\Der_{\pi^{-1}\mathcal O_S,\bar 0}(A_{\spl},A_{\spl})
\]
be the sheaf of parity-preserving derivations of \(A_{\spl}=\wedge^\bullet F\) relative
to \(S\).  For \(q\ge -1\), let
\(\mathcal T^{\langle q\rangle}_{A_{\spl}/S}\) be the subsheaf of homogeneous derivations
of weight \(q\), namely
\[
\mathcal T^{\langle q\rangle}_{A_{\spl}/S}(U)
\defeq
\left\{
D\in\Der_{\pi^{-1}\mathcal O_S}(A_{\spl}|_U,A_{\spl}|_U)
\;\middle|\;
D(\wedge^pF|_U)\subset \wedge^{p+q}F|_U
\text{ for all }p\ge 0
\right\}.
\]
Thus \(\mathcal T^{\langle q\rangle}_{A_{\spl}/S}\) has parity \(q\bmod 2\), and
\[
\bigl[
\mathcal T^{\langle q\rangle}_{A_{\spl}/S},
\mathcal T^{\langle q'\rangle}_{A_{\spl}/S}
\bigr]
\subset
\mathcal T^{\langle q+q'\rangle}_{A_{\spl}/S}.
\]
For \(j\ge 1\), define the even Green tails by
\[
\mathcal T^{[\ge 2j]}_{A_{\spl}/S,\bar 0}
\defeq
\bigoplus_{q\ge j}
\mathcal T^{\langle 2q\rangle}_{A_{\spl}/S}.
\]
In particular,
\(
\mathcal T^{[\ge 2]}_{A_{\spl}/S,\bar 0}
=
\bigoplus_{q\ge 1}
\mathcal T^{\langle 2q\rangle}_{A_{\spl}/S}
\)
is a sheaf of nilpotent Lie algebras.
\end{definition}

\begin{proposition}[Homogeneous derivations and Green sheaves]
\label{prop:homogeneous-derivations-Green-exact}
For every \(j\ge1\), there is a natural short exact sequence
\begin{equation}
\label{eq:homogeneous-Green-exact}
0
\longrightarrow
\mathcal H\!om_{\mathcal O_Y}\!\bigl(F,\wedge^{2j+1}F\bigr)
\longrightarrow
\mathcal T^{\langle 2j\rangle}_{A_{\spl}/S}
\overset{\rho_{2j}}{\longrightarrow}
\mathcal T_{Y/S}\otimes_{\mathcal O_Y}\wedge^{2j}F
\longrightarrow 0.
\end{equation}
Equivalently,
\[
0
\longrightarrow
\mathcal Q^{(2j+1)}_{F/S}
\longrightarrow
\mathcal T^{\langle 2j\rangle}_{A_{\spl}/S}
\overset{\rho_{2j}}{\longrightarrow}
\mathcal Q^{(2j)}_{F/S}
\longrightarrow 0,
\]
with the convention \(\mathcal Q^{(2j+1)}_{F/S}=0\) when \(2j+1>\operatorname{rk}F\).
\end{proposition}

\begin{proof}
A homogeneous derivation of weight \(2j\) is determined by its restriction to
\(\mathcal O_Y\) and to \(F\).  Restriction to functions gives the morphism
\(\rho_{2j}\):
\[
D|_{\mathcal O_Y}:\mathcal O_Y\longrightarrow \wedge^{2j}F.
\]
Because the derivations are relative over \(S\), this is a relative derivation, hence a
section of
\[
\Der_{\pi^{-1}\mathcal O_S}(\mathcal O_Y,\wedge^{2j}F)
\cong
\mathcal T_{Y/S}\otimes\wedge^{2j}F.
\]
Surjectivity is local on \(Y\): after trivializing \(F\), any relative derivation
\(\mathcal O_Y\to\wedge^{2j}F\) extends to a derivation of \(\wedge^\bullet F\) by
declaring its value on the odd generators to be zero.  The kernel consists of
homogeneous derivations of weight \(2j\) vanishing on functions.  Such derivations are
\(\mathcal O_Y\)-linear and are uniquely determined by their restriction
\(
F\rightarrow \wedge^{2j+1}F.
\)
This gives the left-hand term and proves exactness.
\end{proof}

In general \eqref{eq:homogeneous-Green-exact} is not canonically split.  This is the
reason why the formal tower naturally sees the whole homogeneous derivation sheaf
\(\mathcal T^{\langle2j\rangle}_{A_{\spl}/S}\), rather than the two Green sheaves
\(\mathcal Q^{(2j)}_{F/S}\) and \(\mathcal Q^{(2j+1)}_{F/S}\) separately.

\begin{definition}[The Dolbeault dg Lie sheaf and the global controller]
\label{def:Dolbeault-dgLie-sheaf}
Let
\[
\mathfrak t_{F/S}
\defeq
\mathcal A_Y^{0,\bullet}
\bigl(\mathcal T^{[\ge 2]}_{A_{\spl}/S,\bar 0}\bigr)
\]
be the Dolbeault resolution of the holomorphic sheaf
\(\mathcal T^{[\ge 2]}_{A_{\spl}/S,\bar 0}\).  The differential is the Dolbeault
differential \(\bar\partial\), and the bracket is the graded commutator of
relative derivation-valued forms.  The \emph{relative Green dg Lie controller} is
\[
\mathfrak g_{F/S}
\defeq
R\pi_*\mathfrak t_{F/S}.
\]
Thus, for an open set \(V\subset S\),
\[
\mathfrak g_{F/S}(V)
\simeq
R\Gamma\bigl(Y_V,
\mathfrak t_{F_V/V}\bigr),
\qquad
Y_V\defeq \pi^{-1}(V),
\]
where \(F_V\) denotes the pullback of \(F\) to \(Y_V\).
\end{definition}

\begin{remark}[The role of \(\mathfrak g_{F/S}\)]
\label{rem:role-of-global-Lie-object}
The sheaf \(\mathfrak t_{F/S}\) is the infinitesimal Lie algebra of the Green
automorphism sheaf.  Its fibrewise derived global sections control the formal
fixed-retract problem.  The pushforward \(\mathfrak g_{F/S}=R\pi_*\mathfrak t_{F/S}\)
packages these fibrewise controllers over the base, but it represents the fibrewise
controller after base change only under the explicit hypothesis below.  In the absolute
case \(S=\Spec\mathbb C\), it reduces to the filtered dg Lie algebra \(L_X\) used earlier
in the paper.
\end{remark}

\begin{assumption}[Derived base change for the Green controller]
\label{ass:green-derived-base-change}
Let \(f:T\to S\) be a test space, set
\[
Y_T\defeq Y\times_S T,
\qquad
\pi_T:Y_T\to T,
\qquad
F_T\defeq f_Y^*F,
\]
where \(f_Y:Y_T\to Y\) is the natural projection.  We say that the Green controller
satisfies \emph{derived base change along \(T\to S\)} if the canonical base-change morphism
of filtered dg Lie objects
\[
Lf^* R\pi_*\mathfrak t_{F/S}
\longrightarrow
R\pi_{T*}\mathfrak t_{F_T/T}
\]
is a quasi-isomorphism.  We say that derived base change holds on a class of test spaces if
this condition holds for every test space in that class and is compatible with further
base change.

Whenever a statement below is written in terms of \(R\pi_*\mathfrak t_{F/S}\), or in
terms of derived pushforwards of the Green tails and quotient sheaves, this assumption is
part of the statement.  The fibrewise formal statements, formulated directly with
\(R\Gamma(Y_T,\mathfrak t_{F_T/T})\), do not require it.
\end{assumption}

\subsection{\v{C}ech--Deligne description of the formal neighbourhood}
\label{subsec:Cech-Deligne-GO}

We now pass from the classical non-abelian Green cocycles to a derived formal moduli
problem \cite{LurieDAGX}.  Locally near the split section, all transition functions reduce to the identity;
hence they take values in a nilpotent subgroup of \(G^{(2)}_{F/S}\).  In characteristic
zero, this nilpotent group is obtained by exponentiating the nilpotent Lie algebra of
relative filtered derivations. Therefore infinitesimal Green cocycles are equivalently
Maurer--Cartan elements in the Thom--Whitney totalization of the \v{C}ech
semicosimplicial dg Lie algebra associated with \(\mathfrak t_{F/S}\) -- see for example \cite{NavarroAznarHodgeDeligne,FiorenzaIaconoMartinengoDGLA,GetzlerLieTheory,HinichDGCoalgebras}. 

Since Thom--Whitney totalization was already introduced in
Section~\ref{sec:cech-kuranishi-at-symbols}, we only recall the notation needed in the
relative setting.  Ordinary \v{C}ech cocycles will still be used as representatives; the
strict dg Lie structure is placed on the Thom--Whitney totalization.

\begin{definition}[Classical formal completion at the split section]
\label{def:classical-formal-completion-split}
Let \(T\to S\) be a complex analytic test space, and let \(A\) be a local Artinian
\(\mathbb C\)-algebra with maximal ideal \(\mathfrak m_A\).  Write \(T_A\) for the
nilpotent thickening of \(T\) with structure sheaf \(\mathcal O_T\otimes_{\mathbb C}A\).
The classical formal completion of \(\Tw^{\cl}_{F/S}\) at the split section is the
functor
\(
\widehat{\Tw^{\cl}_{F/S}}{}_{\,s_{\spl}}
\)
whose value on \((T,A)\) is the full subgroupoid of \(\Tw^{\cl}_{F/S}(T_A)\)
consisting of those objects whose reduction modulo \(\mathfrak m_A\) is identified with
the split object \(s_{\spl}(T)\).
\end{definition}

\begin{definition}[Relative Thom--Whitney \v{C}ech controllers]
\label{def:relative-TW-Cech-controller}
Let \(T\to S\) be a test space.  A cover
\(\mathfrak U_T=\{U_i\}_{i\in I}\) of \(Y_T\) will be called sufficiently fine if all
finite intersections
\[
U_{i_0\cdots i_p}\defeq U_{i_0}\cap\cdots\cap U_{i_p}
\]
are Stein and trivialize \(F_T\).  Such a cover determines a semicosimplicial dg Lie
algebra
\(
\check C^\bullet(\mathfrak U_T,\mathfrak t_{F_T/T}),\) with
\[
\check C^p(\mathfrak U_T,\mathfrak t_{F_T/T})
\defeq
\prod_{i_0<\cdots<i_p}
\Gamma(U_{i_0\cdots i_p},\mathfrak t_{F_T/T}),
\]
and coface maps given by restriction.  Using the Thom--Whitney totalization of
Section~\ref{sec:cech-kuranishi-at-symbols}, we set
\[
\check{\mathfrak g}^{\mathrm{TW}}_{F_T/T}(\mathfrak U_T)
\defeq
\Tot_{\mathrm{TW}}
\bigl(
\check C^\bullet(\mathfrak U_T,\mathfrak t_{F_T/T})
\bigr).
\]
\end{definition}

\begin{definition}[The relative \v{C}ech--Deligne functor]
\label{def:Cech-Deligne-FMP}
Let \(B\) be a small augmented connective Artin cdga over \(\mathbb C\), with
augmentation ideal \(\mathfrak m_B\).  The \emph{relative \v{C}ech--Deligne functor}
attached to \(\mathfrak U_T\) is
\[
\Def_{\mathfrak U_T}(\mathfrak t_{F_T/T})(B)
\defeq
\MC_\infty\!\left(
\check{\mathfrak g}^{\mathrm{TW}}_{F_T/T}(\mathfrak U_T)
\otimes \mathfrak m_B
\right).
\]
Here \(\MC_\infty\) denotes the simplicial Maurer--Cartan functor.  For a nilpotent
dg Lie algebra \(\mathfrak h\), it is the Kan complex whose \(p\)-simplices are
Maurer--Cartan elements of \(\mathfrak h\otimes\Omega^\bullet(\Delta^p)\), where \(\Omega^\bullet(\Delta^p)\) denotes the cdga of polynomial differential forms
on the standard simplex. 
\end{definition}
Here we used the standard Thom--Whitney model for
semicosimplicial dg Lie algebras and the associated simplicial Maurer--Cartan
formalism; see
\cite{NavarroAznarHodgeDeligne,FiorenzaIaconoMartinengoDGLA,GetzlerLieTheory,HinichDGCoalgebras} as above.
The interpretation in terms of formal moduli problems is the one of
\cite{LurieDAGX,PridhamUDDT}.

\begin{remark}[Infinitesimal Green cocycles]
\label{rem:infinitesimal-Green-cocycles}
An infinitesimal Green cocycle over an Artin thickening is a system of transition
functions
\(
g_{ij}\in G^{(2)}_{F_T/T}(U_{ij})
\)
which reduces to the identity modulo the Artin maximal ideal and satisfies the usual
cocycle identity on triple overlaps.  Since the Green filtration and the Artin ideal are
nilpotent, one may write
\[
g_{ij}=\exp(\gamma_{ij}),
\qquad
\gamma_{ij}\in
\Gamma(U_{ij},\mathfrak t_{F_T/T})\otimes \mathfrak m_A .
\]
The Baker--Campbell--Hausdorff formula translates the non-abelian cocycle equation
for the \(g_{ij}\)'s into the Maurer--Cartan equation in the totalized \v{C}ech dg Lie
algebra.  Changes of local trivializations give the corresponding gauge equivalences.
This is the \v{C}ech--Deligne interpretation of Green's non-abelian cohomology, see \cite{DonagiWittenNotProjected, Green1982}.
\end{remark}

\begin{proposition}[\v{C}ech--Deligne identification]
\label{prop:Cech-Deligne-identification}
Let \(T\to S\) be a test space and let \(\mathfrak U_T\) be a sufficiently fine Stein
cover of \(Y_T\).  On ordinary Artin algebras, there is a natural equivalence of
pointed groupoids
\[
\widehat{\Tw^{\cl}_{F/S}}{}_{\,s_{\spl}}(T,-)
\simeq
\Def_{\mathfrak U_T}(\mathfrak t_{F_T/T})(-).
\]
Consequently, \(\Def_{\mathfrak U_T}(\mathfrak t_{F_T/T})\) is a derived enhancement of
the classical formal completion at the split section.
\end{proposition}

\begin{proof}
Over the cover \(\mathfrak U_T\), every object in the formal neighbourhood of the split
section is locally isomorphic to the split model.  Since its reduction is the split object,
all transition functions lie in the kernel of the reduction map to the central fibre,
hence in the nilpotent identity neighbourhood of \(G^{(2)}_{F_T/T}\).  The exponential
map identifies this nilpotent group with
\(\mathfrak t_{F_T/T}\otimes\mathfrak m_A\).  The descent condition for the transition
functions is the non-abelian \v{C}ech cocycle equation, and the standard
\v{C}ech--Deligne correspondence identifies such cocycles, together with their gauge
transformations, with the Deligne groupoid of
\(
\check{\mathfrak g}^{\mathrm{TW}}_{F_T/T}(\mathfrak U_T)\otimes\mathfrak m_A .
\)
This is the truncation of the simplicial Maurer--Cartan functor appearing in
Definition~\ref{def:Cech-Deligne-FMP}.  Hence the two pointed groupoids agree on
ordinary Artin algebras, and the latter functor gives the derived enhancement.
\end{proof}

\begin{lemma}[The \v{C}ech model computes derived global sections]
\label{lem:Cech-computes-global}
For every sufficiently fine Stein cover \(\mathfrak U_T\) of \(Y_T\), there is a natural
quasi-isomorphism of dg Lie algebras
\[
\check{\mathfrak g}^{\mathrm{TW}}_{F_T/T}(\mathfrak U_T)
\simeq
R\Gamma(Y_T,\mathfrak t_{F_T/T}).
\]
\end{lemma}

\begin{proof}
The Dolbeault complex
\(
\mathfrak t_{F_T/T}
=
\mathcal A_{Y_T}^{0,\bullet}
\bigl(\mathcal T^{[\ge2]}_{A_{\spl,T}/T,\bar 0}\bigr)
\)
is a fine resolution of the coherent sheaf
\(\mathcal T^{[\ge2]}_{A_{\spl,T}/T,\bar 0}\).  Thus its derived global sections may be
computed either by global sections of the fine resolution or by any acyclic
\v{C}ech resolution.  Since the chosen cover is Stein on all finite intersections, the
Thom--Whitney totalization of the \v{C}ech semicosimplicial dg Lie algebra gives a dg
Lie model for \(R\Gamma(Y_T,\mathfrak t_{F_T/T})\).
\end{proof}

\begin{corollary}[Independence of the cover]
\label{cor:independence-cover}
The derived formal moduli problem
\(\Def_{\mathfrak U_T}(\mathfrak t_{F_T/T})\) is independent, up to canonical equivalence,
of the sufficiently fine Stein cover \(\mathfrak U_T\).
\end{corollary}

\begin{proof}
Any two sufficiently fine covers admit a common refinement.  By
Lemma~\ref{lem:Cech-computes-global}, all corresponding Thom--Whitney totalizations
are quasi-isomorphic to the same dg Lie algebra
\(R\Gamma(Y_T,\mathfrak t_{F_T/T})\).  The simplicial Maurer--Cartan functor is invariant
under quasi-isomorphisms of nilpotent dg Lie algebras.  Hence the associated formal
moduli problems are canonically equivalent.
\end{proof}

\subsection{Intrinsic formal completion and absolute specialization}
\label{subsec:formal-GO-absolute}

The \v{C}ech construction is local on \(Y\) and depends on a cover only up to canonical
equivalence.  The intrinsic object is therefore the fibrewise formal moduli problem
attached to the dg Lie algebras
\(
R\Gamma(Y_T,\mathfrak t_{F_T/T}).
\)
This is the formal Green--Onishchik enhancement. Note that it is formal rather than algebraized:
no claim is made here that it is represented by an analytic or algebraic derived stack.
Such an algebraization requires additional finiteness, amplitude, and automorphism
hypotheses and is addressed separately.

\begin{definition}[The fibrewise formal Green--Onishchik enhancement]
\label{def:formal-GO-enhancement}
The \emph{formal Green--Onishchik enhancement} of the fixed-retract moduli problem is
the pointed formal moduli problem
\(
\widehat\Tw_{F/S}
\)
whose value over a test space \(T\to S\) and a small augmented Artin cdga \(B\), with
augmentation ideal \(\mathfrak m_B\), is
\[
\widehat\Tw_{F/S}(T,B)
\defeq
\MC_\infty\!\left(
R\Gamma(Y_T,\mathfrak t_{F_T/T})\otimes\mathfrak m_B
\right).
\]
Its base point is the zero Maurer--Cartan element, corresponding to the split section.
\end{definition}

\begin{definition}[Formal moduli problem attached to a relative dg Lie algebra]
\label{def:relative-Def-functor}
Let \(\mathfrak g\) be a filtered nilpotent dg Lie object over \(S\).  We denote by
\(
\Def_S(\mathfrak g)
\)
the pointed formal moduli problem over \(S\) whose value over a test space
\(f:T\to S\) and a small augmented Artin cdga \(B\), with augmentation ideal
\(\mathfrak m_B\), is
\[
\Def_S(\mathfrak g)(T,B)
\defeq
\MC_\infty\!\bigl(R\Gamma(T,Lf^*\mathfrak g)\otimes\mathfrak m_B\bigr).
\]
This notation will be used for the fixed-retract problem only when
\(Lf^*\mathfrak g\) computes the intended fibrewise controller.  In particular, under
Assumption~\ref{ass:green-derived-base-change}, with
\(\mathfrak g=\mathfrak g_{F/S}=R\pi_*\mathfrak t_{F/S}\), one has
\[
R\Gamma(T,Lf^*\mathfrak g_{F/S})
\simeq
R\Gamma(Y_T,\mathfrak t_{F_T/T}),
\]
and therefore
\[
\widehat\Tw_{F/S}(T,B)
\simeq
\Def_S(\mathfrak g_{F/S})(T,B).
\]
\end{definition}

\begin{theorem}[Fibrewise formal Green--Onishchik theorem]
\label{thm:formal-GO}
There is a canonical equivalence of pointed fibrewise formal moduli problems over \(S\)
\[
\widehat{\Tw^{\cl}_{F/S}}{}_{\,s_{\spl}}^{\,\mathrm{der}}
\simeq
\widehat\Tw_{F/S}.
\]
Equivalently, for every test space \(T\to S\), the derived \v{C}ech--Deligne enhancement
of the classical formal completion at the split object over \(T\) is controlled by the dg
Lie algebra
\(
R\Gamma(Y_T,\mathfrak t_{F_T/T}).
\)
Hence its tangent complex at the split point is
\(
R\Gamma(Y_T,\mathfrak t_{F_T/T})[1].
\)

\smallskip

This fibrewise statement is unconditional.
\end{theorem}

\begin{proof}
By Proposition~\ref{prop:Cech-Deligne-identification}, the formal neighbourhood of the
split section is described, on each test space \(T\to S\), by the \v{C}ech--Deligne
formal moduli problem associated with any sufficiently fine Stein cover of \(Y_T\).  By
Lemma~\ref{lem:Cech-computes-global}, this \v{C}ech dg Lie algebra is quasi-isomorphic
to
\(
R\Gamma(Y_T,\mathfrak t_{F_T/T}).
\)
Invariance of the Maurer--Cartan formal moduli functor under quasi-isomorphisms of
nilpotent dg Lie algebras identifies the resulting formal moduli problem with
\(\widehat\Tw_{F/S}(T,B)\).  The tangent-complex statement is the standard correspondence
between formal moduli problems and dg Lie algebras in characteristic zero.
\end{proof}

\begin{corollary}[Relative base-change form]
\label{cor:formal-GO-basechange-form}
Assume that the Green controller satisfies derived base change in the sense of
Assumption~\ref{ass:green-derived-base-change} for the test spaces under consideration.
Then the fibrewise equivalence of Theorem~\ref{thm:formal-GO} may be written as an
equivalence of relative pointed formal moduli problems
\[
\widehat{\Tw^{\cl}_{F/S}}{}_{\,s_{\spl}}^{\,\mathrm{der}}
\simeq
\Def_S(\mathfrak g_{F/S}),
\qquad
\mathfrak g_{F/S}=R\pi_*\mathfrak t_{F/S}.
\]
In this relative form, the tangent complex along the split section is
\[
\mathcal T_{\widehat\Tw_{F/S}/S,s_{\spl}}
\simeq
\mathfrak g_{F/S}[1]
\simeq
R\pi_*\mathfrak t_{F/S}[1]
\simeq
R\pi_*\mathcal T^{[\ge2]}_{A_{\spl}/S,\bar0}[1].
\]
\end{corollary}

\begin{proof}
Under Assumption~\ref{ass:green-derived-base-change}, for every \(f:T\to S\) one has
\[
R\Gamma(T,Lf^*\mathfrak g_{F/S})
\simeq
R\Gamma(Y_T,\mathfrak t_{F_T/T}).
\]
Substituting this identification into Definition~\ref{def:relative-Def-functor} gives
\(\widehat\Tw_{F/S}\simeq\Def_S(\mathfrak g_{F/S})\).  The tangent-complex formula follows
from the same dg-Lie/formal-moduli correspondence used in Theorem~\ref{thm:formal-GO}.
\end{proof}

\begin{remark}[Analytic and algebraic readings]
\label{rem:analytic-algebraic-readings}
In the analytic formulation above, \(\mathfrak t_{F/S}\) is a Dolbeault dg Lie model.
Algebraically, one should replace it by the coherent sheaf
\(\mathcal T^{[\ge2]}_{A_{\spl}/S,\bar0}\) and write the controller as
\(
R\pi_*\mathcal T^{[\ge2]}_{A_{\spl}/S,\bar0}.
\)
We remark that the two descriptions agree after analytification, because the Dolbeault complex is a
fine resolution of the holomorphic sheaf of relative filtered derivations.
\end{remark}

We finally check that this relative formal moduli problem specializes to the intrinsic
formal moduli problem constructed earlier from a single supermanifold \(X\).

\begin{lemma}[Identification of the absolute controller]
\label{lem:absolute-controller-identification}
Let \(X\) be a complex supermanifold, let
\[
S=\Spec\mathbb C,
\qquad
Y=\Xred,
\qquad
F=\FF_X,
\qquad
A_{\spl}=\wedge^\bullet\FF_X.
\]
Then the relative Green controller specializes to
\[
\mathfrak g_{F/S}
\simeq
R\Gamma\!\left(
\Xred,
\mathcal A_{\Xred}^{0,\bullet}
\bigl(\mathcal T^{[\ge2]}_{\widehat X,\bar0}\bigr)
\right).
\]
With the notation of the preceding sections, this dg Lie algebra is exactly
\[
L_X=A^{0,\bullet}\!\left(\Xred,
\mathcal T^{[\ge2]}_{\widehat X,\bar0}\right).
\]
\end{lemma}

\begin{proof}
When \(S=\Spec\mathbb C\), relative derivations over \(S\) are simply
\(\mathbb C\)-linear derivations.  The split model \(A_{\spl}=\wedge^\bullet\FF_X\) is the
split model \(\widehat{\mathcal O}_X\) used earlier in the paper.  Hence
\(
\mathcal T^{[\ge2]}_{A_{\spl}/S,\bar0}
=
\mathcal T^{[\ge2]}_{\widehat X,\bar0}.
\)
Taking the Dolbeault resolution and derived global sections gives exactly the definition
of \(L_X\) from the previous sections.
\end{proof}

\begin{corollary}[Absolute specialization]
\label{cor:absolute-formal-GO}
Under the specialization
\[
S=\Spec\mathbb C,
\qquad
Y=\Xred,
\qquad
F=\FF_X,
\]
there is a canonical equivalence of pointed formal moduli problems
\[
\widehat{\Tw^{\cl}_{F/S}}{}_{\,s_{\spl}}^{\,\mathrm{der}}
\simeq
\Def(L_X)
=
\Def_X.
\]
\end{corollary}

\begin{proof}
By Theorem~\ref{thm:formal-GO}, the derived formal completion of the classical
Green--Onishchik stack at the split section is controlled, in the absolute case, by
\(
R\Gamma(\Xred,\mathfrak t_{F/S}).
\)
By Lemma~\ref{lem:absolute-controller-identification}, this controller is precisely
\(L_X\).  Therefore the corresponding pointed formal moduli problem is
\(\Def(L_X)=\Def_X\).
\end{proof}

\subsection{The classical Green obstruction tower}
\label{subsec:relative-classical-obstruction-tower}

We now return to Green's classical lifting tower we considered in the previous sections, but in the relative fixed-retract
setting introduced above.  The point of this subsection is deliberately classical: given a
superalgebra with associated graded \(\wedge^\bullet F\), and given a splitting up to a
fixed order in the odd filtration, one asks whether this splitting can be lifted by one
more order.  The answer is measured by a cohomology class with values in one of the
classical Green obstruction sheaves.  The precise formulation is slightly delicate,
because the filtered algebra is not an \(\mathcal O_Y\)-algebra.  Thus all splittings and
all automorphisms below are relative over \(\pi^{-1}\mathcal O_S\), and are required only
to be compatible with the reduced quotient and the associated graded identification.

\begin{definition}[Truncated split models]
\label{def:truncated-split-models-relative}
Let \(T\to S\) be a test space.  For \(0\le k\le n\), set
\[
A_{\spl,T}^{\le k}
\defeq
A_{\spl,T}/F^{k+1}A_{\spl,T}
=
\bigl(\wedge^\bullet F_T\bigr)\big/F^{k+1}\bigl(\wedge^\bullet F_T\bigr).
\]
Thus \(A_{\spl,T}^{\le k}\) is the split model truncated modulo terms of odd degree
\(\ge k+1\).
\end{definition}

\begin{definition}[The stacks of \(k\)-splittings]
\label{def:classical-k-split-forms}
For \(1\le k\le n\), define \(\mathfrak U_k^{\cl}\) to be the stack whose value on
\(T\to S\) is the groupoid of triples
\(
(\mathcal A,\iota,\sigma_k),
\)
where \((\mathcal A,\iota)\in \Tw^{\cl}_{F/S}(T)\), and
\[
\sigma_k:
\mathcal A/F^{k+1}\mathcal A
\overset{\sim}{\longrightarrow}
A_{\spl,T}^{\le k}
\]
is an isomorphism of filtered \(\pi_T^{-1}\mathcal O_T\)-superalgebras such that:
\begin{enumerate}
\item it induces the identity on the reduced quotient
\[
\mathcal A/F^1\mathcal A\cong \mathcal O_{Y_T}
\cong
A_{\spl,T}^{\le k}/F^1A_{\spl,T}^{\le k};
\]
\item the induced map on associated graded algebras agrees with the fixed
identification \(\iota:\operatorname{gr}_F\mathcal A\cong \wedge^\bullet F_T\).
\end{enumerate}
A morphism in \(\mathfrak U_k^{\cl}(T)\) is a morphism in \(\Tw^{\cl}_{F/S}(T)\) whose
truncation to order \(k\) is compatible with the chosen \(k\)-splittings.
\end{definition}

\begin{remark}
\label{rem:k-splitting-not-projection}
Note that a \(k\)-splitting is not a projection \(\mathcal O_{Y_T}\to\mathcal A\).  It is only a
truncated identification of the filtered relative superalgebra with the split model.
In particular, the existence of splittings for all \(k\) is equivalent to splitness of the
fixed-retract object, whereas the existence of a projection is a weaker and different
condition in the usual supergeometric terminology, as we discussed.
\end{remark}

\begin{definition}[The classical Green obstruction sheaves]
\label{def:relative-Green-obstruction-sheaves}
For every integer \(r\ge2\), define
\[
\mathcal Q^{(r)}_{F/S}
\defeq
\begin{cases}
\mathcal T_{Y/S}\otimes_{\mathcal O_Y}\wedge^r F, & r \text{ even},\\[4pt]
\mathcal H\!om_{\mathcal O_Y}\!\bigl(F,\wedge^r F\bigr), & r \text{ odd}.
\end{cases}
\]
Here \(\wedge^rF=0\) for \(r>n\).  The first case records the possibility of changing
even functions by terms of even exterior degree \(r\); the second records the possibility
of changing odd generators by terms of odd degree \(r\).
\end{definition}

\begin{proposition}[Relative Green obstruction for one-step liftings]
\label{prop:classical-relative-obstruction}
Let \(T\to S\) be a test space, let \(1\le k\le n-1\), and let
\(
(\mathcal A,\iota,\sigma_k)
\in
\mathfrak U_k^{\cl}(T).
\)
Then there is a canonically associated class
\[
\omega_{k+1}(\mathcal A,\iota,\sigma_k)
\in
H^1\bigl(Y_T,\mathcal Q^{(k+1)}_{F_T/T}\bigr)
\]
with the following properties:
\begin{enumerate}
\item it is functorial in \(T\) and depends only on the isomorphism class of
\((\mathcal A,\iota,\sigma_k)\);
\item it vanishes if and only if \(\sigma_k\) extends to a \((k+1)\)-splitting
\[
\sigma_{k+1}:
\mathcal A/F^{k+2}\mathcal A
\overset{\sim}{\longrightarrow}
A_{\spl,T}^{\le k+1};
\]
\item whenever it vanishes, the set of isomorphism classes of extensions
\(\sigma_{k+1}\) of the fixed \(k\)-splitting \(\sigma_k\) is a torsor under
\(
H^0\bigl(Y_T,\mathcal Q^{(k+1)}_{F_T/T}\bigr).
\)
\end{enumerate}
\end{proposition}

\begin{proof}
Choose a Stein cover \(\mathfrak U_T=\{U_i\}\) of \(Y_T\), sufficiently fine so that
\(F_T\) is trivial and \((\mathcal A,\iota)\) is locally isomorphic to the split model.
Since \(\sigma_k\) is fixed globally, its restrictions admit local lifts
\[
\widetilde\sigma_{i,k+1}:
(\mathcal A/F^{k+2}\mathcal A)|_{U_i}
\overset{\sim}{\longrightarrow}
A_{\spl,T}^{\le k+1}|_{U_i}.
\]
On overlaps set
\(
g_{ij}
\defeq
\widetilde\sigma_{i,k+1}\circ\widetilde\sigma_{j,k+1}^{-1}.
\)
Then \(g_{ij}\) is an even filtered automorphism of \(A_{\spl,T}^{\le k+1}|_{U_{ij}}\)
which is the identity modulo \(F^{k+1}\).  Its first possible nonzero term therefore has
order \(k+1\).  If \(k+1\) is even, this term is obtained by restricting \(g_{ij}-1\) to
functions and is a section of
\(\mathcal T_{Y_T/T}\otimes\wedge^{k+1}F_T\).  If \(k+1\) is odd, parity forces the
function part to vanish, and the first nonzero term is the induced
\(\mathcal O_{Y_T}\)-linear map
\(F_T\to\wedge^{k+1}F_T\).  Thus in both cases one obtains a \v{C}ech
\(1\)-cochain
\(
q_{ij}\in \Gamma(U_{ij},\mathcal Q^{(k+1)}_{F_T/T}).
\)
The cocycle condition for the transition functions \(g_{ij}\) implies that
\((q_{ij})\) is a \v{C}ech \(1\)-cocycle.  Changing the local lifts
\(\widetilde\sigma_{i,k+1}\) changes \((q_{ij})\) by a \v{C}ech coboundary.  Hence its
cohomology class is well-defined; this is
\(\omega_{k+1}(\mathcal A,\iota,\sigma_k)\).

The class vanishes exactly when the local lifts may be modified so that the corrected
transition functions are trivial to order \(k+1\).  This is precisely the condition for the
local lifts to glue to a global \((k+1)\)-splitting.  When this is possible, two such global
extensions differ by an automorphism of the split \((k+1)\)-truncation which is the
identity modulo \(F^{k+1}\), equivalently through order \(k\), hence by a global section of
\(\mathcal Q^{(k+1)}_{F_T/T}\).  This gives the stated torsor description.
\end{proof}

\begin{definition}[Classical and formal vector stacks]
\label{def:classical-formal-vector-stacks}
If \(E\) is a holomorphic vector bundle on \(S\), we write
\(
\mathbb V^{\cl}_S(E)
\)
for its total space, viewed as the functor
\(
T\mapsto H^0(T,E_T).
\)

If \(K\) is a complex on \(S\), we write
\(
\mathbb V^{\mathrm{for}}_S(K)
\)
for the formal moduli problem attached to the abelian Lie algebra object \(K[-1]\).
Thus, if \(M\) is an abelian dg Lie algebra object over \(S\), then
\[
\Def_S(M)=\mathbb V^{\mathrm{for}}_S(M[1]).
\]
When \(K\) is perfect, this is represented by the usual formal vector stack with tangent
complex \(K\).  This convention means that a section of \(E\) is a point of
\(\mathbb V(E)\).\footnote{Consequently no duals appear in the Kuranishi zero loci below.}  When
no confusion is possible, the subscript \(S\) will be omitted.

If \(\mathcal M\) is a coherent sheaf, or a bounded complex, on \(Y\), we also use the
fibrewise notation
\[
\mathbb V^{\mathrm{for}}_{F/S}(\mathcal M[1])(T,B)
\defeq
\MC_\infty\!\left(R\Gamma(Y_T,\mathcal M_T)\otimes\mathfrak m_B\right),
\]
where \(\mathcal M_T\) denotes the pullback to \(Y_T\), regarded as an abelian dg Lie
algebra.  If \(R\pi_*\mathcal M\) satisfies derived base change, then
\[
\mathbb V^{\mathrm{for}}_{F/S}(\mathcal M[1])
\simeq
\mathbb V^{\mathrm{for}}_S\!\bigl(R\pi_*\mathcal M[1]\bigr).
\]
\end{definition}

\begin{corollary}[Classical obstruction morphism and its zero locus]
\label{cor:classical-obstruction-morphism}
Fix \(1\le k\le n-1\).  Assume that
\(
R^0\pi_*\mathcal Q^{(k+1)}_{F/S}
\) and \(
R^1\pi_*\mathcal Q^{(k+1)}_{F/S}
\)
are locally free and commute with base change.  Then the classes of
Proposition~\ref{prop:classical-relative-obstruction} assemble into a natural morphism
\[
\omega_{k+1}^{\cl}:
\mathfrak U_k^{\cl}
\longrightarrow
\mathbb V^{\cl}\!\bigl(R^1\pi_*\mathcal Q^{(k+1)}_{F/S}\bigr).
\]
Let \(Z(\omega_{k+1}^{\cl})\subset \mathfrak U_k^{\cl}\) denote its zero locus.  Then the
forgetful morphism
\(
\mathfrak U_{k+1}^{\cl}
\rightarrow
Z(\omega_{k+1}^{\cl})
\)
is a torsor under the vector group
\(
\mathbb V^{\cl}\!\bigl(R^0\pi_*\mathcal Q^{(k+1)}_{F/S}\bigr).
\)

In particular, \(Z(\omega_{k+1}^{\cl})\) is the stack of \(k\)-splittings which are
extendable to order \(k+1\), whereas \(\mathfrak U_{k+1}^{\cl}\) remembers the additional
choice of such an extension.
\end{corollary}

\begin{proof}
By the base-change hypothesis, for every \(T\to S\) the group
\(H^1(Y_T,\mathcal Q^{(k+1)}_{F_T/T})\) is functorially the fibre of
\(R^1\pi_*\mathcal Q^{(k+1)}_{F/S}\) over \(T\).  Therefore the obstruction classes
assemble into the displayed morphism.  Proposition~\ref{prop:classical-relative-obstruction}
identifies the zero locus with the substack of \(\mathfrak U_k^{\cl}\) consisting of
extendable \(k\)-splittings.  The same proposition identifies the choices of extensions
with a torsor under
\(H^0(Y_T,\mathcal Q^{(k+1)}_{F_T/T})\), which is the fibre of
\(R^0\pi_*\mathcal Q^{(k+1)}_{F/S}\) by the base-change hypothesis.
\end{proof}

\subsection{The formal Green tower and its classical projections}
\label{subsec:canonical-formal-Green-tower}

We now explain how the classical one-step obstruction tower is reflected in the formal
Green--Onishchik enhancement.  The key point is that the formal tower is not indexed by
the individual sheaves \(\mathcal Q^{(r)}_{F/S}\).  The dg Lie algebra controlling the
formal moduli problem consists of even derivations, and its associated graded pieces are
therefore the even homogeneous derivation sheaves
\(
\mathcal T^{\langle 2j\rangle}_{A_{\spl}/S}.
\)
Each such homogeneous piece contains two successive classical Green sheaves: the quotient
\(\mathcal Q^{(2j)}_{F/S}\), which measures the obstruction to obtaining a
\(2j\)-splitting, and the kernel \(\mathcal Q^{(2j+1)}_{F/S}\), which gives the residual
obstruction to obtaining a \((2j+1)\)-splitting after the even obstruction has vanished and
a \(2j\)-splitting has been chosen. This is reminiscent of the absolute theory developed in the previous sections. 

\smallskip

Set
\(
N\defeq \bigl\lfloor n/2\bigr\rfloor .
\)
For \(1\le j\le N\), write
\(
\mathcal E_j
\defeq
\mathcal T^{\langle 2j\rangle}_{A_{\spl}/S}.
\)
By Proposition~\ref{prop:homogeneous-derivations-Green-exact}, there is a natural short
exact sequence
\begin{equation}
\label{eq:E-j-Q-exact}
0
\longrightarrow
\mathcal Q^{(2j+1)}_{F/S}
\longrightarrow
\mathcal E_j
\overset{\rho_{2j}}{\longrightarrow}
\mathcal Q^{(2j)}_{F/S}
\longrightarrow
0,
\end{equation}
where \(\mathcal Q^{(2j+1)}_{F/S}=0\) if \(2j+1>n\).

\begin{definition}[Formal Green stages]
\label{def:formal-Green-stages}
For \(1\le j\le N+1\), define the fibrewise formal Green stage
\(\widehat{\mathfrak U}^{[j]}_{F/S}\) by
\[
\widehat{\mathfrak U}^{[j]}_{F/S}(T,B)
\defeq
\MC_\infty\!\left(
R\Gamma\bigl(Y_T,
\mathcal A^{0,\bullet}_{Y_T}
(\mathcal T^{[\ge 2j]}_{A_{\spl,T}/T,\bar0})
\bigr)
\otimes\mathfrak m_B
\right),
\]
where \(T\to S\) is a test space and \(B\) is a small augmented Artin cdga.  We use the
convention
\(
\mathcal T^{[\ge 2(N+1)]}_{A_{\spl}/S,\bar0}=0.
\)
Thus
\[
\widehat{\mathfrak U}^{[1]}_{F/S}=\widehat\Tw_{F/S},
\qquad
\widehat{\mathfrak U}^{[N+1]}_{F/S}=S
\]
as pointed fibrewise formal moduli problems.

If derived base change holds for the Green tail
\(R\pi_*\mathcal T^{[\ge 2j]}_{A_{\spl}/S,\bar0}\), then this same stage may be written in
relative notation as
\[
\widehat{\mathfrak U}^{[j]}_{F/S}
\simeq
\Def_S\!\bigl(\mathfrak g^{[j]}_{F/S}\bigr),
\qquad
\mathfrak g^{[j]}_{F/S}
\defeq
R\pi_*\mathcal T^{[\ge 2j]}_{A_{\spl}/S,\bar0}.
\]
\end{definition}

\begin{remark}[Meaning of the index]
\label{rem:formal-index-meaning}
The stage \(\widehat{\mathfrak U}^{[j]}_{F/S}\) is the formal counterpart of having made
all Green terms of weights \(<2j\) vanish in an adapted presentation.  In the absolute
Dolbeault language, this means that one works with Maurer--Cartan representatives lying
in the filtration step corresponding to derivations of weight at least \(2j\).  Classically,
this is the formal analogue of having chosen a \((2j-1)\)-splitting, as discussed in the previous sections.
\end{remark}

\begin{theorem}[Canonical formal Green tower]
\label{thm:formal-Green-tower}
For every \(1\le j\le N\), there is a canonical morphism of pointed fibrewise formal
moduli problems over \(S\)
\[
\widetilde\Omega_j^{\mathrm{for}}:
\widehat{\mathfrak U}^{[j]}_{F/S}
\longrightarrow
\mathbb V^{\mathrm{for}}_{F/S}\!\bigl(\mathcal E_j[1]\bigr),
\]
where \(\mathbb V^{\mathrm{for}}_{F/S}(\mathcal E_j[1])\) denotes the fibrewise formal
vector stack whose value over \(T\to S\) is attached to the abelian dg Lie algebra
\(
R\Gamma(Y_T,\mathcal E_{j,T}).
\)
Fibrewise, the map is induced by the quotient
\[
\mathcal T^{[\ge 2j]}_{A_{\spl,T}/T,\bar0}
\longrightarrow
\mathcal E_{j,T}.
\]
Moreover there is a canonical fibre sequence
\[
\widehat{\mathfrak U}^{[j+1]}_{F/S}
\longrightarrow
\widehat{\mathfrak U}^{[j]}_{F/S}
\xrightarrow{\ \widetilde\Omega_j^{\mathrm{for}}\ }
\mathbb V^{\mathrm{for}}_{F/S}\!\bigl(\mathcal E_j[1]\bigr).
\]

If the derived base-change assumption holds for the relevant tails and for
\(\mathcal E_j\), then the target can be written as
\[
\mathbb V^{\mathrm{for}}_S\!\bigl(R\pi_*\mathcal E_j[1]\bigr)
=
\Def_S\!\bigl(R\pi_*\mathcal E_j\bigr),
\]
and the displayed fibre sequence is a fibre sequence of relative formal moduli problems
over \(S\).
\end{theorem}

\begin{proof}
The Green filtration on
\(\mathcal T^{[\ge2]}_{A_{\spl}/S,\bar0}\) is a finite filtration by Lie ideals.  For
each \(j\), the quotient of consecutive tails is
\[
\mathcal T^{[\ge 2j]}_{A_{\spl}/S,\bar0}
\big/
\mathcal T^{[\ge 2(j+1)]}_{A_{\spl}/S,\bar0}
\cong
\mathcal T^{\langle 2j\rangle}_{A_{\spl}/S}
=
\mathcal E_j.
\]
This quotient is abelian: the bracket of two derivations of weight \(2j\) has weight
\(4j\), which lies in the next tail \(\mathcal T^{[\ge 2(j+1)]}_{A_{\spl}/S,\bar0}\) for
\(j\ge1\).  Hence, after pulling back to any test space \(T\to S\), one obtains a short
exact sequence of nilpotent dg Lie sheaves, or equivalently of their Dolbeault
resolutions,
\[
0
\longrightarrow
\mathcal T^{[\ge 2(j+1)]}_{A_{\spl,T}/T,\bar0}
\longrightarrow
\mathcal T^{[\ge 2j]}_{A_{\spl,T}/T,\bar0}
\longrightarrow
\mathcal E_{j,T}
\longrightarrow
0,
\]
with abelian quotient.  Applying derived global sections on \(Y_T\) gives a fibre
sequence of dg Lie algebras.  The Maurer--Cartan formal moduli functor sends such a
sequence with abelian quotient to the displayed fibre sequence of pointed formal moduli
problems.  The final relative formulation follows by applying the stated derived
base-change hypothesis.
\end{proof}

\begin{definition}[Projected and residual formal components]
\label{def:projected-residual-formal-components}
The composition of \(\widetilde\Omega_j^{\mathrm{for}}\) with the morphism induced
fibrewise by
\(
\rho_{2j}:\mathcal E_j\rightarrow \mathcal Q^{(2j)}_{F/S}
\)
is denoted
\[
\Def_{2j}^{\mathrm{for}}:
\widehat{\mathfrak U}^{[j]}_{F/S}
\longrightarrow
\mathbb V^{\mathrm{for}}_{F/S}
\bigl(\mathcal Q^{(2j)}_{F/S}[1]\bigr)
\]
and is called the \emph{formal projected defect} in weight \(2j\).  Explicitly, over a
test space \(T\to S\), its target is the formal vector stack attached to the abelian
complex
\(
R\Gamma\bigl(Y_T,\mathcal Q^{(2j)}_{F_T/T}\bigr)[1].
\)

If derived base change holds for \(\mathcal Q^{(2j)}_{F/S}\), then the same target may be
written in relative pushforward notation as
\[
\mathbb V^{\mathrm{for}}_{F/S}
\bigl(\mathcal Q^{(2j)}_{F/S}[1]\bigr)
\simeq
\mathbb V^{\mathrm{for}}_S
\bigl(R\pi_*\mathcal Q^{(2j)}_{F/S}[1]\bigr).
\]

If \(2j+1\le n\), and if this projected defect vanishes at a point and a lift to the
homotopy fibre is chosen, the remaining component of
\(\widetilde\Omega_j^{\mathrm{for}}\) lies in the kernel term
\(
\mathcal Q^{(2j+1)}_{F/S}
\)
of \eqref{eq:E-j-Q-exact}.  This is the corresponding \emph{formal residual class}.
Fibrewise over \(T\to S\), it is a class controlled by
\(
R\Gamma\bigl(Y_T,\mathcal Q^{(2j+1)}_{F_T/T}\bigr)[1].
\)\\
If \(2j+1>n\), the kernel term is zero and there is no residual obstruction at this
weight.
\end{definition}

\begin{proposition}[Comparison with the classical Green obstructions]
\label{prop:formal-classical-comparison-Green-tower}
Let \(1\le j\le N\).  Fibrewise over every test space \(T\to S\), and on ordinary
Artin thickenings after passing to \(\pi_0\), the formal leading map
\[
\widetilde\Omega_{j,T}^{\mathrm{for}}:
\widehat{\mathfrak U}^{[j]}_{F/S}(T)
\longrightarrow
\mathbb V^{\mathrm{for}}_{F_T/T}\!\bigl(\mathcal E_{j,T}[1]\bigr)
\]
recovers the two successive classical Green obstruction steps as follows.

\begin{enumerate}
\item The projected component
\(
\Def_{2j,T}^{\mathrm{for}}
\)
is represented by the classical obstruction class
\[
\omega_{2j}(\mathcal A_T,\iota_T,\sigma^{(2j-1)}_T)
\in
H^1\bigl(Y_T,\mathcal Q^{(2j)}_{F_T/T}\bigr)
\]
to lifting a \((2j-1)\)-splitting to a \(2j\)-splitting.

\item If \(2j+1\le n\), and if this obstruction vanishes and a \(2j\)-splitting
\(\sigma_T^{(2j)}\) is chosen, the normalized kernel component of
\(\widetilde\Omega_{j,T}^{\mathrm{for}}\) is represented by the classical obstruction
class
\[
\omega_{2j+1}(\mathcal A_T,\iota_T,\sigma_T^{(2j)})
\in
H^1\bigl(Y_T,\mathcal Q^{(2j+1)}_{F_T/T}\bigr)
\]
to lifting the chosen \(2j\)-splitting to a \((2j+1)\)-splitting.
\end{enumerate}

If, in addition, the sheaves
\(
R^1\pi_*\mathcal Q^{(2j)}_{F/S},
\) and \(
R^1\pi_*\mathcal Q^{(2j+1)}_{F/S}
\)
are locally free and commute with base change, then these fibrewise obstruction classes
assemble into the classical relative obstruction morphisms
\[
\omega_{2j}^{\cl}:
\mathfrak U_{2j-1}^{\cl}
\longrightarrow
\mathbb V^{\cl}_S\!\bigl(R^1\pi_*\mathcal Q^{(2j)}_{F/S}\bigr),
\]
and, after vanishing of \(\omega_{2j}^{\cl}\) and choice of a \(2j\)-splitting,
\[
\omega_{2j+1}^{\cl}:
\mathfrak U_{2j}^{\cl}
\longrightarrow
\mathbb V^{\cl}_S\!\bigl(R^1\pi_*\mathcal Q^{(2j+1)}_{F/S}\bigr).
\]
\end{proposition}

\begin{proof}
The statement is fibrewise on \(S\) and local on \(Y\), so we may use the
\v{C}ech--Deligne model of Subsection~\ref{subsec:Cech-Deligne-GO}.  A point of
\(\widehat{\mathfrak U}^{[j]}_{F/S}(T)\) is represented by a Maurer--Cartan element whose
first possible nonzero Green term has weight \(2j\).  Its leading component is a
\v{C}ech--Dolbeault \(1\)-cocycle with values in
\(
\mathcal E_{j,T}
=
\mathcal T^{\langle 2j\rangle}_{A_{\spl,T}/T}.
\)
Projection along
\(
\rho_{2j}:\mathcal E_{j,T}\to \mathcal Q^{(2j)}_{F_T/T}
\)
gives a cocycle with values in \(\mathcal Q^{(2j)}_{F_T/T}\).  By the construction of
Proposition~\ref{prop:classical-relative-obstruction}, this is precisely the classical
cocycle measuring the obstruction to lifting a \((2j-1)\)-splitting to order \(2j\).
Changing the Maurer--Cartan representative by gauge changes the projected cocycle by a
coboundary, so the induced cohomology class is the classical obstruction class.

Assume now that \(2j+1\le n\).  If the projected class vanishes, one may choose a
\(2j\)-splitting.  In the \v{C}ech--Deligne model this is the same as choosing a gauge
normalization for which the leading \(\mathcal E_{j,T}\)-valued cocycle has zero image in
\(\mathcal Q^{(2j)}_{F_T/T}\).  The leading cocycle then takes values in the kernel of
\(\rho_{2j}\), namely \(\mathcal Q^{(2j+1)}_{F_T/T}\).  The resulting kernel-valued
cohomology class is, again by Proposition~\ref{prop:classical-relative-obstruction}, the
obstruction to lifting the chosen \(2j\)-splitting to order \(2j+1\).

The final assertion is exactly the usual base-change assembly of fibrewise cohomology
classes into morphisms to the vector bundles
\(R^1\pi_*\mathcal Q^{(2j)}_{F/S}\) and
\(R^1\pi_*\mathcal Q^{(2j+1)}_{F/S}\), under the stated local-freeness and base-change
hypotheses.
\end{proof}

\begin{remark}[What the formal tower does and does not say]
\label{rem:formal-tower-correction}
Theorem~\ref{thm:formal-Green-tower} is the formal analogue of Green's tower.  Its
successive fibrewise quotients are governed by
\(
R\Gamma (Y_T,\mathcal T^{\langle 2j\rangle}_{A_{\spl,T}/T} ),
\)
not directly by the individual Green obstruction sheaves
\(\mathcal Q^{(r)}_{F_T/T}\).  Under derived base change, this fibrewise quotient may be
written relatively as
\(
R\pi_*\mathcal T^{\langle 2j\rangle}_{A_{\spl}/S}.
\)

The individual Green obstruction sheaves are recovered from the exact sequence
\eqref{eq:E-j-Q-exact}: first by projecting to
\(\mathcal Q^{(2j)}_{F/S}\), and then, after vanishing of this projected obstruction and
a choice of \(2j\)-splitting, by taking the residual kernel class in
\(\mathcal Q^{(2j+1)}_{F/S}\).

This is why a formal derived obstruction tower exists canonically, while a tower whose
successive formal quotients are the individual classical Green sheaves is not canonical
without additional choices.
\end{remark}

\subsection{Low odd rank: linearity and completeness of the primary obstruction}
\label{subsec:low-odd-rank-linearity}

The formal Green tower has an immediate consequence in small odd dimension.  The
first possible nonlinear term in the Maurer--Cartan equation is the bracket of two
weight-two derivations.  This bracket has weight four.  Hence, if the odd rank is at
most three, there is no room for such a term.  In this range the fixed-retract formal
moduli problem is therefore linear: the whole deformation class is its first Green
coordinate.  This gives a useful test case for the preceding formalism and explains why
odd rank two examples are governed entirely by their primary obstruction class.

\begin{theorem}[Low odd rank linearity]
\label{thm:low-odd-rank-linearity}
Let \(\pi:Y\to S\) be as above, let \(F\) be locally free of rank \(n\leq 3\), and set
\(A_{\spl}=\wedge^\bullet F\).  Then
\(
\mathcal T^{[\ge 4]}_{A_{\spl}/S,\bar 0}=0.
\)
Consequently the fibrewise Green controller is abelian and the split formal
neighbourhood is linear:
\[
\widehat\Tw_{F/S}(T,B)
\simeq
\MC_\infty\!\left(
R\Gamma\bigl(Y_T,\mathcal T^{\langle 2\rangle}_{A_{\spl,T}/T}\bigr)
\otimes\mathfrak m_B
\right),
\]
with the right-hand side regarded as the formal vector stack attached to the abelian
complex
\(
R\Gamma\bigl(Y_T,\mathcal T^{\langle 2\rangle}_{A_{\spl,T}/T}\bigr)[1].
\)

If derived base change holds for
\(\mathcal T^{\langle 2\rangle}_{A_{\spl}/S}\), then this fibrewise equivalence may be
written relatively as
\[
\widehat\Tw_{F/S}
\simeq
\mathbb V^{\mathrm{for}}_S\!\left(
R\pi_*\mathcal T^{\langle 2\rangle}_{A_{\spl}/S}[1]
\right).
\]
In particular, in odd rank at most three, the fixed-retract formal moduli problem has no
nonlinear Kuranishi equations.
\end{theorem}

\begin{proof}
A homogeneous even derivation of Green weight \(2j\) is determined by its two components
\[
\mathcal O_Y\longrightarrow \wedge^{2j}F,
\qquad
F\longrightarrow \wedge^{2j+1}F.
\]
If \(\operatorname{rk}F\leq 3\), then \(\wedge^rF=0\) for every \(r\geq 4\).  Hence no
nonzero homogeneous even derivation of weight at least four exists, and therefore
\(
\mathcal T^{[\ge 2]}_{A_{\spl}/S,\bar0}
=
\mathcal T^{\langle 2\rangle}_{A_{\spl}/S}.
\)
The bracket of two weight-two derivations has weight four.  Since the weight-four part
vanishes, the bracket on \(\mathcal T^{\langle 2\rangle}_{A_{\spl}/S}\) is zero.
Therefore the fibrewise dg Lie controller is abelian.  The formal moduli problem attached
to an abelian dg Lie algebra \(K\) is the formal vector stack with tangent complex \(K[1]\).
The displayed relative expression follows when
\(\mathcal T^{\langle 2\rangle}_{A_{\spl}/S}\) satisfies derived base change.
\end{proof}

\begin{corollary}[Odd rank two]
\label{cor:odd-rank-two-primary-complete}
Assume \(\operatorname{rk}F=2\).  Then
\[
\mathcal T^{\langle 2\rangle}_{A_{\spl}/S}
\cong
\mathcal Q^{(2)}_{F/S}
=
\mathcal T_{Y/S}\otimes \wedge^2F.
\]
Consequently, fibrewise over \(T\to S\),
\[
\widehat\Tw_{F/S}(T,B)
\simeq
\MC_\infty\!\left(
R\Gamma\bigl(Y_T,\mathcal T_{Y_T/T}\otimes\wedge^2F_T\bigr)
\otimes\mathfrak m_B
\right).
\]
If \(\mathcal T_{Y/S}\otimes\wedge^2F\) satisfies derived base change, this may be
written relatively as
\[
\widehat\Tw_{F/S}
\simeq
\mathbb V^{\mathrm{for}}_S\!\left(
R\pi_*(\mathcal T_{Y/S}\otimes \wedge^2F)[1]
\right).
\]

In particular, if \(S=\operatorname{Spec}\mathbb C\), then the set of fixed-retract
isomorphism classes in the split formal neighbourhood is
\(
H^1\bigl(Y,\mathcal T_Y\otimes\wedge^2F\bigr),
\)
and the class of an odd rank two supermanifold \(X\) is precisely its primary Green
obstruction
\[
\omega_2(X)
\in
H^1\bigl(Y,\mathcal T_Y\otimes\wedge^2F\bigr).
\]
Thus, in odd rank two, the primary obstruction is complete: it vanishes if and only if
\(X\) is split within the fixed retract.
\end{corollary}

\begin{proof}
If \(\operatorname{rk}F=2\), then \(\wedge^3F=0\).  The exact sequence
\eqref{eq:E-j-Q-exact} for \(j=1\) becomes
\[
0
\longrightarrow
0
\longrightarrow
\mathcal T^{\langle 2\rangle}_{A_{\spl}/S}
\longrightarrow
\mathcal T_{Y/S}\otimes\wedge^2F
\longrightarrow
0.
\]
The fibrewise formal equivalence follows from
Theorem~\ref{thm:low-odd-rank-linearity}.  The relative vector-stack expression follows
under the stated derived base-change hypothesis.

In the absolute case, the group of connected components of the formal vector stack
associated to
\(
R\Gamma(Y,\mathcal T_Y\otimes\wedge^2F)[1]
\)
is
\(
H^1(Y,\mathcal T_Y\otimes\wedge^2F).
\)
The comparison with the classical Green obstruction is
Proposition~\ref{prop:formal-classical-comparison-Green-tower} for \(j=1\).  Vanishing
of the class means that the corresponding point of the fixed-retract formal moduli problem
is the split point.
\end{proof}

\begin{corollary}[Odd rank three]
\label{cor:odd-rank-three-projected-residual}
Assume \(\operatorname{rk}F=3\).  Then the fixed-retract formal moduli problem is still
linear.  Fibrewise over \(T\to S\),
\[
\widehat\Tw_{F/S}(T,B)
\simeq
\MC_\infty\!\left(
R\Gamma\bigl(Y_T,\mathcal T^{\langle 2\rangle}_{A_{\spl,T}/T}\bigr)
\otimes\mathfrak m_B
\right).
\]
If \(\mathcal T^{\langle 2\rangle}_{A_{\spl}/S}\) satisfies derived base change, this may
be written relatively as
\[
\widehat\Tw_{F/S}
\simeq
\mathbb V^{\mathrm{for}}_S\!\left(
R\pi_*\mathcal T^{\langle 2\rangle}_{A_{\spl}/S}[1]
\right).
\]

The first Green coordinate sits in the exact sequence
\begin{equation}
\label{eq:rank-three-first-Green-coordinate}
0
\longrightarrow
\mathcal H\!om(F,\wedge^3F)
\longrightarrow
\mathcal T^{\langle 2\rangle}_{A_{\spl}/S}
\longrightarrow
\mathcal T_{Y/S}\otimes\wedge^2F
\longrightarrow
0.
\end{equation}
For an absolute fixed-retract supermanifold \(X\) over \((Y,F)\), let
\(
\widetilde\omega_2(X)
\in
H^1 (Y,\mathcal T^{\langle 2\rangle}_{A_{\spl}} )
\)
denote its linear Green class.  Its image
\[
\omega_2^{\proj}(X)
\in
H^1\bigl(Y,\mathcal T_Y\otimes\wedge^2F\bigr)
\]
is the primary projectedness obstruction.  If \(\omega_2^{\proj}(X)=0\), then the
remaining class is canonically represented by an element
\[
\omega_3^{\res}(X)
\in
\operatorname{coker}\!\left(
H^0(Y,\mathcal T_Y\otimes\wedge^2F)
\xrightarrow{\ \delta\ }
H^1(Y,\mathcal H\!om(F,\wedge^3F))
\right),
\]
where \(\delta\) is the connecting morphism associated with
\eqref{eq:rank-three-first-Green-coordinate}.  This residual class measures the failure
of the projected object to be split.
\end{corollary}

\begin{proof}
The fibrewise linearity statement is Theorem~\ref{thm:low-odd-rank-linearity}.  The
relative vector-stack expression follows under the stated derived base-change hypothesis.
Since \(\operatorname{rk}F=3\), the exact sequence \eqref{eq:E-j-Q-exact} for \(j=1\)
becomes \eqref{eq:rank-three-first-Green-coordinate}.  The image of the linear Green
class in
\(
H^1(Y,\mathcal T_Y\otimes\wedge^2F)
\)
is the projected component of the first Green coordinate, hence the primary projectedness
obstruction by Proposition~\ref{prop:formal-classical-comparison-Green-tower}.

If this image vanishes, exactness of the long cohomology sequence associated with
\eqref{eq:rank-three-first-Green-coordinate} shows that
\(\widetilde\omega_2(X)\) lies in the image of
\[
H^1(Y,\mathcal H\!om(F,\wedge^3F))
\longrightarrow
H^1(Y,\mathcal T^{\langle 2\rangle}_{A_{\spl}}).
\]
The ambiguity in choosing such a preimage is exactly the image of the connecting map
\[
\delta:
H^0(Y,\mathcal T_Y\otimes\wedge^2F)
\longrightarrow
H^1(Y,\mathcal H\!om(F,\wedge^3F)).
\]
Thus the residual class is a canonical element of the displayed cokernel.  It vanishes
precisely when the first Green class itself is zero, after the projected component has
been normalized away; equivalently, the projected fixed-retract object is split.
\end{proof}

\begin{remark}[The first nonlinear threshold]
\label{rem:first-nonlinear-threshold}
The preceding results separate the low-rank cases from the genuinely nonlinear range.
For \(\operatorname{rk}F=2\), the primary obstruction is the whole formal invariant.  For
\(\operatorname{rk}F=3\), the formal theory is still linear, but projectedness and
splitness separate through the exact sequence
\eqref{eq:rank-three-first-Green-coordinate}.  Nonlinear Green--Kuranishi relations can
appear only from odd rank four onward, where the bracket of two weight-two derivations may
have a nonzero weight-four component.
\end{remark}

\begin{problem}[The first nonlinear Green relation]
\label{prob:first-nonlinear-Green-relation}
Assume \(\operatorname{rk}F\geq 4\).  Describe explicitly the first nonlinear Kuranishi
relation in the Green tower.

\smallskip

More precisely, fibrewise over a test space \(T\to S\), let
\(
\alpha^{\langle2\rangle}
\in
\left(
R\Gamma\bigl(Y_T,\mathcal T^{\langle 2\rangle}_{A_{\spl,T}/T}\bigr)
\right)^1
\)
be the degree-one weight-two component of a Maurer--Cartan representative, after
undoing the formal shift. Then the first quadratic curvature term determines a degree-two
cohomology class
\[
\left[
\frac{1}{2}
[\alpha^{\langle2\rangle},\alpha^{\langle2\rangle}]
\right]
\in
H^2\!\left(
Y_T,
\mathcal T^{\langle 4\rangle}_{A_{\spl,T}/T}
\right).
\]
Under derived base change for
\(\mathcal T^{\langle4\rangle}_{A_{\spl}/S}\), these fibrewise classes may be written in
relative pushforward notation as classes in
\(
H^2\!\left(
T,
Lf^*R\pi_*\mathcal T^{\langle 4\rangle}_{A_{\spl}/S}
\right).
\)
In the absolute case \(S=\operatorname{Spec}\mathbb C\), this reduces to
\(
\left[
\frac{1}{2}
[\alpha^{\langle2\rangle},\alpha^{\langle2\rangle}]
\right]
\in
H^2\!\left(
Y,
\mathcal T^{\langle 4\rangle}_{A_{\spl}}
\right).
\)

Determine this class intrinsically and compute its projected component under
\[
\mathcal T^{\langle 4\rangle}_{A_{\spl,T}/T}
\longrightarrow
\mathcal Q^{(4)}_{F_T/T}
=
\mathcal T_{Y_T/T}\otimes\wedge^4F_T.
\]
Equivalently, compute its image in
\(
H^2\!(
Y_T,
\mathcal Q^{(4)}_{F_T/T} ),
\)
or, under derived base change for \(\mathcal Q^{(4)}_{F/S}\), in
\(
H^2\! (
T,
Lf^*R\pi_*\mathcal Q^{(4)}_{F/S}
).
\)
The projection is governed by the exact sequence
\[
0
\longrightarrow
\mathcal Q^{(5)}_{F/S}
\longrightarrow
\mathcal T^{\langle 4\rangle}_{A_{\spl}/S}
\longrightarrow
\mathcal Q^{(4)}_{F/S}
\longrightarrow
0,
\]
where
\(
\mathcal Q^{(5)}_{F/S}
=
\mathcal H\!om(F,\wedge^5F).
\)
If the projected degree-two class vanishes, describe the residual ambiguity governed
fibrewise by
\(
H^2 (
Y_T,
\mathcal Q^{(5)}_{F_T/T}
),
\)
or relatively, under derived base change, by
\(
H^2 (
T,
Lf^*R\pi_*\mathcal Q^{(5)}_{F/S}).
\)

\smallskip

Equivalently, one seeks an intrinsic description of the first Kuranishi compatibility
relation among degree-one classical Green coordinates.  This relation lives in degree two;
it is not itself a one-step Green obstruction class, which lives in degree one.
\end{problem}

\subsection{Local Kuranishi models}
\label{subsec:local-derived-kuranishi-algebraization}

The formal Green--Onishchik theorem \ref{thm:formal-GO} is a statement about formal moduli problems.  It
does not by itself produce an algebraic or analytic derived moduli space.  To obtain a
Kuranishi chart over an open set \(U\subset S\), we impose three additional requirements:
\begin{enumerate}
\item the relative controller is represented over \(U\) by a finite filtered dg Lie model;
\item this finite model admits filtered contraction data, so that homotopy transfer gives a
minimal filtered \(L_\infty\)-model on cohomology;
\item the resulting finite Maurer--Cartan formal moduli problem is algebraizable, or else
we are in the two-term situation where it is explicitly a derived zero locus.
\end{enumerate}
The derived zero locus \(\kappa_U^{-1}(0)\) is therefore used only in the two-term case.
In general the local object is the finite Maurer--Cartan derived stack associated with the
minimal filtered \(L_\infty\)-model, provided such a stack has been algebraized.

\begin{definition}[Finite filtered Kuranishi datum]
\label{def:finite-filtered-kuranishi-datum}
Let \(U\to S\) be an open, or analytic, chart on which \(\mathfrak g_{F/S}\) satisfies the
base-change convention above.  A \emph{finite filtered Kuranishi datum} for
\(\mathfrak g_{F/S}|_U\) consists of:
\begin{enumerate}
\item a bounded filtered dg Lie algebra \((E_U^\bullet,d,[\ ,\ ])\) of vector bundles on
\(U\), representing \(\mathfrak g_{F/S}|_U\) in the derived category of filtered Lie
algebra objects;

\item a graded vector bundle
\[
\mathcal H_U
\defeq
\bigoplus_i \mathcal H_U^i,
\qquad
\mathcal H_U^i\cong \mathcal H^i(\mathfrak g_{F/S}|_U),
\]
endowed with the induced finite Green filtration;

\item filtered strong deformation retract data
\[
(\mathcal H_U,0)
\;\substack{\xrightarrow{\ i\ }\\[-0.6ex]\xleftarrow[\ p\ ]{}}\;
(E_U^\bullet,d),
\qquad
\id_{E_U}-i\circ p=d h+h d,
\]
where \(i,p,h\) preserve the Green filtration.
\end{enumerate}
Applying homotopy transfer to these data gives a minimal filtered \(L_\infty\)-algebra
\(
(\mathcal H_U,\{\ell_m\}_{m\ge2})
\)
and a filtered \(L_\infty\)-quasi-isomorphism
\(
(\mathcal H_U,\{\ell_m\}_{m\ge2})\xrightarrow{\ \sim\ }\mathfrak g_{F/S}|_U.
\)
\end{definition}

\begin{remark}[Why the datum is an assumption]
\label{rem:kuranishi-datum-assumption}
Perfectness of \(\mathfrak g_{F/S}|_U\) and local freeness of its cohomology sheaves often
allow one, after shrinking \(U\), to represent the underlying complex by a bounded complex
of vector bundles and to split it locally.  However, the arguments below require a
filtered dg Lie model and filtered contraction data.  It is therefore cleaner to make the
finite filtered Kuranishi datum explicit.  Different choices give equivalent formal moduli
problems, but they need not give canonically isomorphic algebraic Kuranishi charts.  This
is the source of the descent problem treated below.
\end{remark}

\begin{proposition}[Local minimal filtered model]
\label{prop:local-minimal-filtered-model}
Assume that a finite filtered Kuranishi datum for \(\mathfrak g_{F/S}|_U\) is given.  Then:
\begin{enumerate}
\item homotopy transfer produces a minimal filtered \(L_\infty\)-algebra
\(
(\mathcal H_U,\{\ell_m\}_{m\ge2}),\)
with \( \ell_1=0;
\)

\item each bracket \(\ell_m\) is an algebraic multilinear morphism of vector bundles on
\(U\) and preserves the Green filtration;

\item because the Green filtration is finite, the Maurer--Cartan series on degree-one
inputs is finite;

\item the pointed formal moduli problem attached to this minimal filtered
\(L_\infty\)-algebra is canonically equivalent to the restriction of the formal
Green--Onishchik enhancement:
\[
\Def_U(\mathcal H_U,\{\ell_m\})
\simeq
\widehat\Tw_{F/S}|_U.
\]
\end{enumerate}
\end{proposition}

\begin{proof}
The homotopy transfer theorem applied to the filtered strong deformation retract in
Definition~\ref{def:finite-filtered-kuranishi-datum} gives a minimal
\(L_\infty\)-structure on \(\mathcal H_U\) and an \(L_\infty\)-quasi-isomorphism to the
dg Lie algebra \(E_U^\bullet\).  Since the contraction data and the original Lie bracket
preserve the Green filtration, the transferred brackets preserve the induced filtration.
The brackets are algebraic morphisms of vector bundles because the transfer formulas are
finite sums of compositions of the algebraic maps \(i,p,h,d\) and the bracket on
\(E_U^\bullet\).

The Green filtration is finite because \(A_{\spl}=\wedge^\bullet F\) has nilpotent odd
ideal.  Therefore only finitely many brackets can contribute to the Maurer--Cartan
series on degree-one inputs.  Finally, the Maurer--Cartan formal moduli problem is
invariant under filtered \(L_\infty\)-quasi-isomorphism.  Since \(E_U^\bullet\) represents
\(\mathfrak g_{F/S}|_U\), the formal moduli problem attached to the transferred minimal
model is \(\Def_U(\mathfrak g_{F/S}|_U)=\widehat\Tw_{F/S}|_U\).
\end{proof}

The preceding proposition is still formal.  To obtain an honest derived chart, one needs an
algebraic representative of the Maurer--Cartan functor of the finite minimal
\(L_\infty\)-algebra.  The following formulation separates the general finite
\(L_\infty\)-case from the simpler two-term case. As we shall see, the general finite \(L_\infty\)-case relies on the assumption of local algebraizability
in the following sense.\footnote{We formulate this as an assumption rather than as a definition because, in the
present paper, we do not prove the existence of such algebraizations in general.  The
subsequent arguments use only the consequences listed below; whenever these conditions
are satisfied in a given geometric situation, the construction produces the corresponding
local derived Maurer--Cartan stack.}

\begin{assumption}[Algebraizable local Maurer--Cartan stack]
\label{ass:algebraizable-local-MC-stack}
Let \((\mathcal H_U,\{\ell_m\})\) be the minimal filtered \(L_\infty\)-algebra obtained from
a finite filtered Kuranishi datum.  We say that it is \emph{locally algebraizable} if
there exists a pointed derived Artin stack
\(
\mathfrak K(\mathcal H_U)\rightarrow U
\)
locally of finite presentation, with zero section \(0:U\to\mathfrak K(\mathcal H_U)\),
such that the formal completion of \(\mathfrak K(\mathcal H_U)\) along the zero section is
canonically equivalent to the Maurer--Cartan formal moduli problem of
\((\mathcal H_U,\{\ell_m\})\).
\end{assumption}

\begin{theorem}[Local derived Kuranishi chart]
\label{thm:local-derived-kuranishi-chart}
Assume that \(\mathfrak g_{F/S}|_U\) admits a finite filtered Kuranishi datum and that the
resulting minimal filtered \(L_\infty\)-algebra is locally algebraizable in the sense of
Assumption~\ref{ass:algebraizable-local-MC-stack}.  Set
\(
\mathfrak K_{F/U}
\defeq
\mathfrak K(\mathcal H_U).
\)
Then there is a canonical equivalence of pointed formal derived stacks
\[
\widehat{\mathfrak K_{F/U}}{}_{\,0}
\simeq
\widehat\Tw_{F/S}|_U.
\]
Thus \(\mathfrak K_{F/U}\) is a local derived algebraization of the formal
Green--Onishchik enhancement near the split section.
\end{theorem}

\begin{proof}
By Assumption~\ref{ass:algebraizable-local-MC-stack}, the formal completion of
\(\mathfrak K(\mathcal H_U)\) is the Maurer--Cartan formal moduli problem of the minimal
filtered \(L_\infty\)-algebra \((\mathcal H_U,\{\ell_m\})\).  By
Proposition~\ref{prop:local-minimal-filtered-model}, this formal moduli problem is
canonically equivalent to \(\widehat\Tw_{F/S}|_U\).  This proves the assertion.
\end{proof}

\begin{corollary}[Two-term local derived Kuranishi neighbourhood]
\label{cor:relative-local-derived-kuranishi}
Assume that \(\mathfrak g_{F/S}|_U\) admits a finite filtered Kuranishi datum and that the
minimal filtered model satisfies
\(
\mathcal H_U^0=0,\) and \(
\mathcal H_U^i=0\text{ for }i\ge3.
\)
Equivalently, the only nonzero deformation and obstruction bundles are
\(\mathcal H_U^1\) and \(\mathcal H_U^2\).  Define
\[
\kappa_U:
\mathbb V_U(\mathcal H_U^1)
\longrightarrow
\mathbb V_U(\mathcal H_U^2)
\]
by
\[
\kappa_U(x)
\defeq
\sum_{m\ge2}\frac{1}{m!}\,\ell_m(x,\ldots,x).
\]
Because the Green filtration is finite, this is a polynomial morphism of vector bundles.
Let
\[
\mathfrak K_{F/U}
\defeq
\mathbb V_U(\mathcal H_U^1)
\times^{h}_{\mathbb V_U(\mathcal H_U^2)}
U,
\]
where \(U\to\mathbb V_U(\mathcal H_U^2)\) is the zero section.  Then:
\begin{enumerate}
\item \(\mathfrak K_{F/U}\) is a quasi-smooth derived Artin stack over \(U\);

\item its formal completion along the zero section is canonically equivalent to the formal
Green--Onishchik enhancement over \(U\):
\[
\widehat{\mathfrak K_{F/U}}{}_{\,0}
\simeq
\widehat\Tw_{F/S}|_U;
\]

\item the truncation \(t_0(\mathfrak K_{F/U})\) carries the canonical perfect obstruction
theory, in the sense of Behrend--Fantechi \cite{BehrendFantechiIntrinsicNormalCone}, induced by the quasi-smooth derived extension
\[
t_0(\mathfrak K_{F/U})\hookrightarrow\mathfrak K_{F/U}.
\]
\end{enumerate}
\end{corollary}

\begin{proof}
Under the stated vanishing assumptions, the Maurer--Cartan equation of the minimal model
has degree two and is exactly the equation \(\kappa_U(x)=0\) for
\(x\in\mathcal H_U^1\).  Since \(\mathcal H_U^0=0\), there is no infinitesimal gauge group
to quotient by.  Since \(\mathcal H_U^i=0\) for \(i\ge3\), there are no higher obstruction
directions beyond \(\mathcal H_U^2\).  Therefore the Maurer--Cartan derived stack is the
derived zero locus of \(\kappa_U\).

The derived zero locus of a morphism between smooth vector bundle stacks is
quasi-smooth.  Its formal completion at the zero section is the Maurer--Cartan formal
moduli problem of the minimal filtered \(L_\infty\)-model, hence is equivalent to
\(\widehat\Tw_{F/S}|_U\) by Proposition~\ref{prop:local-minimal-filtered-model}.  The
perfect obstruction theory on the truncation is the standard one associated with a
quasi-smooth derived enhancement, see \cite{BehrendFantechiIntrinsicNormalCone,SchurgToenVezzosiDerived}.
\end{proof}

\begin{remark}[Automorphisms and higher obstruction spaces]
\label{rem:automorphisms-higher-obstructions}
The hypotheses in Corollary~\ref{cor:relative-local-derived-kuranishi} are not cosmetic.
If \(\mathcal H_U^0\neq0\), then the local model has infinitesimal automorphisms and the
correct object is a derived quotient, or more intrinsically the Maurer--Cartan derived
stack of the full minimal \(L_\infty\)-algebra.  If \(\mathcal H_U^i\neq0\) for some
\(i\ge3\), then a single map \(\mathcal H_U^1\to\mathcal H_U^2\) cannot encode the full
formal moduli problem.  In that case one must retain the complete finite
\(L_\infty\)-model, or equivalently its Chevalley--Eilenberg algebra, rather than replacing
it by a two-term Kuranishi equation.
\end{remark}

\begin{corollary}[Absolute finite Kuranishi specialization]
\label{cor:absolute-derived-Kuranishi-enhancement}
Let
\[
S=\Spec\mathbb C,
\qquad
Y=\Xred,
\qquad
F=\FF_X,
\]
and assume that \(\Xred\) is compact.  Let \(H_X\) be a finite minimal filtered
\(L_\infty\)-model of the absolute controller \(L_X\).  Then the Maurer--Cartan formal
moduli problem of \(H_X\) is canonically equivalent to \(\Def_X\).

If, moreover,
\(
H_X^0=0,
\) and \(
H_X^i=0\quad\text{for }i\ge3,
\)
then the absolute formal moduli problem is algebraized by the derived affine Kuranishi
scheme
\[
\mathfrak K_X
\defeq
H_X^1\times^h_{H_X^2}\{0\},
\qquad
\kappa_X(x)=\sum_{m\ge2}\frac{1}{m!}\ell_m(x,\ldots,x),
\]
where \(H_X^1\) and \(H_X^2\) are finite-dimensional vector spaces.  Its formal completion
at the origin is \(\Def_X\), and its truncation carries the canonical perfect obstruction
theory.  The minimal Maurer--Cartan coordinate vector associated with \(X\), after choosing
the minimal Kuranishi model, satisfies the Kuranishi equation and therefore defines a
\(t_0\)-point of \(\mathfrak K_X\).
\end{corollary}

\begin{proof}
By Corollary~\ref{cor:absolute-formal-GO}, the absolute formal Green--Onishchik problem
is \(\Def_X=\Def(L_X)\).  Since \(H_X\) is a minimal filtered
\(L_\infty\)-model quasi-isomorphic to \(L_X\), invariance of the Maurer--Cartan formal
moduli problem gives \(\Def(H_X)\simeq\Def_X\).

Compactness of \(\Xred\) and coherence of the Green coefficient sheaves imply finite
dimensionality of the cohomology groups appearing in \(H_X\).  Under the two-term
hypothesis, the Maurer--Cartan equation is the finite polynomial equation
\(\kappa_X(x)=0\) between finite-dimensional affine spaces.  Hence the derived zero locus
\(H_X^1\times^h_{H_X^2}\{0\}\) is a derived affine scheme of finite type, and its formal
completion at the origin is the Maurer--Cartan formal moduli problem of \(H_X\), hence
\(\Def_X\).  Quasi-smoothness and the induced perfect obstruction theory follow from the
fact that \(\mathfrak K_X\) is a derived zero locus between smooth affine schemes.  Finally,
the Kuranishi compatibility relations from the minimal model say precisely that the chosen
minimal Maurer--Cartan coordinate vector satisfies the Maurer--Cartan equation.
\end{proof}

\subsection{Global Kuranishi data and sufficient algebraization criteria}
\label{subsec:globalization-natural-hypotheses}

The local Kuranishi construction depends on choices: a finite dg Lie model, a filtered
contraction, and, in the general case, an algebraization of the associated finite
Maurer--Cartan functor.  The formal completion is independent of these choices, because it
is always the intrinsic formal Green--Onishchik enhancement \(\widehat\Tw_{F/S}\).  The
algebraic or analytic derived charts themselves, however, need not glue automatically.
This subsection records natural hypotheses under which they do glue.  We warn the reader that these statements are
intended as positive criteria, not as a general solution to the algebraization problem.

\begin{theorem}[Sufficient criterion: global derived Kuranishi stack from a global finite model]
\label{thm:global-derived-kuranishi-from-contraction}
Assume that \(\mathfrak g_{F/S}=R\pi_*\mathfrak t_{F/S}\) is perfect, satisfies
Assumption~\ref{ass:green-derived-base-change} on \(S\), and admits a finite filtered
Kuranishi datum globally on \(S\), with transferred minimal
filtered \(L_\infty\)-algebra
\(
(\mathcal H_S,\{\ell_m\}_{m\ge2}).
\)
Assume moreover that this finite minimal model is algebraizable by a pointed derived Artin
stack
\(
\mathfrak K(\mathcal H_S)\longrightarrow S
\)
in the sense of Assumption~\ref{ass:algebraizable-local-MC-stack}.  

Then
\(
\Tw^{\der}_{F/S}
=
\mathfrak K(\mathcal H_S)
\)
is a derived Artin stack over \(S\), equipped with a split section
\(s_{\spl}:S\to\Tw^{\der}_{F/S}\), and there is a canonical equivalence of pointed formal
stacks
\[
\widehat{\Tw^{\der}_{F/S}}{}_{\,s_{\spl}}
\simeq
\widehat\Tw_{F/S}.
\]

If in addition
\(
\mathcal H_S^0=0,
\) and \(
\mathcal H_S^i=0\text{ for }i\ge3,
\)
then \(\Tw^{\der}_{F/S}\) may be taken to be the global derived zero locus
\(
\mathbb V_S(\mathcal H_S^1)
\times^h_{\mathbb V_S(\mathcal H_S^2)}
S
\)
of the global Kuranishi map
\[
\kappa_S(x)=\sum_{m\ge2}\frac{1}{m!}\ell_m(x,\ldots,x).
\]
In this two-term case \(\Tw^{\der}_{F/S}\) is quasi-smooth near the split section, and
its truncation carries the induced perfect obstruction theory.
\end{theorem}

\begin{proof}
The proof is the global version of Theorem~\ref{thm:local-derived-kuranishi-chart}.  The
formal completion of \(\mathfrak K(\mathcal H_S)\) is, by assumption, the
Maurer--Cartan formal moduli problem of \((\mathcal H_S,\{\ell_m\})\).  By homotopy
transfer and invariance under filtered \(L_\infty\)-quasi-isomorphisms, this formal
moduli problem is \(\Def_S(\mathfrak g_{F/S})\), hence
\(\widehat\Tw_{F/S}\) by the base-change hypothesis.  This gives the formal
identification.

Under the two-term hypothesis there are no infinitesimal automorphisms and no obstruction
spaces above degree two.  Thus the Maurer--Cartan derived stack is represented by the
derived zero locus of \(\kappa_S\).  Such a derived zero locus is quasi-smooth, and the
perfect obstruction theory on its truncation is the canonical one associated with a
quasi-smooth derived enhancement.
\end{proof}

\begin{proposition}[Criterion for a global filtered contraction]
\label{prop:cohomological-global-filtered-contraction}
Let \((E^\bullet,d)\) be a bounded filtered complex of vector bundles on \(S\) representing
the underlying filtered complex of \(\mathfrak g_{F/S}\).  Assume that the differential has
locally constant rank, so that
\[
Z^i(E^\bullet)=\ker(d:E^i\to E^{i+1}),
\qquad
B^i(E^\bullet)=\operatorname{im}(d:E^{i-1}\to E^i)
\]
are vector bundles, and assume that the cohomology sheaves
\(\mathcal H^i(\mathfrak g_{F/S})\) are vector bundles.  If the short exact sequences
\[
0\longrightarrow B^i(E^\bullet)
\longrightarrow Z^i(E^\bullet)
\longrightarrow \mathcal H^i(\mathfrak g_{F/S})
\longrightarrow 0
\]
and
\[
0\longrightarrow Z^i(E^\bullet)
\longrightarrow E^i
\longrightarrow B^{i+1}(E^\bullet)
\longrightarrow 0
\]
split globally in the category of filtered vector bundles for every \(i\), then
\(\mathfrak g_{F/S}\) admits a global finite filtered Kuranishi datum.

A sufficient cohomological condition for these splittings is the vanishing of the
corresponding filtered extension classes, for instance
\[
H^1\!\left(S,
\mathcal H\!om^{\fil}(\mathcal H^i(\mathfrak g_{F/S}),B^i(E^\bullet))
\right)=0
\quad \text{
and}
\quad 
H^1\!\left(S,
\mathcal H\!om^{\fil}(B^{i+1}(E^\bullet),Z^i(E^\bullet))
\right)=0
\]
for every \(i\).
\end{proposition}

\begin{proof}
The displayed exact sequences express \(E^\bullet\) as the direct sum of its cohomology
and a contractible complex precisely when they split.  Choosing such splittings gives maps
\[
(\mathcal H_S,0)
\;\substack{\xrightarrow{\ i\ }\\[-0.6ex]\xleftarrow[\ p\ ]{}}\;
(E^\bullet,d)
\]
and a homotopy \(h\) satisfying \(\id_E-i\circ p=dh+hd\).  Because the splittings are
filtered, the maps \(i,p,h\) preserve the Green filtration.  Thus they form a global
finite filtered Kuranishi datum.  The final assertion is the standard description of the
extension classes of the two short exact sequences in the indicated \(H^1\)-groups.
\end{proof}

\begin{corollary}[Affine or Stein bases]
\label{cor:affine-stein-unconditional-gluing}
Assume the set-up of Proposition~\ref{prop:cohomological-global-filtered-contraction}.
Assume moreover that either:
\begin{enumerate}
\item \(S\) is affine in the algebraic setting; or
\item \(S\) is Stein in the complex-analytic setting.
\end{enumerate}
Then the cohomological vanishing conditions of
Proposition~\ref{prop:cohomological-global-filtered-contraction} hold for coherent
filtered \(\mathcal H\!om\)-sheaves.  Consequently \(\mathfrak g_{F/S}\) admits a global
finite filtered Kuranishi datum.  If the resulting minimal model is algebraizable, then
Theorem~\ref{thm:global-derived-kuranishi-from-contraction} applies.

In particular, under the additional two-term hypothesis
\(\mathcal H_S^0=0\) and \(\mathcal H_S^i=0\) for \(i\ge3\), the split formal
Green--Onishchik enhancement is algebraized over \(S\) by the global derived zero locus of
its Kuranishi map.
\end{corollary}

\begin{proof}
The filtered \(\mathcal H\!om\)-sheaves appearing in
Proposition~\ref{prop:cohomological-global-filtered-contraction} are coherent, because the
filtration is finite and all graded pieces are vector bundles.  On an affine scheme, higher
cohomology of quasi-coherent sheaves vanishes.  On a Stein analytic space, higher
cohomology of coherent analytic sheaves vanishes by Cartan's theorem~B.  Thus the
extension classes vanish and the required splittings exist.  The remaining conclusions are
those of Theorem~\ref{thm:global-derived-kuranishi-from-contraction}.
\end{proof}

\begin{theorem}[Compatible local Kuranishi data glue]
\label{thm:functorial-filtered-contractions-imply-gluing}
Let \(\{U_i\to S\}_{i\in I}\) be a cover stable under finite intersections.  Assume that
for every finite intersection
\[
U_I=U_{i_0}\times_S\cdots\times_S U_{i_p}
\]
there is a finite filtered Kuranishi datum for \(\mathfrak g_{F/S}|_{U_I}\), and assume
that these data are compatible with restriction in the following strict sense: whenever
\(J\supset I\), the datum on \(U_J\) is the restriction of the datum on \(U_I\).
Assume also that the associated finite minimal \(L_\infty\)-models are algebraized by
pointed derived Artin stacks \(\mathfrak K_{F/U_I}\), compatibly with restriction.

Then the derived stacks \(\mathfrak K_{F/U_i}\) form a strict descent datum and glue to a
derived Artin stack
\(
\Tw^{\der}_{F/S}\rightarrow S
\)
with a split section \(s_{\spl}:S\to\Tw^{\der}_{F/S}\).  Its formal completion is
canonically the formal Green--Onishchik enhancement:
\[
\widehat{\Tw^{\der}_{F/S}}{}_{\,s_{\spl}}
\simeq
\widehat\Tw_{F/S}.
\]
If the local charts are two-term derived zero loci as in
Corollary~\ref{cor:relative-local-derived-kuranishi}, then \(\Tw^{\der}_{F/S}\) is
quasi-smooth near the split section and its truncation carries the perfect obstruction
theory obtained by gluing the local ones.
\end{theorem}

\begin{proof}
Strict compatibility of the filtered contraction data implies compatibility of the
transferred \(L_\infty\)-brackets, because the homotopy transfer formulas are functorial
with respect to restriction of the data \(i,p,h\) and the bracket.  Hence the associated
Maurer--Cartan derived stacks, or in the two-term case the derived zero loci of the
Kuranishi maps, restrict identically on overlaps.  This gives a strict descent datum for
the local derived Artin stacks.  Since derived Artin stacks satisfy descent for the chosen
covering topology, the local charts glue to a derived Artin stack over \(S\).  The zero
sections are compatible and therefore glue to the split section.

The formal completion of each local chart is \(\widehat\Tw_{F/S}|_{U_i}\).  These
identifications are compatible on overlaps by construction, hence the formal completion of
the glued stack is \(\widehat\Tw_{F/S}\).  Quasi-smoothness and the associated perfect
obstruction theory are local properties in the two-term case and therefore glue.
\end{proof}

The results of this subsection give sufficient conditions for an algebraized derived
Green--Onishchik enhancement.  We insist that they do not assert that such an enhancement exists without
choices.  What is canonical without further hypotheses is the formal object
\(\widehat\Tw_{F/S}\).  A global derived stack algebraizing it is obtained either from a
global finite Kuranishi datum, or from local finite data whose algebraizations are
compatible on overlaps.  When these compatibility data are not available, one is left with
the descent and algebraization problem discussed in the next subsection.

\subsection{Descent and the remaining algebraization problem}
\label{subsec:general-descent-open-problems}

We finish the section by isolating the precise global issue which remains after the
local Kuranishi construction.  The formal Green--Onishchik enhancement
\(
\widehat\Tw_{F/S}
\)
is canonical.  By contrast, an algebraic or analytic derived stack whose completion is
\(\widehat\Tw_{F/S}\) is not canonical in general, and may not exist without additional
hypotheses.  Locally, one may choose finite filtered Kuranishi data and obtain derived
charts
\(
\mathfrak K_{F/U}\rightarrow U.
\)
Any two such charts have canonically equivalent formal completions, because both
complete to the restriction of the same intrinsic formal moduli problem.  The non-formal
question is whether these canonical formal equivalences can be algebraized and made
compatible on multiple overlaps.

The correct language for this question is descent for derived stacks; see
\cite{ToenVezzosiHAGII,LurieDAGV}.  Since descent is
homotopy-coherent rather than merely set-theoretic, we formulate the criterion in terms
of the \v{C}ech nerve of a cover.  In applications one often has strict pairwise
transition equivalences; the homotopy-coherent formulation below is the invariant one.

\begin{definition}[Local Kuranishi atlas]
\label{def:local-kuranishi-atlas}
A \emph{local Kuranishi atlas for the split formal Green--Onishchik problem} consists
of:
\begin{enumerate}
\item a cover \(\{U_i\to S\}_{i\in I}\) in the chosen analytic or algebraic topology;
\item for each \(i\), a pointed derived Artin stack
\(
0_i:U_i\longrightarrow \mathfrak K_i
\)
locally of finite presentation over \(U_i\);
\item a formal identification
\(
\eta_i:
\widehat{\mathfrak K_i}_{\,0_i}
\xrightarrow{\sim}
\widehat\Tw_{F/S}|_{U_i}
\)
of pointed formal moduli problems over \(U_i\).
\end{enumerate}
We say that the atlas is \emph{Kuranishi} if each \(\mathfrak K_i\) is obtained from a
finite filtered Kuranishi datum in the sense of
Definition~\ref{def:finite-filtered-kuranishi-datum} and
Theorem~\ref{thm:local-derived-kuranishi-chart}.  In the two-term case of
Corollary~\ref{cor:relative-local-derived-kuranishi}, the charts are the derived zero
loci of the local Kuranishi maps.
\end{definition}

\begin{proposition}[Canonical formal descent datum]
\label{prop:canonical-formal-descent-datum}
Let \(\{\mathfrak K_i,\eta_i\}\) be a local Kuranishi atlas.  For every pair
\(i,j\), set
\[
U_{ij}\defeq U_i\times_S U_j,
\qquad
\mathfrak K_{i,ij}\defeq \mathfrak K_i\times_{U_i}U_{ij},
\qquad
\mathfrak K_{j,ij}\defeq \mathfrak K_j\times_{U_j}U_{ij}.
\]
Then there is a canonical equivalence of pointed formal derived stacks over \(U_{ij}\),
\[
\widehat\phi_{ij}:
\widehat{\mathfrak K_{i,ij}}{}_{\,0_i}
\xrightarrow{\sim}
\widehat{\mathfrak K_{j,ij}}{}_{\,0_j},
\qquad
\widehat\phi_{ij}
\defeq
\eta_j^{-1}\circ\eta_i.
\]
On triple and higher overlaps these equivalences satisfy the canonical
homotopy-coherent cocycle identities.  Equivalently, the formal completions
\(\widehat{\mathfrak K_i}_{0_i}\) form a descent datum which is canonically identified
with the descent datum obtained by restricting the single formal stack
\(\widehat\Tw_{F/S}\) to the cover.
\end{proposition}

\begin{proof}
Restrict the formal identifications \(\eta_i\) and \(\eta_j\) to the overlap
\(U_{ij}\).  Their composition through the common target
\(\widehat\Tw_{F/S}|_{U_{ij}}\) gives \(\widehat\phi_{ij}\).  On a triple overlap
\(U_{ijk}\), the two composites
\(
\widehat\phi_{jk}\circ\widehat\phi_{ij}
\text{ and }
\widehat\phi_{ik}
\)
are both equal, up to the canonical associativity homotopy, to
\(\eta_k^{-1}\circ\eta_i\).  The same argument on higher intersections gives the full
homotopy-coherent \v{C}ech descent datum.  Since all identifications are obtained by
restriction from \(\widehat\Tw_{F/S}\), this formal descent datum is canonical.
\end{proof}

\begin{definition}[Algebraization of the formal descent datum]
\label{def:algebraization-formal-descent-datum}
Let \(\{\mathfrak K_i,\eta_i\}\) be a local Kuranishi atlas.  An
\emph{algebraization of the canonical formal descent datum} consists of a
homotopy-coherent descent datum for the derived Artin stacks \(\mathfrak K_i\) over the
\v{C}ech nerve of the cover \(\{U_i\to S\}\), with transition equivalences
\[
\phi_{ij}:
\mathfrak K_{i,ij}\xrightarrow{\sim}\mathfrak K_{j,ij}
\]
preserving the zero sections, together with higher coherence data on multiple overlaps,
such that after formal completion along the zero sections one recovers the canonical
formal descent datum of Proposition~\ref{prop:canonical-formal-descent-datum}.

Equivalently, in the strict situation, this means that the pairwise equivalences
\(\phi_{ij}\) satisfy the usual cocycle condition on triple overlaps and that
\(
\widehat{\phi}_{ij}=\eta_j^{-1}\circ\eta_i.
\)
The strict formulation is sufficient in many concrete examples, but the
homotopy-coherent formulation is the natural one for derived stacks.
\end{definition}

\begin{theorem}[Descent criterion for a global derived algebraization]
\label{thm:descent-criterion-global-enhancement}
Let \(\{\mathfrak K_i,\eta_i\}\) be a local Kuranishi atlas and assume that its canonical
formal descent datum admits an algebraization in the sense of
Definition~\ref{def:algebraization-formal-descent-datum}.  Then:
\begin{enumerate}
\item the local derived stacks \(\mathfrak K_i\) glue to a derived Artin stack
\(
\Tw^{\der}_{F/S}\rightarrow S,
\)
together with a section
\(
s_{\spl}:S\rightarrow \Tw^{\der}_{F/S};
\)

\item for every \(i\) there is an equivalence
\[
\Tw^{\der}_{F/S}\times_S U_i\simeq \mathfrak K_i
\]
compatible with the local zero sections;

\item the formal completion of \(\Tw^{\der}_{F/S}\) along \(s_{\spl}\) is canonically
equivalent to the intrinsic formal Green--Onishchik enhancement:
\[
\widehat{\Tw^{\der}_{F/S}}{}_{\,s_{\spl}}
\simeq
\widehat\Tw_{F/S};
\]

\item if the local charts \(\mathfrak K_i\) are quasi-smooth near their zero sections,
then \(\Tw^{\der}_{F/S}\) is quasi-smooth near \(s_{\spl}\).  In the algebraic setting,
the truncation \(t_0(\Tw^{\der}_{F/S})\) then carries the perfect obstruction theory
induced by this quasi-smooth derived enhancement, at least in a neighbourhood of the
split section.
\end{enumerate}
\end{theorem}

\begin{proof}
A homotopy-coherent descent datum for the derived Artin stacks \(\mathfrak K_i\) is, by
standard descent for derived Artin stacks, effective.  Therefore it glues to a derived
Artin stack \(\Tw^{\der}_{F/S}\) over \(S\).  Since the transition equivalences preserve
the zero sections and the coherence data are compatible with them, the local zero
sections glue to a global section \(s_{\spl}:S\to\Tw^{\der}_{F/S}\).  This proves (1) and
(2).

Formal completion along a closed section is compatible with restriction to the cover.
By assumption, the completed descent datum is precisely the canonical formal descent
datum of Proposition~\ref{prop:canonical-formal-descent-datum}.  Hence the formal
completion of the glued stack is obtained by gluing the restrictions of
\(\widehat\Tw_{F/S}\), which gives \(\widehat\Tw_{F/S}\) itself.  This proves (3).

Finally, quasi-smoothness is local for the chosen topology and is preserved under
equivalence.  Thus it descends from the charts to the glued stack near the split section.
In the algebraic setting, a quasi-smooth derived enhancement induces the usual perfect
obstruction theory on its truncation.  Since the construction is local and compatible
with descent, the local obstruction theories glue.  This proves (4).
\end{proof}

\begin{remark}[What the theorem does and does not assert]
\label{rem:descent-criterion-scope}
Theorem~\ref{thm:descent-criterion-global-enhancement} gives a global algebraization of
the \emph{formal neighbourhood of the split section}.  It does not assert that the whole
classical stack \(\Tw^{\cl}_{F/S}\) admits a canonical derived enhancement away from the
split point.  Nor does it assert that an algebraization of the formal descent datum always
exists.  A genuine derived stack \(\Tw^{\der}_{F/S}\) is obtained only after choosing local
algebraizations and algebraizing their formal transition maps.
\end{remark}

\begin{proposition}[Compatibility with the sufficient globalization criteria]
\label{prop:compatibility-previous-globalization}
The two globalization mechanisms of
Subsection~\ref{subsec:globalization-natural-hypotheses} produce algebraizations of the
canonical formal descent datum:
\begin{enumerate}
\item if a global finite filtered Kuranishi datum exists and its finite minimal model is
algebraizable, then the associated global derived Kuranishi stack restricts to a local
Kuranishi atlas whose formal descent datum is algebraized by the identity transitions;

\item if compatible local finite filtered Kuranishi data and compatible local
algebraizations exist on a cover stable under finite intersections, then the induced
strict descent datum is an algebraization of the canonical formal descent datum.
\end{enumerate}
\end{proposition}

\begin{proof}
In the first case the global derived stack restricts to the same object on all overlaps;
therefore the transition maps are identities after identifying the restrictions.  The
formal completion of the global object is \(\widehat\Tw_{F/S}\) by
Theorem~\ref{thm:global-derived-kuranishi-from-contraction}, so the induced formal
descent datum is the canonical one.

In the second case, compatibility of the filtered Kuranishi data implies compatibility of
the transferred \(L_\infty\)-structures and of the resulting algebraizations on all finite
intersections.  Hence the local derived stacks form a strict descent datum.  By
Theorem~\ref{thm:functorial-filtered-contractions-imply-gluing}, their formal
completions identify with the restrictions of \(\widehat\Tw_{F/S}\), and the transition
maps complete to the canonical formal overlap equivalences.
\end{proof}

\begin{remark}[Uniqueness holds only in the formal category]
\label{rem:formal-uniqueness-only}
Suppose two global derived algebraizations
\(\Tw^{\der}_{F/S}\) and \((\Tw^{\der}_{F/S})'\) are obtained from two different
algebraizations of the canonical formal descent datum.  Then their formal completions
along the split sections are canonically equivalent:
\[
\widehat{\Tw^{\der}_{F/S}}{}_{\,s_{\spl}}
\simeq
\widehat\Tw_{F/S}
\simeq
\widehat{(\Tw^{\der}_{F/S})'}{}_{\,s'_{\spl}}.
\]
However, the algebraizations themselves need not be canonically equivalent away from the
formal neighbourhood of the split section.  This is the usual loss of uniqueness when
passing from formal moduli problems to algebraic or analytic representatives -- see for instance
\cite{LurieDAGX,ToenVezzosiHAGII,CalaqueGrivauxFMP}.
\end{remark}

\begin{problem}[Algebraization of canonical formal overlaps]
\label{prob:algebraization-formal-overlaps}
Let \(\{\mathfrak K_i,\eta_i\}\) be a local Kuranishi atlas.  Determine natural
conditions under which the canonical formal equivalences
\[
\widehat\phi_{ij}:
\widehat{\mathfrak K_{i,ij}}{}_{\,0_i}
\xrightarrow{\sim}
\widehat{\mathfrak K_{j,ij}}{}_{\,0_j}
\]
are induced by actual equivalences of derived Artin stacks
\[
\phi_{ij}:\mathfrak K_{i,ij}\xrightarrow{\sim}\mathfrak K_{j,ij},
\]
and under which these algebraized equivalences can be chosen with homotopy-coherent
compatibilities on all higher overlaps.
\end{problem}

Problem~\ref{prob:algebraization-formal-overlaps} is the genuinely global obstruction
left by the preceding analysis.  In affine algebraic or Stein analytic situations, the
cohomological splitting criteria of
Proposition~\ref{prop:cohomological-global-filtered-contraction} often give global finite
Kuranishi data, and the descent problem disappears.  Likewise, if local contractions are
chosen functorially on all finite intersections, the strict gluing result of
Theorem~\ref{thm:functorial-filtered-contractions-imply-gluing} applies.

In the non-affine and non-Stein case, however, there need not be a global filtered
contraction, and independently chosen local contractions need not be compatible on
overlaps.  The formal geometry still glues canonically because it is controlled either
fibrewise by the dg Lie algebras \(R\Gamma(Y_T,\mathfrak t_{F_T/T})\), or, under derived
base change, by the single relative dg Lie object \(\mathfrak g_{F/S}\).  What remains
open, in this level of generality, is the algebraization of the canonical formal gluing
maps.

\begin{remark}[Summary \& Conclusions]
\label{rem:section8-conclusion}
The output of this section is therefore twofold.  First, the fixed-retract splitting
problem has a canonical formal derived enhancement,
\[
\widehat\Tw_{F/S}(T,B)
=
\MC_\infty\!\left(
R\Gamma(Y_T,\mathfrak t_{F_T/T})\otimes\mathfrak m_B
\right),
\]
and, in the absolute case, this is exactly the formal moduli problem \(\Def_X\)
controlled by the filtered dg Lie algebra \(L_X\).  Under derived base change it may be
written compactly as
\[
\widehat\Tw_{F/S}
\simeq
\Def_S(R\pi_*\mathfrak t_{F/S}).
\]
Second, algebraic or analytic derived Kuranishi spaces are available under explicit
finite-model, amplitude, base-change, and descent hypotheses.  Thus the formal theory is
intrinsic and unconditional, whereas its global algebraization is a separate geometric
problem.
\end{remark}


\section{Examples and geometric illustrations}
\label{sec:examples-geometric-illustrations}

The purpose of this section is to make the constructions of the previous sections concrete.
We discuss four examples, ordered by increasing complexity.  The first two are absolute
examples in odd rank three.  They illustrate the Green obstruction tower, the adapted
Maurer--Cartan leading classes, and the filtered affine Atiyah symbols in the abelian range.
The third example is relative and is meant to exhibit, in the simplest possible setting, the
Green--Onishchik formal moduli problem of
Section~\ref{sec:green-onishchik-fixed-retract-moduli}.  The last example is in odd rank
four and displays the first genuinely nonlinear Green--Kuranishi relation.

We keep two distinctions throughout.  First, an adapted leading cocycle is not, by itself,
the same thing as a holomorphic supermanifold.  It is the first nonzero term of a
Maurer--Cartan representative.  It defines an actual fixed-retract supermanifold structure
only when it can be completed to a full Maurer--Cartan element, equivalently to a
non-abelian \v{C}ech cocycle of automorphisms of the split model.  In the two odd-rank
three examples below this completion is automatic for weight reasons, since there is no
weight-four even tangent piece.  In the relative example, the object under discussion is
primarily a formal Green--Onishchik moduli problem, and its vector-bundle presentation is
obtained only after rigidifying the infinitesimal automorphism directions.  In the final
odd-rank four example, by contrast, a general first-order direction need not extend to a
full Maurer--Cartan element; the equation \(uv=0\) records precisely this first nonlinear
compatibility condition.

Second, a \v{C}ech cocycle with values in
\(
\mathcal T_{\widehat X,\bar 0}^{\langle 2j\rangle}
\)
is a tangent-valued adapted leading cocycle.  Its cohomology class is the adapted
leading class
\[
\beta_{X,j}(\sigma^{(2j-1)})
\in
H^1\bigl(\Xred,\mathcal T_{\widehat X,\bar 0}^{\langle 2j\rangle}\bigr).
\]
The corresponding filtered affine Atiyah symbol is represented by applying the
normalized pure odd Hessian construction to an adapted logarithmic \v{C}ech
representative of this class.  This is a representative-level construction: as already stressed above, it should
not be read as coming from a canonical sheaf morphism
\(
\mathcal T_{\widehat X}^{\langle 2j\rangle}
\rightarrow
\mathcal H_{\widehat X}^{\langle 2j\rangle}
\)
defined on the whole homogeneous tangent sheaf.  What is canonical, and what is used
below, are the projected comparison identities.

First, the even projection of the filtered affine Atiyah symbol recovers the projected
Green leading class:
\[
H^1(\Pi_{2j}^{\mathrm{ev}})
\bigl(
\mathfrak{At}^{\langle 2j\rangle}(X;\sigma^{(2j-1)})
\bigr)
=
H^1(\rho_{2j})
\bigl(
\beta_{X,j}(\sigma^{(2j-1)})
\bigr).
\]
Second, after this even obstruction has vanished and a \(2j\)-splitting
\(\sigma^{(2j)}\) has been fixed, the adapted representative can be normalized into
\(\ker(\rho_{2j})\).  The resulting residual affine Atiyah symbol
\(
\widetilde{\mathfrak{At}}^{\langle 2j\rangle}(X;\sigma^{(2j)})
\)
then satisfies
\[
H^1(\Pi_{2j}^{\mathrm{odd}})
\bigl(
\widetilde{\mathfrak{At}}^{\langle 2j\rangle}(X;\sigma^{(2j)})
\bigr)
=
\omega_X^{(2j+1)}(\sigma^{(2j)}).
\]
When \(2j+1>\rk(\FF_X)\), the residual odd obstruction is absent.

\subsection{Odd rank three over an elliptic curve}
\label{subsec:example-elliptic-rank-three}

Let \(E\) be a smooth elliptic curve and set
\[
\FF\defeq \mathcal O_E^{\oplus 3}.
\]
Let
\(
\widehat X=(E,\wedge^\bullet\FF)
\)
be the corresponding split model.  Since \(E\) is an elliptic curve, its tangent bundle is trivial:
\(
\mathcal T_E\cong \mathcal O_E.
\)
Moreover,
\[
\wedge^2\FF\cong \mathcal O_E^{\oplus 3},
\qquad
\wedge^3\FF\cong \mathcal O_E.
\]
Thus the first two Green obstruction sheaves are
\[
\mathcal Q_X^{(2)}
=
\mathcal T_E\otimes \wedge^2\FF
\cong
\mathcal O_E^{\oplus 3},
\quad \text{ and } \quad
\mathcal Q_X^{(3)}
=
\mathcal H\!om_{\mathcal O_E}(\FF,\wedge^3\FF)
\cong
\mathcal O_E^{\oplus 3}.
\]
The weight-two split tangent exact sequence is
\[
0
\longrightarrow
\mathcal Q_X^{(3)}
\longrightarrow
\mathcal T_{\widehat X,\bar0}^{\langle2\rangle}
\overset{\rho_2}{\longrightarrow}
\mathcal Q_X^{(2)}
\longrightarrow
0.
\]
In this trivialized example the sequence is globally split.

\smallskip

Fix a global frame \(e_1,e_2,e_3\) of \(\FF\), and let
\(
\theta_1,\theta_2,\theta_3
\)
be the corresponding odd coordinates.  Let
\(
v\in H^0(E,\mathcal T_E)
\)
be a nonzero translation-invariant vector field.  Then
\(
\mathcal T_{\widehat X,\bar0}^{\langle2\rangle}
\)
is generated by the quotient-type derivations
\[
D_{23}\defeq \theta_2\theta_3 v,
\qquad
D_{31}\defeq \theta_3\theta_1 v,
\qquad
D_{12}\defeq \theta_1\theta_2 v,
\]
and by the kernel-type derivations
\[
K_1\defeq \theta_1\theta_2\theta_3\partial_{\theta_1},
\qquad
K_2\defeq \theta_1\theta_2\theta_3\partial_{\theta_2},
\qquad
K_3\defeq \theta_1\theta_2\theta_3\partial_{\theta_3}.
\]
The map \(\rho_2\) sends the three \(D\)-terms to their images in
\(
\mathcal Q_X^{(2)}
\)
and sends all \(K_\alpha\) to zero.  Hence
\[
\mathcal T_{\widehat X,\bar0}^{\langle2\rangle}
\cong
\mathcal Q_X^{(2)}\oplus \mathcal Q_X^{(3)}
\cong
\mathcal O_E^{\oplus 6}.
\]
Since the odd rank is three, there is no weight-four even tangent piece:
\(
\mathcal T_{\widehat X,\bar0}^{\langle4\rangle}=0.
\)
Consequently the filtered dg Lie algebra governing the splitting problem is abelian after
passing to the only possible positive even weight.  In particular, the minimal filtered
\(L_\infty\)-model has no nonzero higher brackets.

\smallskip

Let \(\mathfrak U=\{U_i\}\) be a Stein cover of \(E\), and choose scalar \v{C}ech
\(1\)-cocycles
\[
\eta_{23},\eta_{31},\eta_{12},\xi^1,\xi^2,\xi^3
\in
\check Z^1(\mathfrak U,\mathcal O_E).
\]
Define the tangent-valued weight-two cocycle
\[
a_{ij}
\defeq
\eta_{23;ij}D_{23}+\eta_{31;ij}D_{31}+\eta_{12;ij}D_{12}
+
\xi^1_{ij}K_1+\xi^2_{ij}K_2+\xi^3_{ij}K_3.
\]
It represents the adapted leading class
\[
\beta_{X,1}\defeq[a]
\in
H^1\bigl(E,\mathcal T_{\widehat X,\bar0}^{\langle2\rangle}\bigr).
\]
Since there is no weight-four term, this cocycle is not merely first-order data.  It is
already a full Maurer--Cartan representative: the possible quadratic term
\(
\frac12[a,a]
\)
would have weight four and hence vanishes.  Therefore every such cocycle integrates to
transition functions
\(
\exp(a_{ij})=1+a_{ij}
\)
for an actual holomorphic supermanifold structure on the fixed smooth split model. In local even
coordinates in which \(v=\partial_z\), these transition functions have the form
\[
z_i
=
z_j
+
\eta_{23;ij}\theta_{j,2}\theta_{j,3}
+
\eta_{31;ij}\theta_{j,3}\theta_{j,1}
+
\eta_{12;ij}\theta_{j,1}\theta_{j,2},
\]
and
\[
\theta_{i,\alpha}
=
\theta_{j,\alpha}
+
\xi^\alpha_{ij}\theta_{j,1}\theta_{j,2}\theta_{j,3},
\qquad
\alpha=1,2,3.
\]
The cocycle condition is exactly the ordinary \v{C}ech cocycle condition for the six scalar
cocycles above.

\smallskip

The projected part of \(\beta_{X,1}\) is
\[
H^1(\rho_2)(\beta_{X,1})
=
\bigl([\eta_{23}],[\eta_{31}],[\eta_{12}]\bigr)
\in
H^1(E,\mathcal O_E)^{\oplus3}
\cong
H^1(E,\mathcal Q_X^{(2)}).
\]
This is the first Green obstruction class
\(
\omega_X^{(2)}
\).
If this class vanishes, then the quotient-type part may be normalized away.  More explicitly,
choose \(0\)-cochains
\(
b^{23},b^{31},b^{12}\in \check C^0(\mathfrak U,\mathcal O_E)
\)
such that
\[
\eta_{23}=\check\delta b^{23},
\qquad
\eta_{31}=\check\delta b^{31},
\qquad
\eta_{12}=\check\delta b^{12}.
\]
Then the normalized representative
\[
\widetilde a
=
a
-
\check\delta
\bigl(b^{23}D_{23}+b^{31}D_{31}+b^{12}D_{12}\bigr)
\]
is kernel-valued:
\[
\widetilde a_{ij}
=
\xi^1_{ij}K_1+\xi^2_{ij}K_2+\xi^3_{ij}K_3
\in
\check Z^1(\mathfrak U,\mathcal Q_X^{(3)}).
\]
The corresponding residual odd obstruction is
\[
\omega_X^{(3)}
=
[\widetilde a]
=
\bigl([\xi^1],[\xi^2],[\xi^3]\bigr)
\in
H^1(E,\mathcal Q_X^{(3)}).
\]
Thus a single weight-two tangent class contains both the even obstruction and, after the
even obstruction has vanished and a normalization has been chosen, the residual odd
obstruction.

There are two extreme cases.  If
\(
a_{ij}=\eta_{ij}D_{23}
\)
with \([\eta]\neq0\in H^1(E,\mathcal O_E)\), then
\[
\omega_X^{(2)}=[\eta]\otimes D_{23}\neq0,
\]
so no \(2\)-splitting exists and no residual odd obstruction is defined.  If instead
\(
a_{ij}=\xi_{ij}K_1
\)
with \([\xi]\neq0\), then
\[
\omega_X^{(2)}=0,
\qquad
\omega_X^{(3)}=[\xi]\otimes K_1\neq0.
\]
This is the simplest situation in which projectedness holds to order two, but the residual
odd obstruction is nonzero.

Finally, the minimal \(L_\infty\)-picture is completely linear.  The Kuranishi term
\(
\frac12[a,a]
\)
would have weight four, but the weight-four sheaf vanishes.  Hence
\(
\kappa_2(\beta_{X,1})=0,
\)
and all higher Kuranishi operations vanish for weight reasons.  This example therefore
separates the \(H^1\)-valued Green obstruction tower from the \(H^2\)-valued nonlinear
Kuranishi relations: in odd rank three, the former may be nontrivial, while the latter are
absent.

\subsection{Odd rank three over a K3 surface}
\label{subsec:example-k3-rank-three}

We now keep odd rank three but replace the reduced curve by a surface.  This makes the
cohomology of the obstruction sheaves richer, while the filtered Lie algebra remains abelian for
weight reasons.

Let \(S\) be an elliptic K3 surface with section class \(C\) and fibre class \(f\), so that
\[
C^2=-2,
\qquad
f^2=0,
\qquad
C\cdot f=1.
\]
Set
\[
L\defeq \mathcal O_S(C-2f),
\qquad
\FF\defeq \mathcal O_S\oplus L\oplus L^{-1}.
\]
Then
\[
\wedge^3\FF\cong \det \FF\cong \mathcal O_S,
\qquad
\wedge^2\FF\cong \mathcal O_S\oplus L\oplus L^{-1}.
\]
Consequently
\[
\mathcal Q_X^{(2)}
=
\mathcal T_S\otimes \wedge^2\FF
\cong
\mathcal T_S\otimes(\mathcal O_S\oplus L\oplus L^{-1}),
\]
and
\[
\mathcal Q_X^{(3)}
=
\mathcal H\!om_{\mathcal O_S}(\FF,\wedge^3\FF)
\cong
\FF^\vee
\cong
\mathcal O_S\oplus L^{-1}\oplus L.
\]

\smallskip

We compute the relevant cohomology of \(L\).  Let
\(
D=C-2f
\).
Then
\[
D^2
=
C^2-4(C\cdot f)+4f^2
=
-2-4
=
-6.
\]
By Riemann--Roch on a K3 surface,
\[
\chi(L)=\chi(\mathcal O_S)+\frac12D^2=2-3=-1.
\]
Assume, as for a standard elliptic K3 surface with section, that
\(
H=C+3f
\)
is ample.  Then
\[
H\cdot D=(C+3f)\cdot(C-2f)=-1,
\]
so \(D\) is not effective and \(H^0(S,L)=0\).  On the other hand,
\[
(-D)\cdot f=(2f-C)\cdot f=-1,
\]
and since \(f\) is nef, \(-D\) is not effective; hence \(H^0(S,L^{-1})=0\).  Using
\(K_S\cong\mathcal O_S\) and Serre duality, we get
\[
H^2(S,L)=H^0(S,L^{-1})^\vee=0,
\qquad
H^2(S,L^{-1})=H^0(S,L)^\vee=0.
\]
Therefore
\[
h^1(S,L)=1,
\qquad
h^1(S,L^{-1})=1.
\]
In particular,
\[
H^1(S,\mathcal Q_X^{(3)})
\cong
H^1(S,L^{-1})\oplus H^1(S,L)
\cong
\CC^2.
\]
The even obstruction space is also nonzero.  Indeed,
\(
H^1(S,\mathcal Q_X^{(2)})
\supset
H^1(S,\mathcal T_S),
\)
and on a K3 surface
\[
H^1(S,\mathcal T_S)
\cong
H^1(S,\Omega_S^1)
\cong
\CC^{20},
\]
where the first isomorphism is induced by a holomorphic symplectic form.

\smallskip

As in the elliptic example, odd rank three implies
\(
\mathcal T_{\widehat X,\bar0}^{\langle4\rangle}=0.
\)
Thus the weight-two exact sequence
\[
0
\longrightarrow
\mathcal Q_X^{(3)}
\longrightarrow
\mathcal T_{\widehat X,\bar0}^{\langle2\rangle}
\overset{\rho_2}{\longrightarrow}
\mathcal Q_X^{(2)}
\longrightarrow
0
\]
contains all possible positive even weights.  The minimal filtered \(L_\infty\)-model is
again abelian, despite the fact that \(S\) has nonzero \(H^2\)-cohomology.

\smallskip

We now build an explicit residual class.  Let \(\mathfrak U=\{U_i\}\) be a Stein cover on
which the three summands of \(\FF\) are trivialized.  Let
\(
\theta_{i,1},\theta_{i,2},\theta_{i,3}
\)
be the corresponding odd coordinates, with \(\theta_{i,2}\) a local frame for \(L\) and
\(\theta_{i,3}\) a local frame for \(L^{-1}\).  The local derivation
\[
K_{3,i}
\defeq
\theta_{i,1}\theta_{i,2}\theta_{i,3}\partial_{\theta_{i,3}}
\]
transforms as a local frame of the \(L\)-summand of
\(
\mathcal Q_X^{(3)}\cong\mathcal O_S\oplus L^{-1}\oplus L
\).
Choose a nonzero class
\(
\xi\in H^1(S,L),
\)
and represent it by local scalar \v{C}ech functions \(\xi_{ij}\) with respect to the frame
\(K_{3,j}\).  Then
\[
a_{ij}\defeq \xi_{ij}K_{3,j}
\in
\check Z^1(\mathfrak U,\mathcal Q_X^{(3)})
\subset
\check Z^1(\mathfrak U,\mathcal T_{\widehat X,\bar0}^{\langle2\rangle})
\]
defines an adapted leading class
\(
\beta_{X,1}\defeq[a]
\in
H^1(S,\mathcal T_{\widehat X,\bar0}^{\langle2\rangle}).
\)
As in the elliptic curve example, the odd rank is three, so there is no weight-four
even tangent piece.  Thus this class is not only a formal first-order direction: the
cocycle \(a\) satisfies the full Maurer--Cartan equation for weight reasons, and
\(
\exp(a_{ij})=1+a_{ij}
\)
defines an actual holomorphic supermanifold structure with fixed split model
\((S,\wedge^\bullet\FF)\). 
Since \(a\) is already kernel-valued,
\[
H^1(\rho_2)(\beta_{X,1})=0.
\]
Hence the even obstruction vanishes.  The normalized residual class is
\[
\omega_X^{(3)}=[a]
\in
H^1(S,\mathcal Q_X^{(3)}),
\]
and it is nonzero by construction. Thus this example realizes
\[
\omega_X^{(2)}=0,
\qquad
\omega_X^{(3)}\neq0,
\]
on a reduced surface.  Nevertheless, because the odd rank is three, there is still no nonlinear
Kuranishi relation: the obstruction tower is nontrivial, but the minimal filtered
\(L_\infty\)-model is abelian.

\subsection{A relative Green--Onishchik moduli example in odd rank three}
\label{subsec:example-relative-GO-rank-three}

We now give a simple relative example illustrating Section~\ref{sec:green-onishchik-fixed-retract-moduli}.
Let \(E\) be a fixed smooth elliptic curve, let \(S\) be a complex base, and consider the
trivial smooth proper family
\[
\pi:Y=E\times S\longrightarrow S.
\]
Let
\[
F\defeq\mathcal O_Y^{\oplus3},
\qquad
A_{\spl}\defeq \wedge^\bullet F.
\]
The relative tangent sheaf is
\(
\mathcal T_{Y/S}\cong \mathcal O_Y,
\)
and therefore
\[
\mathcal Q^{(2)}_{F/S}
=
\mathcal T_{Y/S}\otimes\wedge^2F
\cong
\mathcal O_Y^{\oplus3},
\]
and similarly
\[
\mathcal Q^{(3)}_{F/S}
=
\mathcal H\!om_{\mathcal O_Y}(F,\wedge^3F)
\cong
\mathcal O_Y^{\oplus3}.
\]
Since \(F\) has rank three,
\(
\mathcal T^{[\ge4]}_{A_{\spl}/S,\bar0}=0.
\)
Thus the relative Green controller has only its weight-two piece, and that piece is
\[
\mathcal T^{\langle2\rangle}_{A_{\spl}/S,\bar0}
\cong
\mathcal Q^{(2)}_{F/S}\oplus \mathcal Q^{(3)}_{F/S}
\cong
\mathcal O_Y^{\oplus6}.
\]
In particular the controller is abelian.

\smallskip

For the trivial elliptic family, relative cohomology gives
\[
R^0\pi_*\mathcal O_Y\cong\mathcal O_S,
\qquad
R^1\pi_*\mathcal O_Y\cong\mathcal O_S.
\]
Hence
\[
R^0\pi_*\mathcal T^{\langle2\rangle}_{A_{\spl}/S,\bar0}
\cong
\mathcal O_S^{\oplus6},
\qquad
R^1\pi_*\mathcal T^{\langle2\rangle}_{A_{\spl}/S,\bar0}
\cong
\mathcal O_S^{\oplus6}.
\]
There are no higher obstruction sheaves and no nonzero brackets.  Therefore the intrinsic
Green--Onishchik formal moduli problem
\(
\widehat\Tw_{F/S}
\)
is a linear abelian formal stack over \(S\).  This statement should be read at the
stack-theoretic level: the degree-zero part records infinitesimal automorphisms of the
split model, while the degree-one part records deformation directions.

After rigidifying the degree-zero automorphism directions, the corresponding classical
local moduli space is the vector bundle
\[
\mathbb V_S\!
\left(
R^1\pi_*\mathcal T^{\langle2\rangle}_{A_{\spl}/S,\bar0}
\right)
\cong
\mathbb V_S(\mathcal O_S^{\oplus6}).
\]
Thus the vector bundle above is not the full Green--Onishchik stack; it is the
rigidified local deformation space obtained from the abelian formal controller.  In this
rank-three situation, its points do correspond to actual fixed-retract supermanifold
structures locally on the chosen base, because the Maurer--Cartan equation has no
higher terms.

The decomposition
\[
\mathcal T^{\langle2\rangle}_{A_{\spl}/S,\bar0}
\cong
\mathcal Q^{(2)}_{F/S}\oplus\mathcal Q^{(3)}_{F/S}
\]
correspondingly decomposes the local deformation coordinates into relative projectedness
coordinates and relative residual coordinates:
\[
\mathcal O_S^{\oplus6}
\cong
\mathcal O_S^{\oplus3}\oplus\mathcal O_S^{\oplus3}.
\]
This example is the relative analogue of the elliptic curve example above.  It shows that, in
low odd rank, the Green--Onishchik formal moduli problem can be completely linear.  The
algebraized local Kuranishi model of Section~\ref{subsec:local-derived-kuranishi-algebraization}
is then simply a vector bundle, with zero Kuranishi map.

\subsection{Odd rank four on an abelian surface and the first nonlinear relation}
\label{subsec:example-abelian-surface-rank-four}

We now exhibit the first genuinely nonlinear phenomenon.  Let
\[
A\defeq E_1\times E_2
\]
be the product of two smooth elliptic curves, and set
\[
F\defeq\mathcal O_A^{\oplus4},
\qquad
A_{\spl}:=\wedge^\bullet F.
\]
Since \(A\) is an abelian surface,
\[
\mathcal T_A\cong\mathcal O_A^{\oplus2},
\qquad
H^1(A,\mathcal O_A)\cong\CC^2,
\qquad
H^2(A,\mathcal O_A)\cong\CC.
\]
Moreover, by K\"unneth,
\[
H^2(A,\mathcal O_A)
\cong
H^1(E_1,\mathcal O_{E_1})\otimes
H^1(E_2,\mathcal O_{E_2}).
\]
The exterior powers of \(F\) are trivial:
\[
\wedge^2F\cong\mathcal O_A^{\oplus6},
\qquad
\wedge^3F\cong\mathcal O_A^{\oplus4},
\qquad
\wedge^4F\cong\mathcal O_A.
\]
Thus
\[
\mathcal Q^{(2)}
=
\mathcal T_A\otimes\wedge^2F
\cong
\mathcal O_A^{\oplus12},
\]
\[
\mathcal Q^{(3)}
=
\mathcal H\!om_{\mathcal O_A}(F,\wedge^3F)
\cong
\mathcal O_A^{\oplus16},
\]
and, because \(\wedge^5F=0\),
\[
\mathcal T_{A_{\spl},\bar0}^{\langle4\rangle}
\cong
\mathcal T_A\otimes\wedge^4F
\cong
\mathcal O_A^{\oplus2}.
\]
This is the first odd rank in which the weight-four target of the quadratic Kuranishi term can
be nonzero.

\smallskip

Choose a global frame
\(
e_1,e_2,e_3,e_4
\)
of \(F\), with odd coordinates
\(
\theta_1,\theta_2,\theta_3,\theta_4
\).
Let
\(
v_1,v_2\in H^0(A,\mathcal T_A)
\)
be the translation-invariant vector fields induced by the two factors.  Put
\[
\Theta\defeq \theta_1\theta_2\theta_3\theta_4.
\]
Consider the two weight-two even derivations
\[
D\defeq \theta_3\theta_4 v_1,
\qquad
K\defeq \theta_1\theta_2\theta_3\partial_{\theta_3}.
\]
The derivation \(D\) is quotient-type, while \(K\) is kernel-type.  A direct calculation gives
\[
[D,D]=0,
\qquad
[K,K]=0,
\qquad
[D,K]= -\Theta v_1.
\]
Indeed, \(D\) differentiates only even functions on the first elliptic factor, while the coefficient
of \(K\) is constant, so \(D\circ K=0\).  On the other hand,
\[
K(\theta_3\theta_4)
=
\theta_1\theta_2\theta_3\partial_{\theta_3}(\theta_3\theta_4)
=
\Theta,
\]
and therefore \(K\circ D=\Theta v_1\).

\smallskip

Let \(\mathfrak U^{(1)}\) and \(\mathfrak U^{(2)}\) be Stein covers of \(E_1\) and \(E_2\), and let
\(
\mathfrak U
\)
be the product cover of \(A\).  Choose \v{C}ech cocycles
\[
\eta\in\check Z^1(\mathfrak U^{(1)},\mathcal O_{E_1}),
\qquad
\xi\in\check Z^1(\mathfrak U^{(2)},\mathcal O_{E_2})
\]
representing nonzero classes in
\(
H^1(E_1,\mathcal O_{E_1}),
\) and \(
H^1(E_2,\mathcal O_{E_2}).
\)
We denote their pullbacks to \(A\) by the same symbols.  Then
\[
[\eta]\smile[\xi]\neq0
\quad\text{in}\quad
H^2(A,\mathcal O_A).
\]

Define two classes
\(
x\defeq [\eta D]
\) and \(
y\defeq [\xi K]
\)
in
\(
H^1\bigl(A,\mathcal T_{A_{\spl},\bar0}^{\langle2\rangle}\bigr),
\)
and consider the two-dimensional subspace
\(
P\defeq\langle x,y\rangle.
\)
A point of \(P\) has the form
\(
ux+vy,
\) for \(
u,v\in\CC
\)
and it can be represented by the \v{C}ech cocycle
\[
a_1(u,v)\defeq u\eta D+v\xi K.
\]
The quadratic Kuranishi term is
\[
\frac12[a_1(u,v),a_1(u,v)]
=
-uv\,(\eta\smile\xi)\,\Theta v_1,
\]
up to the harmless sign determined by the chosen \v{C}ech cup convention.  Thus the
cohomology class of the quadratic term is
\[
\kappa_2(ux+vy)
=
-uv\,[\eta]\smile[\xi]\cdot \Theta v_1
\in
H^2\bigl(A,\mathcal T_{A_{\spl},\bar0}^{\langle4\rangle}\bigr).
\]
Set
\[
\Omega
\defeq
[\eta]\smile[\xi]\cdot\Theta v_1.
\]
Since \([\eta]\smile[\xi]\neq0\), the class \(\Omega\) is nonzero.  After projecting the target to
the line \(\CC\Omega\), the Kuranishi map on the slice \(P\) has quadratic term
\[
\kappa_P:\mathbb A^2_{u,v}\longrightarrow\mathbb A^1_w,
\qquad
w=-uv.
\]
Consequently, after choosing the corresponding Kuranishi slice and rigidifying the
infinitesimal automorphism directions, the local derived Kuranishi model on this slice is the
derived zero locus
\[
\mathfrak K_P
=
\mathbb A^2_{u,v}
\times^h_{\mathbb A^1_w}
\operatorname{Spec}\CC,
\]
where \(\operatorname{Spec}\CC\to\mathbb A^1_w\) is the origin.  Its classical truncation is
\[
t_0(\mathfrak K_P)
=
\{uv=0\}\subset\mathbb A^2.
\]
Thus the two coordinate axes are unobstructed at the quadratic level, whereas mixed
directions with \(uv\neq0\) are obstructed.

\smallskip

This example has a direct Green-theoretic interpretation.  The class \(x=[\eta D]\) is a
quotient-type direction, and
\[
H^1(\rho_2)(x)=[\eta]\otimes D\neq0.
\]
The class \(y=[\xi K]\) is kernel-type, and
\(
H^1(\rho_2)(y)=0.
\)
Separately, each pure direction has zero quadratic bracket:
\[
\kappa_2(x)=0,
\qquad
\kappa_2(y)=0.
\]
The obstruction is genuinely mixed.  It comes from the interaction between a quotient-type
Green coordinate and a kernel-type Green coordinate in the same weight-two tangent class.

\smallskip

Equivalently, suppose that the mixed first-order class
\(
ux+vy
\)
with \(uv\neq0\) extended to a full Maurer--Cartan element.  At the next stage one would
need a weight-four \(1\)-cochain
\(
a_2\in\check C^1(\mathfrak U,
\mathcal T_{A_{\spl},\bar0}^{\langle4\rangle})
\)
such that
\[
\check\delta a_2+\frac12[a_1(u,v),a_1(u,v)]=0.
\]
But the cohomology class of the second term is
\(
-uv\,\Omega
\), which is nonzero.  Therefore no such \(a_2\) exists.

\begin{remark}[Actual objects versus formal directions in the rank-four example]
This example should be read differently from the two odd-rank three examples above.
The pure directions \(x=[\eta D]\) and \(y=[\xi K]\) are represented by full
Maurer--Cartan elements inside the chosen slice, because
\[
[D,D]=0,
\qquad
[K,K]=0.
\]
Thus the corresponding coordinate axes are unobstructed in this two-dimensional
Kuranishi slice. By contrast, a mixed vector
\[
ux+vy
\qquad (uv\neq0)
\]
is initially only a first-order, or Kuranishi, direction.  It represents an actual formal
supermanifold structure only if it can be completed by a weight-four correction
\(
a_2\in
\check C^1\bigl(
\mathfrak U,
\mathcal T_{A_{\spl},\bar0}^{\langle4\rangle}
\bigr)
\)
satisfying the Maurer--Cartan equation
\[
\check\delta a_2+\frac12[a_1(u,v),a_1(u,v)]=0.
\]
The calculation above shows that the obstruction class is
\[
-uv\,\Omega
\in
H^2\bigl(A,\mathcal T_{A_{\spl},\bar0}^{\langle4\rangle}\bigr),
\]
which is nonzero for \(uv\neq0\).  Hence the mixed directions do not define actual
supermanifold structures.  The equation
\[
uv=0
\]
is therefore not a pre-existing family of supermanifolds; it is the classical truncation
of the local derived Kuranishi model, and its two components are precisely the
unobstructed branches in this chosen slice.
\end{remark}


\end{document}